%% file: 0_Main.tex
\documentclass[abstracton, paper=a4, fontsize=11pt,DIV=12,bibliography=totoc,final]{scrartcl}

\pdfoutput=1

\input{header.tex}

\begin{document}

%%%%%%%%%%%%%%%%%%%% Arxiv %%%%%%%%%%%%%%%%

\title{\LARGE Sedimentation of Particles with Very Small Inertia I:  Convergence to the Transport-Stokes Equation}
\author[1]{Richard M. H\"ofer\thanks{richard.hoefer@ur.de}}
\author[2]{Richard Schubert\thanks{schubert@iam.uni-bonn.de}}
\affil[1]{Faculty of Mathematics, University of Regensburg, Germany}
\affil[2]{Institute for Applied Mathematics, University of Bonn, Germany}

%%%%%%%%%%%%%%%%%%%%%%%%%%%%%%%%%%%%

\maketitle

\begin{abstract}
    We consider the sedimentation of $N$ spherical particles with identical radii $R$ in a  Stokes flow in $\R^3$. The particles satisfy a no-slip boundary condition and are subject to constant gravity. The dynamics of the particles is modeled by Newton's law but with very small particle inertia as $N$ tends to infinity and $R$ to $0$.
   In a mean-field scaling, we show that the particle evolution is well approximated by the transport-Stokes system which has been derived previously as the mean-field limit of inertialess particles. In particular this justifies to neglect the particle inertia in the microscopic system, which is  a typical modelling assumption in this and related contexts. 

   The proof is based on a relative energy argument that exploits the coercivity of the particle forces with respect to the particle velocities  in a Stokes flow. We combine this with an adaptation of Hauray's method for mean-field limits to $2$-Wasserstein distances.
   Moreover, in order to control the minimal distance between particles, we prove a representation of the particle forces. This representation makes the heuristic \enquote{Stokes law} rigorous that the force on each particle is proportional to the difference of the velocity of the individual particle and the mean-field fluid velocity generated by the other particles.
\end{abstract}

\tableofcontents

\input{1.Introduction.tex}

\input{2.Meanfield}

\input{3.ProofMain}

\input{4.EstimatesR}

\input{5.MinDist}

\section*{Acknowledgements}

The authors are grateful for the opportunity of an intensive two-week ''Research in pairs`` stay at ''Mathematisches Forschungszentrum Oberwolfach``  during which the main part of this article was conceived. R.S. thanks David G\'erard-Varet and the Institut de Mathématiques de Jussieu-Paris Rive Gauche for the hospitality during the stay in Paris. R.H. thanks Barbara Niethammer and the Hausdorff Center for Mathematics for the hospitality during the stays in Bonn. We thank the anonymous referees for valuable suggestions which helped to improve the final version of the article.

R.H. has been supported  by the German National Academy of Science Leopoldina, grant LPDS 2020-10.

R.S. has been supported by the Deutsche Forschungsgemeinschaft (DFG, German Research Foundation) through the research training group ''Energy,
Entropy, and Dissipative Dynamics (EDDy)`` (Project-ID 320021702 /GRK2326) and the collaborative research
centre ‘The mathematics of emergent effects’ (CRC 1060, Project-ID 211504053).

\appendix

\input{Appendix}

 \begin{refcontext}[sorting=nyt]
\printbibliography
 \end{refcontext}
\end{document}

%% file: header.tex
\usepackage[utf8]{inputenc}
\usepackage[T1]{fontenc}
\usepackage{lmodern}
\usepackage[english]{babel}
\usepackage{csquotes}
\usepackage{amsthm, amssymb, amsmath, amsfonts, mathrsfs, dsfont, esint,textcomp}
\usepackage[colorlinks=true, pdfstartview=FitV, linkcolor=blue, citecolor=blue, urlcolor=blue,pagebackref=false]{hyperref}
\usepackage{mathtools}
\usepackage[shortlabels]{enumitem}

\usepackage[backend=biber,giveninits,maxnames=4,maxalphanames=5, style=alphabetic,isbn=false, date=year, doi=false, eprint= true, url= false, sorting=ynt, sortcites = true]{biblatex}

 \usepackage{authblk}

 \setkomafont{date}{\large}

 \addtokomafont{disposition}{\rmfamily}

\usepackage{microtype}

\usepackage[notcite,notref,color]{showkeys}
\definecolor{labelkey}{gray}{.8}
\definecolor{refkey}{gray}{.8}

\definecolor{darkblue}{rgb}{0,0,0.7} 
\definecolor{darkred}{rgb}{0.9,0.1,0.1}
\definecolor{darkgreen}{rgb}{0,0.5,0}

\newcommand{\rs}[1]{#1}

\newcommand{\rh}[1]{#1}

\newcommand{\ad}[1]{#1}

\renewenvironment{proof}[1][\smallskip\noindent\proofname]{{\smallskip\noindent\bfseries #1. }}{\qed \medskip}

\newtheorem{thm}{Theorem}[section]
\newtheorem{theorem}[thm]{Theorem}
\newtheorem{prop}[thm]{Proposition}
\newtheorem{lem}[thm]{Lemma}

\theoremstyle{definition}

\newtheorem{rem}[thm]{Remark}

\newcommand\norm[1]{\left\lVert#1\right\rVert}

\newcommand\bra[1]{\left({#1}\right)}

\newcommand\abs[1]{\left\lvert#1\right\rvert}

\renewcommand{\le}{\leqslant}
\renewcommand{\ge}{\geqslant}
\renewcommand{\leq}{\leqslant}
\renewcommand{\geq}{\geqslant}
\newcommand{\ls}{\lesssim}
\newcommand{\gs}{\gtrsim}
\renewcommand{\subset}{\subseteq}

\newcommand{\cH}{\mathcal{H}}
\newcommand{\mH}{\mathcal{H}}
\newcommand{\J}{\mathcal{J}}

\newcommand{\mP}{\mathcal{P}}

\newcommand{\N}{\mathbb{N}}

\newcommand{\1}{\mathbf{1}}
\newcommand{\R}{\mathbb{R}}

\newcommand{\mfk}{\mathfrak}

\newcommand{\dd}{\, \mathrm{d}}
\newcommand{\wto}{\rightharpoonup}

\newcommand{\loc}{\mathrm{loc}}

\DeclareMathOperator{\St}{\mathrm{St}}

\newcommand{\W}{\mathcal{W}}
\newcommand{\dmin}{d_{\mathrm{min}}}

\newcommand{\mR}{\mathcal R}

\newcommand{\m}{\mathcal}

\renewcommand{\varrho}{\rho}

\DeclareMathOperator{\curl}{curl}

\DeclareMathOperator{\Id}{Id}

\DeclareMathOperator{\dv}{div}

\DeclareMathOperator*{\supp}{supp}

\DeclareMathOperator{\dist}{dist}

\numberwithin{equation}{section}

\mathtoolsset{showonlyrefs}

\bibliography{bib.bib}

%% file: 1.Introduction.tex
\section{Introduction}

Particulate flows are ubiquitous in nature and engineering. They appear for instance in the atmosphere and in biological fluids. These \emph{suspensions} can be described on different scales. The most accurate, but from a practical perspective also most challenging, perspective is the microscopic one, where each particle in the fluid is resolved in the mathematical model and their evolution is governed by Newton's laws. Looking for an \emph{effective} behaviour of the mixture, the density of the particles is described in a continuous manner. Using a phase space density for the mesoscopic (kinetic) perspective and using a spatial density on the macroscopic (hydrodynamic) scale. In the spirit of Hilbert's sixth problem, this and the accompanying article \cite{HoferSchubert23b} connect the microscopic model to both of the larger scales in the important case of the sedimentation of particles with inertia. This is particularly relevant for aerosols.

For a number $N$ of particles, the state of the particles at any time is determined by the particle positions $X_i\in \R^3$, velocities $V_i\in \R^3$ and angular velocities $\Omega_i\in \R^3$ for $1\le i\le N$. We model the particles by identical spheres with radius $R$ which implicitly depends on $N$. We denote the space occupied by particle $i$ as $B_i = \overline{B_R(X_i)}$. The sedimentation of the particles in a Stokes flow is governed by Newton's law which takes the dimensionless form
\begin{align}
    \dot X_i &= V_i,\label{eq:Xi}\\
    \dot{V}_i &= \lambda_N \left(g + \frac{\gamma_N}{R} \int_{\partial B_i} \sigma[u_N] n \dd  \mathcal{H}^2 \right), \label{eq:Vi}\\
    \dot \Omega_i &= \rs{\frac{5}{2}} \lambda_N \frac{\gamma_N}{R^3} \int_{\partial B_i} (x - X_i)\times (\sigma[u_N] n)\dd  \mathcal{H}^2,\label{eq:Oi}
\end{align}
where $g \in S^2$ is the (normalized) gravitational acceleration, $\sigma[u_N] = 2 e u_N - p_N \Id$ is the fluid stress, $e  u_N = \frac 12(\nabla u_N + (\nabla  u_N)^T) $ denotes the symmetric gradient of $u_N$, the outer normal of $\partial B_i$ is denoted by $n$, and $ \mathcal{H}^2$ is the two-dimensional Hausdorff measure. 
The fluid flow $u_N:\R^3\to \R^3$ itself with associated pressure $p_N:\R^3\to \R$ is herein given as the unique solution in $\dot H^1(\R^3)$ to the Stokes problem
\begin{equation}
	\label{eq:fluid.micro.rot}
\left\{\begin{array}{rl}
		- \Delta u_N + \nabla p_N = 0, ~ \dv u_N &=0 \quad \text{in} ~ \R^3 \setminus \bigcup_i B_i, \\
		  u_N(x) &= V_i + \Omega_i \times (x - X_i) \quad \text{in} ~ B_i.
\end{array}\right.
\end{equation}

The parameter $\gamma_N = NR$ accounts for the interaction strength and the Stokes number $\St_N = \frac 1 {\gamma_N \lambda_N}$ accounts for the strength of  inertial forces (if $\gamma_N$ is at least of order $1$). Note that in view of the Stokes drag formula for the drag force of a single particle, $\int_{\partial B} \sigma[u_N] n \dd  \mathcal{H}^2=-6\pi RV$ (and the corresponding formula for the torque), the surface integrals in \eqref{eq:Vi} and \eqref{eq:Oi} are of order $R$ and $R^3$, respectively. 
For details on the non-dimensionalization we refer to \cite{Hofer18MeanField}.

In the present paper, we consider the effective behavior of the evolution of inertial particles  \eqref{eq:Xi}--\eqref{eq:fluid.micro.rot} 
for very small particle inertia.
Loosely stated, we show that for $N \to \infty$ with  $\gamma_N \to \gamma_\ast \in (0,\infty)$ and $\lambda_N \to \infty$, the solution to  \eqref{eq:Xi}--\eqref{eq:fluid.micro.rot} is
close to the solution of the transport-Stokes equation
 \begin{align}  \label{eq:transport-Stokes}
\left\{\begin{array}{rl}
\partial_t \rho_\ast +(u_\ast+ \bra{6 \pi \gamma_\ast}^{-1} g) \cdot \nabla \rho_\ast&=0,\\
-\Delta u_\ast+\nabla p_\ast= \rho_\ast g, ~ \dv u_\ast&=0,\\
\rho_\ast(0)&=\rho^0.
\end{array}\right.
\end{align} 
The precise statement of the result is given in Theorem~\ref{th:diagonal}.

The transport-Stokes equation \eqref{eq:transport-Stokes} has 
been established before in \cite{Hofer18MeanField, Mecherbet19,Hofer&Schubert} as mean-field limit of the inertialess microscopic system 
\begin{equation}
\left\{\begin{array}{rl}
	\label{eq:inertialess.rot}
\dot {\mfk X}_i = \mfk V_i &=  \mfk u_N(\mfk X_i),\\
		- \Delta \mfk u_N + \nabla \mfk p_N = 0, ~ \dv  \mfk u_N &=0 \quad \text{in} ~ \R^3 \setminus \bigcup_i \mfk B_i, \\
		  \mfk u_N(x) &= \mfk V_i+\omega_i\times(x-\mfk X_i) \quad \text{in} ~ \mfk B_i, \\
		 -\int_{\partial \mfk B_i} \sigma[ \mfk u_N] n \dd  \mathcal{H}^2  &= \frac g N, \quad \text{for all} ~ 1 \leq i \leq N,\\ 
		 \int_{\partial \mfk B_i} (x - \mfk X_i)\times (\sigma[\mfk u_N] n) \dd  \mathcal{H}^2 &= 0, \quad \text{for all} ~ 1 \leq i \leq N,
\end{array}\right.
\end{equation}
where $\mfk B_i = \overline{B_R(\mfk X_i)}$. Here, in contrast to \eqref{eq:Xi}--\eqref{eq:fluid.micro.rot}, $\mfk V_i$ and $\omega_i$ are implicitly determined by the positions $(\mfk X_i)_i$ of the particles.
The inertialess system \eqref{eq:inertialess.rot} can formally be derived from \eqref{eq:Xi}--\eqref{eq:fluid.micro.rot} by setting $\lambda_N=\infty$. 

Modelling  particles as being inertialess is always an idealization. Therefore, it is an important problem to rigorously justify inertialess models starting from models with very small inertia. In particular, in combination with \cite{Hofer&Schubert}, the result in this paper shows that the microscopic system with inertia is close to the corresponding microscopic system \emph{without} inertia \eqref{eq:inertialess.rot}.

\subsection{Previous results}

Models of inertial and inertialess particles suspended in viscous incompressible fluids have been extensively studied in various contexts to investigate the effective behavior of suspensions. We focus our discussion on homogenization and mean-field results related to the microscopic systems \eqref{eq:Xi}--\eqref{eq:fluid.micro.rot} and \eqref{eq:inertialess.rot}, on related problems with vanishing inertia and on results regarding the transport-Stokes equation \eqref{eq:transport-Stokes}. For a discussion on results on the Vlasov-Stokes equation \eqref{eq:Vlasov.Stokes} and related systems, we refer to our accompanying paper \cite{HoferSchubert23b}.

\subsubsection*{Homogenization results}

Considering the limit $N \to \infty$ with $R \to 0$ for the static problem \eqref{eq:fluid.micro.rot}, when the particle velocities $V_i$ and $\Omega_i$ are given, is a classical homogenization problem in perforated domains. It is by now mathematically well-understood that the drag force at the particles gives rise to a collective Brinkman force, see e.g. \cite{Allaire90a, DesvillettesGolseRicci08, Feireisl2016, HillairetMoussaSueur17, GiuntiHofer18, CarrapatosoHillairet18, HoeferJansen20}. More precisely, if the empirical phase space density 
\begin{align}
    f_N=\frac 1N\sum_i \delta_{X_i}\otimes\delta_{V_i}
\end{align}
converges  to some $f = f(x,v)$ as $N \to \infty$, $R \to 0$ with $\gamma_N = NR \to \gamma_\ast \in [0,\infty)$, then (under some additional mild assumption on the particle configuration) $u_N \wto u_\ast$ in $\dot H^1(\R^3)$, where $u_\ast$ is the unique weak solution to the Brinkman equation
\begin{align}
    -\Delta u_\ast + \nabla p_\ast =  6 \pi\gamma_\ast \int (v- u_\ast) f \dd v, \quad \dv u_\ast &=0. 
\end{align}

On the other hand, if the analogous static problem derived from \eqref{eq:inertialess.rot} is studied for inertialess particles in the absence of gravity (respectively for neutrally buoyant particles)
but with some external force $h$, the presence of the particles leads to an increase of the effective viscosity of the fluid, proportional to the particle volume fraction $\phi = NR^3$. This means, the homogenized equation reads, to first order in $\phi$,
\begin{align}
    -\dv((2 + 5 \phi \rho) e \mfk u_\ast) + \nabla \mfk p_\ast =  h, \quad \dv \mfk u_\ast &=0,
\end{align}
where $\rho = \rho(x)$ is the limit of the empirical particle density
\begin{align}
    \rho_N=\frac 1N\sum_i \delta_{\mfk X_i}.
\end{align}This has been made rigorous in \cite{HainesMazzucato12, NiethammerSchubert19, HillairetWu19, 
Gerard-VaretHoefer21, DuerinckxGloria21, Duerinckx22}.

\subsubsection*{Mean-field and related results}

As mentioned before, considering the full dynamic problem \eqref{eq:inertialess.rot}, the transport-Stokes equation \eqref{eq:transport-Stokes} has been derived in \cite{Hofer18MeanField, Mecherbet19}. The result in \cite{Hofer18MeanField} covers the cases $\gamma_\ast \in (0,\infty]$, whereas \cite{Mecherbet19} is able to allow for a larger class of initial configurations for $\gamma_\ast$ sufficiently small. Perturbatively, the effective increase of the viscosity has been shown to persist in the dynamical setting in \cite{Hofer&Schubert}.

In \cite{Hofer18MeanField, Mecherbet19, Hofer&Schubert}, suspensions of spherical particles are studied. 
The recent paper \cite{Duerinckx23} considers the sedimentation of non-spherical and possibly active particles in a perturbative setting where both gravity and activeness have an effect of order $\phi$ on the fluid. Notably, the activeness as well as the particle orientation dependent viscosity lead to non-Newtownian behavior.
For related results of non-Newtonian effects (also due to thermal noise) we refer to \cite{Girodroux-Lavigne22, HoeferLeocataMecherbet22, HoferMecherbetSchubert22}.

Another related question concerns the sedimentation of stratified suspensions in \emph{bounded domains}\footnote{For mathematical convenience, also suspensions on the torus or of infinitely many particles in the whole space are considered}, i.e. suspensions that are uniformly distributed in directions orthogonal to the direction of gravity. In this case, the particle interaction does not increase but \emph{hinder} the settling speed which remains of the same order as for an isolated particle. The variance of the settling velocity diverges with the container size, though. We refer to \cite{Gloria21,DuerinckxGloria22,HillairetHoefer23} for mathematical results (in a static setting).

\medskip

Coming back to the microscopic system with inertia, \eqref{eq:Xi}--\eqref{eq:fluid.micro.rot},
the formal mean-field limit for $\lambda_N \to \lambda_\ast \in (0,\infty)$ and $\gamma_N \to \gamma_\ast \in (0,\infty)$ is the Vlasov-Stokes system
\begin{align} \label{eq:Vlasov.Stokes}
\left\{\begin{array}{rl}
    \partial_t f + v \cdot \nabla_x f + \lambda_\ast {\dv}_v((g+6 \pi \gamma_\ast(u_\ast- v))f) &= 0, \\
    -\Delta u_\ast + \nabla p_\ast = 6 \pi \gamma_\ast \int (v-u_\ast) f \dd v, ~ \dv u_\ast &=0, \\
    f(0) &= f^0.
    \end{array}\right.
\end{align}
Currently, the rigorous derivation of the Vlasov-Stokes system starting from \eqref{eq:Xi}--\eqref{eq:fluid.micro.rot} seems out of reach owing to the singular and implicit nature of the particle interaction through the fluid.
We refer to \cite{BernardDesvillettesGolseRicci17, BernardDesvillettesGolseRicci18} for a formal derivation from the Boltzmann equation and to  \cite{FlandoliLeocataRicci19,FlandoliLeocataRicci21} for a derivation of the Vlasov-Navier-Stokes-(Fokker-Planck) equation from an intermediate fluid model where the particle-fluid interaction is modeled via an explicit kernel with cutoff singularity  instead of boundary conditions between the particles and the fluid.

In the accompanying paper \cite{HoferSchubert23b}, we provide a perturbative derivation of the Vlasov-Stokes system \eqref{eq:Vlasov.Stokes} for very small particle inertia. In the same setting considered here, we show that the Vlasov-Stokes system \eqref{eq:Vlasov.Stokes} with parameter $\lambda_N$ is a better approximation for the microscopic system than the transport Stokes system \eqref{eq:transport-Stokes}.

\subsubsection*{Related results on problems with vanishing particle inertia}

Limits of vanishing inertia (also called massless limits) have been considered in various situations.
On the level of the effective equation, passing from Vlasov-Stokes \eqref{eq:Vlasov.Stokes} to transport-Stokes \eqref{eq:transport-Stokes}, has been made rigorous in \cite{Hofer18InertialessLimit}. For related results, we refer to \cite{GoudonJabinVasseur04a, GoudonJabinVasseur04b, Jab00,Gou01,HanKwanMichel21, Ertzbischoff22}.

Starting from microscopic models, limits of vanishing particle inertia for  particles immersed in a fluid have been studied for example in \cite{GonzalesGrafMaddocks04, GlassSueur19, HoeferPrangeSueur22, FeireislRoyZarnescu22, BravinNecasova22} (see also the references therein).
In \cite{GonzalesGrafMaddocks04}, the case of a single rigid body of fixed shape in a Stokes flow is studied and 
\cite{HoeferPrangeSueur22} deals with several particles shrinking to curves in a Stokes flow. On the other hand, \cite{FeireislRoyZarnescu22, GlassSueur19, BravinNecasova22}  are concerned with  one or several particles shrinking to points in an incompressible viscous and inviscid  fluid, respectively.

A combined small-inertia/mean-field limit has recently been studied in \cite{CarilloChoi21} where the particles interact  via binary interaction and alignment forces which are assumed to be \ad{either Lipschitz continuous or given through repulsive Coulomb or Riesz potentials}.
Similar models with noise have been treated in \cite{WangLvWei22}. 
In contrast, the implicit and singular nature of the interaction in \eqref{eq:Xi}--\eqref{eq:fluid.micro.rot} poses the main difficulties for our analysis.

We also mention the related works \cite{HanKwanIacobelli21, Rosenzweig23}, where the Euler equation is obtained in a supercritical mean-field limit of Coulomb interaction that can be viewed as a combined mean-field and quasineutral limit. This has the common feature with our result that in the mean-field limit the empirical particle density concentrates  in a Dirac distribution in the velocity space. However, this concentration is driven by strong interaction between the particles rather than by the high friction in our case.

\subsubsection*{Results on the transport-Stokes equation}

The well-posedness of the transport-Stokes equation in $L^1(\R^3)\cap L^\infty(\R^3)$ for initial data with bounded first moment has been shown in \cite{Mecherbet21}, and subsequently in \cite{Hofer&Schubert} without the moment restriction. Data in $L^1(\R^3)\cap L^q(\R^3)$ for $q\ge 3$ still provide enough integrability for the velocity field to be (log-)Lipschitz. Consequently well-posedness was extended to data $\rho^0 \in L^1(\R^3)\cap L^q(\R^3)$ for $q \geq 3$ in \cite{MecherbetSueur22}, where also analyticity of trajectories and controllability is provided. Finally, the recent \cite{Inversi23} proves existence of Lagrangian solutions for even lower regularity ($L^1(\R^3)$) of the data, where uniqueness is lost. Well-posedness for the transport-Stokes system has been extended to (partially) bounded domains \cite{Leblond22}, two dimensions \cite{Grayer22}, and for fractional Stokes operators \cite{Cobb23}. Formulated for densities that are essentially the sum of two scaled characteristic functions (patch problem), the transport-Stokes dynamics become a free-boundary problem which has been analyzed with regard to well-posedness and long-time asymptotics in \cite{AMY00, Mecherbet21,GancedoGranero-BelinchonSalguero22}. In \cite{DalibardGuillodLeblond23}, long-time asymptotics for settling of densities close to a vertically linear profile which allow for a convenient linearization of the problem, are proved.

\subsection{Statement of the main results}

For given  initial data $X_i^0, V_i^0$, $1 \leq i \leq N$ (implicitly depending on $N$), that satisfy the no-touch condition
\begin{align}  \label{nonoverlapping}
       |X_i^0 - X_j^0| > 2R,
\end{align}
standard arguments provide a unique solution $u$, $X_i, V_i$ of \eqref{eq:acceleration},\eqref{eq:fluid.micro} until the first time two particles collide. (In fact, it was recently shown in \cite{HillairetSabbagh22} that such collisions do not occur in finite time and there is a global solution.)

With the solution we associate the empirical density
\begin{align}
    f_N(t) = \frac 1 N \sum_i \delta_{X_i(t)} \otimes \delta_{V_i(t)} 
\end{align}
and its first marginal
\begin{align}
    \rho_N(t) = \frac 1 N \sum_i \delta_{X_i(t)}.
\end{align}
Moreover, we introduce the minimal distance between the particles (which implicitly depends on $N$) 
\begin{align}
    \dmin(t) = \min_{i \neq j} |X_i(t) - X_j(t)|.
\end{align}

We make the  following assumptions on the scaling of $\gamma_N$, $\lambda_N$ and a given sequence of initial configuration of particles.
\begin{align} 
    % \label{ass:lambda}\tag{H1}
    % &\lim_{N \to \infty} 1/\lambda_N = 0, \\[5pt]
    \label{ass:gamma} \tag{H1}
    &\exists C_\gamma: ~ \forall N \in \N: ~ \frac 1 {C_\gamma} \leq  \gamma_N:=  NR \leq C_\gamma, \\[5pt]
	&\exists q> 3, \rho^0 \in \mathcal P(\R^3) \cap L^q(\R^3): \\\label{ass:Wasserstein.00}  \tag{H2}
    &\qquad \quad \lim_{N \to \infty} \bra{\W_2(\rho_N(0), \rho^0)  + \frac 1 {\lambda_N}}^{\frac 1 {3 + 2q'}}\left(1 + \frac 1 {N^{1/3}\dmin(0)}\right) = 0, \quad  \\[10pt]
    \label{ass:V.Lipschitz} \tag{H3}
    &\forall N \in \N: ~ \forall\, 1 \leq i,j \leq N: ~ |V_i^0 - V_j^0| \leq 3 \pi \gamma_N \lambda_N |X_i^0 - X_j^0|,   \\[7pt] 
    \label{ass:V_infty}\tag{H4}
    &\exists C_V > 0: ~ \forall N \in \N: ~ \frac 1 N  |V^0|_2^2 \rs{+ \frac {R^2} N |\Omega^0|_2^2}  + \frac{1}{\lambda_N}|V^0|_\infty \leq C_V.
\end{align} 
In assumption \eqref{ass:Wasserstein.00}, $\mathcal P(\R^3)$ denotes the space of probability measures on $\R^3$ and $\W_2$ denotes the $2$-Wasserstein metric (for a definition, also for general $p$-Wasserstein distances and their properties we refer to \cite{Santambrogio15}). The exponent $q'$ is the H\"older dual of $q$ with $\tfrac 1q+\tfrac 1{q'}=1$. Moreover, we denote by $|\cdot|_2$ the standard Euclidean norm and by $|\cdot|_\infty$ the maximum norm for vectors in $(\R^3)^N$ (see \eqref{euclidean} and \eqref{maximum} below).
For a discussion of the assumptions we refer to Remark~\ref{rem:assumptions} below.

In order to formulate the main result, we finally introduce \rs{$(\tilde V,\tilde \Omega)\in (\R^3)^N \times (\R^3)^N$}, given by
\begin{align}\label{eq:V_tilde.rot}
\left\{\begin{array}{rl}
		- \Delta v_N + \nabla p = 0, ~ \dv  v_N &=0 \quad \text{in} ~ \R^3 \setminus \bigcup_i B_i, \\
		  v_N(x) &= \tilde V_i \rs{+\tilde \Omega \times (x- X_i)} \quad \text{in} ~  B_i, \\
		 -\int_{\partial  B_i} \sigma[ v_N] n \dd  \mathcal{H}^2  &= \frac g N, \quad \text{for all} ~ 1 \leq i \leq N, \\
\rs{-\int_{\partial  B_i} (x-X_i) \times \sigma[ v_N] n \dd  \mathcal{H}^2  }&\rs{= 0, \quad \text{for all} ~ 1 \leq i \leq N,}
\end{array}\right.
\end{align}
which is the instantaneous inertialess velocities of the particles associated to their momentary position.

\begin{theorem} \label{th:diagonal}
For $N \in \N$ assume that $\lambda_N$, $\gamma_N$, $X_i^0, V_i^0$  satisfy \eqref{ass:gamma}--\eqref{ass:V_infty}.
Let $(\rho_\ast,u_\ast) \in L^\infty((0,T);L^q(\R^3))\times L^\infty((0,T);\dot H^1(\R^3))$ be the unique weak solution of \eqref{eq:transport-Stokes} associated to the initial value $\rho^0$ and with $\gamma_\ast$ replaced by $\gamma_N$. Then, there exists $C > 0$ depending only on the constants from \eqref{ass:gamma}--\eqref{ass:V_infty} and on $\|\rho^0\|_{L^q(\R^3)}$ such that the following holds. For all $T \geq 0$ there exists $N_0$ with the property that for all $N \geq N_0$ and all $t \leq T$ it holds that
\begin{align} \label{W_2.simplified}
    \W_2 (\rho_N(t),\rho_\ast(t)) &\leq C \left(\W_2 (\rho_N(0),\rho^0) +\frac{1}{\lambda_N}\right) e^{C t}.\quad
    %  \W_2 (\rho_N,\bar \rho_N) &\leq \left(\frac C \lambda N^{-\frac 1 2 } \left(\sum_i |V^0_i - \mfk V^0_i|^2 \right)^{\frac 1 2} + C \frac t \lambda\right ) e^{C t}
\end{align}
Moreover the following refined estimates hold

\begin{align} \label{W_2.thm}
    &\W_2 (\rho_N(t),\rho_\ast(t)) 
     \leq  \Biggl(\W_2 (\rho_N(0),\rho^0)  \\
     &\hspace{1cm}\left.+ C  \min\left\{\frac 1 {\lambda_N},t\right\}\left(\frac{\rs{|(V - \tilde V,R^2 (\Omega - \tilde \Omega))|_2(0)} }{\sqrt N} +    t \right) + C R t\right) e^{C  t},\\[0.3cm]
   & \dmin(t) \geq \frac{\dmin(0)} C  e^{-C t}, \label{dmin.thm}
\end{align}
and, for all $x\in \R^3$,
\begin{align}
    &\W_2(f_N(t),\rho_\ast(t)\otimes\delta_{(u_\ast(t)+g/(6\pi\gamma_N))})+\|u_N(t) - u_\ast(t)\|_{L^2(B_1(x))} \\
    &\quad \leq C\left( \W_2 (\rho_N(0),\rho^0) + \frac{\rs{|(V - \tilde V,R^2 (\Omega - \tilde \Omega))|_2(0)}}{\sqrt N} e^{- \frac{\lambda_N t}{C}} +\frac 1 {\lambda_N} \right) e^{Ct}. \qquad \label{u.thm}
    % \rh{\frac{1}{\lambda_N}}\deleted{\frac{1 -  e^{- \frac{\lambda_N t}{C C_1}}}{\lambda_N}\frac{|V-\tilde V|_2(0)}{\sqrt N} +  t\left( \frac 1 {\lambda_N} + R\right)}\right) e^{C t}  
    %  \W_2 (\rho_N,\bar \rho_N) &\leq \left(\frac C \lambda N^{-\frac 1 2 } \left(\sum_i |V^0_i - \mfk V^0_i|^2 \right)^{\frac 1 2} + C \frac t \lambda\right ) e^{C t}.
\end{align}
\end{theorem}

 We comment first on the statement of the theorem in Remark~\ref{rem:thm} before discussing the assumptions in Remark~\ref{rem:assumptions}.
\begin{rem} \label{rem:thm}
\begin{enumerate}[(i)]
    \item 
We recall that well-posedness of the transport-Stokes system for data $\rho^0 \in L^1\cap L^q$ for $q \geq 3$ has been shown in \cite{MecherbetSueur22}. 

\item The macroscopic quantities $(\rho_\ast,u_\ast)$ depend implicitly on $N$ through $\gamma_N$. We replace $\gamma_\ast$ by $\gamma_N$ in \eqref{eq:transport-Stokes}  in order to avoid an error that depends on $\gamma_N^{-1}-\gamma_\ast^{-1}$.

\item As will become clear in the proof of Theorem \ref{th:diagonal}, estimate \eqref{W_2.thm} implies \eqref {W_2.simplified}. In contrast to \eqref{W_2.simplified}, the estimate \eqref{W_2.thm} is sharp at $t=0$.

\item\label{it:errors} The three errors in \eqref{W_2.thm} additional to $\W_2 (\rho_N(0),\rho^0)$ account for different effects. 
The term involving \rs{$(V - \tilde V,R^2 (\Omega - \tilde \Omega))(0)$}  reflects an initial  layer at $0 \leq t \leq 1/\lambda_N$ in which the particle inertia induces an error of order $1/\lambda_N$ unless the data are well-prepared in the sense of \rs{$|(V - \tilde V,R^2 (\Omega - \tilde \Omega))(0)|_2  \ll \sqrt N$}. 
The other error of order $1/\lambda_N$ in \eqref{W_2.thm} is simply due to neglecting the inertia in the limit system. The last error of order of the particle radius $R$ can be viewed as a discretization error.

\item In view of the results in \cite{Hofer&Schubert}, one could expect an error of order $\phi_N = N R^3$ in addition to the error $R$ in \eqref{W_2.thm} due to effective increase of the viscosity of the suspension. However, assumption \eqref{ass:gamma} implies $\phi_N \ll R$, such that this error is of lower order in the present setting.

\item Estimate \eqref{u.thm} shows that both the fluid velocity and the particle velocities (up to self-interaction) converge to the limiting fluid velocity $u_\ast$. In particular, after an initial layer, the phase space density $f_N$ becomes monokinetic up to some errors. As for the spatial density, the error of order $1/\lambda_N$ is optimal since we neglect the particle inertia in the limit system. This is in contrast with the better estimate for the distance to the solution of the \emph{Vlasov-Stokes equation} which we show in \cite[Theorem 2.6]{HoferSchubert23b}.

\item One can improve the Lebesgue norm in \eqref{u.thm} to $L^p_\loc$ for $p <3$ (which is the usual threshold in view of integrability of the Oseen tensor) at the cost of slightly strengthening assumption \eqref{ass:Wasserstein.00}. 
See Remark \ref{p>2} for the detailed statement.

\item In the proof, we also show the following bound on the particle velocities and drag forces
    \begin{align} \label{Propagation.velocities}
        \forall 1 \leq i \leq N: \quad  |N F_i(t)| + \abs{V_i(t)} \leq C \left(1 + |V_i^0| e^{- \lambda_N \gamma_N t }\right),
    \end{align}
    where $F_i$ is the drag force particle $i$ exerts on the fluid, see \eqref{def:F}.

\end{enumerate}
\end{rem}

\begin{rem}\label{rem:assumptions}
\begin{enumerate}[(i)]
\item Assumption \eqref{ass:gamma} implies that the particle interaction strength stays finite. At the same time, \eqref{ass:Wasserstein.00} implies $\lambda_N \to \infty$. Together this implies that the particle inertia vanishes asymptotically.

We restrict ourselves to the case $\gamma_N \sim 1$ even though the Stokes number that quantifies the inertial forces on the particles is given by $(\lambda_N \gamma_N)^{-1}$ (and even by $(\lambda_N \gamma_N^2)^{-1}$ for $\gamma_N \ll 1$ as can be seen by an appropriate rescaling of the equations in this case). The case $\gamma_N \ll 1$ corresponds to a very dilute regime in which the particles essentially behave like single particles in the flow (see \cite{JabinOtto04}). 
On the other hand, in the case $\lambda_N \sim 1$, $\gamma_N \to \infty$, the transport-Stokes  equation \eqref{eq:transport-Stokes} is not the expected limit system but rather a system where an inertial term appears in the fluid equation in \eqref{eq:transport-Stokes}. On the level of the macroscopic equations, this has been formally observed in \cite{Hoferthesis} and proved (including the fluid inertia but without gravity) in \cite{HanKwanMichel21}.

Unfortunately, we are presently not able to treat the case when both $\lambda_N \to \infty, \gamma_N \to \infty$. Roughly speaking, the main additional difficulty in this case consists in estimates of the forces  $F_i = -\int_{\partial B_i} \sigma [u] n \dd \mathcal H^2$ that the particles exert on the fluid, which we use for controlling $\dmin$. 
 By a local version of the Stokes law, this force is given by $F_i = 6 \pi R(V_i-(u)_i )$ where $(u)_i$ is some local average of the velocity field $u$ around $B_i$ (cf. Lemma \ref{le:force.representation}). For $\gamma_N = NR \sim 1$, $|F_i| \lesssim  1/N$ follows then from uniform bounds on $(u)_i$ and $V_i$. For  $\gamma_N \to \infty$, though, one would need to show the expected cancellation  $V_i-(u)_i\sim \gamma_N^{-1}$ in order to show $|F_i| \lesssim  1/N$.

\item Assumption  \eqref{ass:Wasserstein.00}, is reminiscent of the conditions in \cite[Theorem 2.1]{Hauray09} and \cite[Theorem 2.4]{Hofer&Schubert} for first-order mean-field limits, but with the $2$-Wasserstein distance instead of the $\infty$-Wasserstein distance,  and $\rho^0 \in L^q(\R^3)$, $q > 3$ instead of $\rho^0 \in L^\infty(\R^3)$. Let us emphasize that $q >3$ is the minimal assumption that guarantees that the fluid velocity $u$ satisfies $\nabla u \in L^\infty((0,\infty) \times \R^3)$ and therefore allows to define characteristics by the Cauchy-Lipschitz theorem. The well-posedness result in \cite{MecherbetSueur22} includes the borderline case $q = 3$, but our current methods do not seem to allow to include this case.

Under the condition $\dmin(0) \gtrsim N^{-1/3}$, assumption \eqref{ass:Wasserstein.00} just requires $\W_2(\rho^0, \rho_N(0)) \to 0$ and $\lambda_N^{-1} \to 0$. 
On the other hand, let us point out that since $\rho_N$ is discrete whereas $\rho^0$ is continuous with respect to the Lebesgue measure, we have the \rh{(sharp)} discretization error 
\begin{align} \label{Wasserstein.lower.bound}
    \W_2(\rho_N(0),\rho^0) \geq c N^{\rh{-1/3}},
\end{align}
where $c$ is a constant that depends only on \rh{$\rho^0$} (cf. \cite[equation (1.11)]{Hofer&Schubert} for $q=\infty$). \rh{Indeed, let $c' > 0$ be such that $\mathcal L^n (\supp \rho^0) > c'^3 \mathcal L (B_1(0))$ and let $M$ be the set of all measurable subsets of $\R^3$ with $\mathcal L^n(A) \leq c'^3 \mathcal L (B_1(0))$.  
Then, 
\begin{align}
     \inf_{A \in M}  \int_{\R^3 \setminus A} \dd \rho^0 =: \delta > 0.
\end{align}

For any admissible transport plan $\gamma$ between $\rho_N(0)$ and $\rho^0$, we have
\begin{align}
    \int_{\R^3 \times \R^3} |x-y|^2 \dd \gamma(x,y)& \geq  \int_{\R^3 \times (\R^3 \setminus \cup_{i=1}^N B_{c' N^{-1/3}}(X_i)) } |x-y|^2 \dd \gamma(x,y) \\
    &\geq 
    \int_{\R^3 \setminus \cup_{i=1}^N B_{c N^{-1/3}}(X_i) } (c' N^{-1/3})^2 \dd \rho^0 \\
    &\geq  (c' N^{-1/3})^2 \inf_{A \in M}  \int_{\R^3 \setminus A} \dd \rho^0 \geq c'^2 \delta N^{-2/3}.
\end{align}

}

Thus, \rh{combining \eqref{Wasserstein.lower.bound} and} \eqref{ass:Wasserstein.00} implies  
\begin{align} 
    \dmin(0) \gg N^\rh{{-\frac 2 3 \frac {3q-2} {5q-3}}},
\end{align}
In particular, \rh{our assumptions require the asymptotic lower bound (which is sharp for $q = \infty$)}
\begin{align} \label{dmin.threshold}
    \dmin(0) \gg N^{\rh{-2/5}}
\end{align}
Moreover, combining with \eqref{ass:gamma} yields for all $N$ sufficiently large
\begin{align} \label{theta}
        \dmin(0) \gg R.
\end{align}
In particular, the particles do not overlap initially for $N$ sufficiently large.

Comparing to \cite{Hofer18MeanField, Hofer&Schubert}, where we assumed $\dmin \gtrsim N^{-1/3}$, 
assumption \eqref{ass:Wasserstein.00} is less stringent. We point out that under the assumption $\gamma_N \lesssim 1$, the assumption $\dmin \gtrsim N^{-1/3}$ in \cite{Hofer&Schubert} could be -- up to a logarithmic correction -- relaxed  towards
\begin{align}\label{eq:Wasserstein_lower_bound}
 \lim_{N \to \infty} \W_\infty(\rho^0, \rho_N(0))^{\frac 1 2}\left(1 + \frac{1}{N^{\frac 1 3 }\dmin}\right) = 0.
\end{align}
Just as in  \cite{Mecherbet19}, this allows to push the scaling of $\dmin$ up to $N^{-1/2}$ under the assumption that $\gamma_N$ is merely bounded  whereas in  \cite{Mecherbet19} $\gamma_N$ needs to be sufficiently small.

\item Assumptions \eqref{ass:V.Lipschitz} and \eqref{ass:V_infty} ensure  that during an initial layer in time of order $1/ \lambda_N$ where the system behaves very differently from its inertialess counterpart, the
particle density does not behave too wildly.
In particular, \eqref{ass:V.Lipschitz} ensures that the minimal distance does not implode during this initial layer. Note that the estimate on $|V_i^0|$ in \eqref{ass:V_infty} follows from assumption \eqref{ass:V.Lipschitz} and the bound on the kinetic energy in \eqref{ass:V_infty} provided the particle cloud is uniformly bounded in space.
 \end{enumerate}
\end{rem}

\subsection{Elements of the proof and outline of the the paper} \label{sec:strategy}
The proof of Theorem~\ref{th:diagonal} is given in Section \ref{sec:proof.main} together with key estimates for that proof. Nevertheless, we already explain the main strategy of the proof in this subsection without introducing much notation. \\

\rs{For the sake of simplicity of the presentation of the proof, we replace the models \eqref{eq:Xi}--\eqref{eq:fluid.micro.rot} and \eqref{eq:inertialess.rot} by corresponding models, where the particles do not rotate.  More precisely, we replace \eqref{eq:Xi}--\eqref{eq:fluid.micro.rot} by
\begin{align}	\label{eq:acceleration}
    &\dot X_i = V_i, \qquad \dot{V}_i = \lambda_N \left(g +  \frac{\gamma_N}{R} \int_{\partial B_i} \sigma[u_N] n \dd  \mathcal{H}^2 \right),\\
\label{eq:fluid.micro}
&\left\{\begin{array}{rl}
		- \Delta u_N + \nabla q_N = 0, ~\dv u_N &=0 \quad \text{in} ~ \R^3 \setminus \bigcup_i B_i, \\
		  u_N(x) &= V_i\quad \text{in} ~ B_i.
\end{array}\right.
\end{align}
We then prove the main theorem with all $\tilde \Omega$ removed and with $\tilde V$ being defined through the solution $v_N \in  \dot H^1(\R^3)$ to
\begin{align}\label{eq:V_tilde}
\left\{\begin{array}{rl}
		- \Delta v_N + \nabla p = 0, ~ \dv  v_N &=0 \quad \text{in} ~ \R^3 \setminus \bigcup_i B_i, \\
		  v_N(x) &= \tilde V_i \quad \text{in} ~  B_i, \\
		 -\int_{\partial  B_i} \sigma[ v_N] n \dd  \mathcal{H}^2  &= \frac g N, \quad \text{for all} ~ 1 \leq i \leq N, 
\end{array}\right.
\end{align}
Neglecting particle rotations does not affect the limit system to leading order in the particle volume fraction. Indeed, in \cite{Hofer18MeanField}, the transport-Stokes equation \eqref{eq:transport-Stokes} has been derived from an inertialess microscopic system as in \eqref{eq:inertialess.rot} but without particle rotations. The same limit has been obtained from \eqref{eq:inertialess.rot} including the rotations in \cite{Mecherbet19}. Only in the related result in \cite{Hofer&Schubert} where the influence of the the increase of the viscosity is studied, particle rotations play an important role. In the scaling considered in the present paper, this effect, proportional to the particle volume fraction, is negligible. 
Appendix \ref{sec:rot} contains the necessary adaptions of the proof to rotating particles.} \\

\textit{Modulated energy argument.} The starting point for the proof of Theorem~\ref{th:diagonal} is a modulated energy argument for the difference of the particle velocities with inertia (denoted by $V$) and without inertia (denoted by $\tilde V$).
\rs{Modulated/relative energy/entropy techniques are often performed for the study of singular limits in particular for kinetic equations. For a similar modulated energy argument in the same context of particles with small inertia interacting through a Stokes flow, we refer to \cite{HoeferPrangeSueur22}. For a related modulated energy argument in the context of mean-field limits with vanishing inertia, we refer to to \cite{CarilloChoi21}.

On the continuous level, a similar modulated/relative energy argument can be used to derive the hydrodynamic limit of the Vlasov-Stokes to the transport-Stokes equation. We refer to \cite{HoferSchubert23b} for details and a discussion of related results and techniques.

In the seminal paper \cite{Serfaty20}, a modulated energy argument is used to deal with first-order mean-field limits with Coulomb and Riesz type repulsive interaction. We emphasize that the modulated energy in \cite{Serfaty20} is quite different (it connects the microscopic and macroscopic energies) from ours. Moreover, the singularity of the  interaction in our case is sufficiently weak that we are able to deal with it through an adaptation of a classical method by Hauray.}

We show that the modulated energy
\begin{align} \label{E}
    E\coloneqq \frac 1{2N} \sum_i \abs{V_i-\tilde V_i}_2^2
\end{align}
satisfies the differential inequality
\begin{align} \label{diff.E}
    \frac {\dd E} {\dd t}\le -\frac{\lambda_N} C E+ C \sqrt E.
\end{align}
Here,  $C$ is a constant that depends on the behavior of the so called $N$-particle resistance matrix $\mathcal R \in \R^{3N \times 3N}$  which is the velocity-to-forces mapping   defined through solving the $N$-particle Stokes problem. More precisely the constant $C$ involves the coercivity of $\mathcal R$ as well as the derivative of  $\mathcal R$ with respect to the particle positions.
We therefore need uniform estimates on these quantities at least as long as we can propagate the minimal distance condition \eqref{ass:Wasserstein.00}.
Section \ref{sec:EstR} is devoted to these estimates. The estimates on the derivative of   $\mathcal R$ are obtained through the method of reflections similarly as in \cite{HoeferLeocataMecherbet22}.

The differential inequality \eqref{diff.E} yields control of $\rs{\sqrt{E}}$ up to an error of order $1/\lambda_N$ and after some initial layer in time. 
The challenge is to leverage this control to compare the microscopic system with the solutions of the transport-Stokes equation.
In our previous work \cite{Hofer&Schubert} (see also \cite{Mecherbet19} and \cite{Hofer18MeanField}), where we started from the inertialess microscopic dynamics \eqref{eq:inertialess.rot}, this has been achieved by adapting Hauray's method \cite{Hauray09} in order to simultaneously control the $\infty$-Wasserstein distance $\mathcal W_\infty(\rho_N,\rho)$ and the minimal particle distance $\dmin$ by a buckling argument. 

In the present setting, we face two additional difficulties. Firstly, we cannot work with the $\infty$-Wasserstein distance, since control on $\rs{\sqrt{E}}$ only fits well with the $2$-Wasserstein distance. Secondly, the error of order $\rs{1/}\lambda_N$ in the control of $\rs{\sqrt{E}}$ is too big in order to directly yield a control on $\dmin$ unless we assume a very fast rate of convergence $\lambda_N \to \infty$. \\

\textit{Generalization of Hauray's method to $p$-Wasserstein distances.}  As stated in Theorem~\ref{th:Hauray}, it turns out that Hauray's method can be adapted without much difficulty to using any $p$-Wasserstein distance. This might seem somewhat surprising since by definition only the $\infty$-Wasserstein distance gives uniform control on every single particle trajectory. (cf. \cite{CarrilloChoiHauray14}: \enquote{Because of the singularity in the interaction force, the natural transport distance to use
is the one induced by the $W_\infty$-topology}). However, this is actually not needed. The key ingredient in Hauray's buckling argument is an  estimate of the form
\begin{align}\label{eq:sumcontrol}
    \frac 1 N \sup_i \sum_{j \neq i} \frac{1}{|X_i - X_j|^\beta} \leq C(\rho)
    \left(1 + \frac{(\W_\infty(\rho_N,\rho))^{p_1}}  {(N^{-1/3} \dmin)^{p_2}} \right)
\end{align}
for all $\beta < d$ ($d$ being the space dimension) with some constants $p_1, p_2$ that depend only on $\beta$. We show in Lemma \ref{lem:sums.Wasserstein} that such an estimate persists when one replaces $W_\infty$ by $W_p$ for any $p \in [1,\infty)$. 
Section \ref{sec:Hauray} is devoted to the proof of Hauray's method replacing the $\infty$-Wasserstein distance by any $p$-distance. The result is not directly applicable in our setting due to the small inertia (and additional errors owing to the implicit nature of the interaction through the fluid). We still include the result, Theorem 2.1, to shed more light on this aspect of the proof of the main theorem and  because we think that it is of independent interest. \rs{In particular, Theorem~\ref{th:Hauray} can be adapted without much effort to a mean-field result for the following second-order binary system with asymptotically vanishing inertia
\begin{align}
    \dot X_i&=V_i,\label{eq:bin_inertial_1}\\
    \dot V_i&=\lambda_N\bra{-V_i+\frac 1N\sum_{j\neq i} K(X_i-X_j)},\label{eq:bin_inertial_2}
\end{align}
where $K$ is a (not too) singular interaction kernel. This is stated in Theorem~\ref{th:binary}. However, there are two important differences of system \eqref{eq:bin_inertial_1}--\eqref{eq:bin_inertial_2} to system \eqref{eq:acceleration}--\eqref{eq:fluid.micro}.

First, (using \eqref{force.repr} below) it is possible to see that using \eqref{eq:bin_inertial_1}--\eqref{eq:bin_inertial_2} is only a good approximation to the sedimentation problem if one already knows that the inertial velocities are comparable to the non-inertial velocities (which is not true in an initial layer in time). A better approximation (which we justify in Section \ref{sec:MinDist}) would include the implicit relations
\begin{align}
    \dot V_i &= \lambda_N\left(g + (u)_i - V_i \right), \label{Stokes.law.implicit}\\
    (u)_i&\approx \sum_{j\neq i} \Phi(X_i-X_j)(V_j - (u)_j), 
\end{align}
where $\Phi$ is the Oseen tensor and $(u)_i\in \R^3$ can be interpreted as mean-field velocity at particle $i$.  

Second, as is apparent from the proof of Theorem~\ref{th:binary}, for the explicit binary interaction \eqref{eq:bin_inertial_1}--\eqref{eq:bin_inertial_2} the modulated energy argument allows to control the velocities strongly in $l_p$ ($p\in[1,\infty]$) (in the sense that they converge in this norm to the inertialess velocities) relative to the initial $l_p$ distance. Thus, the advantage that the $p$-Wasserstein adaption of Hauray's approach gives over the $\infty$-Wasserstein version in the case of binary systems with asymptotically vanishing inertia, is to allow for more general initial data (for example with a bound on the kinetic energy instead of an $l_\infty$ bound).  We stress that for the implicit problem of sedimentation, the modulated energy argument relies on the $l_2$ coercivity of $\mR$ and it therefore seems unclear how to adapt the strategy to  $p=\infty$. 
}\\

\emph{Propagation of the control on the minimal particle distance.} \rs{As we see in the proof of Theorem 2.3, it is not too hard to propagate a minimal distance condition along the lines of Hauray's method for the second order binary system \eqref{eq:bin_inertial_1}--\eqref{eq:bin_inertial_2}. For the system \eqref{eq:acceleration}--\eqref{eq:fluid.micro} the propagation of condition \eqref{ass:Wasserstein.00} is much less straightforward.
The technical ingredients to  tackle this propagation are given in Section \ref{sec:MinDist}.}
Our main working horse is Lemma~\ref{le:force.representation}. It makes the heuristic \enquote{Stokes law} for the drag force rigorous, that the force on each particle is proportional to the difference of the velocity of the individual particle and the mean-field fluid velocity generated by the other particles \rs{(which justifies \eqref{Stokes.law.implicit} above)}. This implies an estimate for the difference of forces of neighbouring particles which is essential to control the minimal distance. However, we need as a second ingredient a uniform bound (after some initial layer in time) on the individual forces that the particles exert on the fluid in order to control this mean-field velocity.
By the Stokes type law of Lemma \ref{le:force.representation}, it is then enough to control $E$ (cf. \eqref{E}) and the $l_\infty$ norm of $V$. Finally,  the $l_\infty$ norm of $V$ we control using once again  Lemma \ref{le:force.representation} since it yields a differential inequality analogous to the one for the modulated energy.

%% file: 2.Meanfield.tex
\section{Generalization of Hauray's mean-field limit result to \texorpdfstring{$p$}{p}-Wasserstein distances}

\label{sec:Hauray}

We consider a general interaction kernel $K$ satisfying the conditions
\begin{align} \label{eq:C_alpha} \tag{$C_\alpha$}
	\forall x \in \R^d: |K(x)| + |x| |\nabla K(x)| \leq \frac{C}{|x|^\alpha}, ~ \dv K = 0,
\end{align}
with $\alpha < d-1$. The structural assumption $\dv K=0$ is not strictly necessary, see Remark~\ref{rem:Hauray} below. Then we have the following 

\begin{theorem} \label{th:Hauray}
Let $K \colon \R^d \to \R^d$ satisfy \eqref{eq:C_alpha} for some $\alpha < d -1$ and let $q > \frac{d}{d - \alpha -1}$. Let $\sigma^0 \in \mP(\R^d)\cap L^q(\R^d)$, where $\mP$ denotes the space of probability densities
and let $\sigma\in L^\infty([0,\infty);\mP(\R^d) \cap L^q(\R^d))$ be the global in time solution to 
\begin{align}\label{eq:K_eq}
    \left\{\begin{array}{rl}
		\partial_t \sigma + \dv\bra{(K \ast \sigma)\sigma} &= 0, \\
		\sigma(0) &= \sigma^0.
    \end{array}\right.
\end{align}
Let $X_i^0 \in \R^d$, $1 \leq i \leq N$, and consider the dynamics \ad{(until first collision)} defined through
\begin{align} \label{eq:X_i}
		\dot X_i(t) = \frac 1 N \sum_{j \neq i} K(X_i(t) - X_j(t)).
\end{align}
For $p\in [1,\infty]$ let
\begin{align}
		\sigma_N(t)\coloneqq \frac 1 N \sum_i \delta_{X_i(t)}, \quad \text{and} \quad
		\eta(t) \coloneqq \W_p(\sigma_N(t),\sigma(t)).
\end{align}
Assume that $\eta(0)$ and $\dmin(0)$ satisfy
\begin{align} \label{eq:initial_cond_Hauray}
		\lim_{N \to \infty} \eta(0)^{\frac{p(d-\alpha-1)}{d+pq'}}\bra{1+\dmin(0)^{-(\alpha+1)}N^{-(\alpha+1)/d}} = 0.
\end{align}
Then, for all $T>0$ and all $N$ sufficiently large (depending on $T$) \ad{there is no collision up to time $T$ and} it holds for all $t\in [0,T]$ that:
\begin{equation}\label{eq:eta.d_min.est}
	\begin{aligned}
		\eta(t) &\leq e^{Ct}\eta(0),\\
		\dmin(t) &\geq \dmin(0) e^{-C t},
	\end{aligned}
\end{equation}
	where $C$ depends only on $d$, $\alpha$, $q$, $p$,  the constant in \eqref{eq:C_alpha} and  $\|\sigma^0\|_{L^q}$.
\end{theorem}
\begin{rem} \label{rem:Hauray}
\begin{enumerate}[(i)]
    \item In the case $p=\infty$, \eqref{eq:initial_cond_Hauray} should be understood as
\begin{align}
    		\lim_{N \to \infty} \eta(0)^{d-\alpha-1}\bra{1+\dmin(0)^{-(\alpha+1)}N^{-(\alpha+1)/d}} = 0.
\end{align} 
This case is contained in \cite{Hofer&Schubert}[Theorem 2.4] which is adapted from \cite{Hauray09}.  

\item For well-posedness of the limiting equation \eqref{eq:K_eq} we refer to \cite{Hauray09} in the case $q=\infty$. The case $q<\infty$ follows from a straightforward adaption of \cite{MecherbetSueur22}.

\item It is possible to relax the structural assumption $\dv K=0$   at the cost of losing global in time existence of the limiting equation (cf. \cite{CarrilloChoiHauray14}).
\rs{\item A combination of the modulated energy method with the strategy of Theorem~\ref{th:Hauray} allows to obtain mean-field results in second-order binary self-damping systems with asymptotically vanishing inertia for physically more relevant assumptions on the initial data. It yields the following statement.}
\end{enumerate}
\end{rem}
\rs{
\begin{theorem}\label{th:binary}
Let $K,\alpha, q,\sigma^0,\sigma$ be as in Theorem \ref{th:Hauray}.
For given initial data $(X^0,V^0)\in (\R^d)^N\times (\R^d)^N$ let $X_i,V_i$ be the solution (until first collision) of the system \eqref{eq:bin_inertial_1}--\eqref{eq:bin_inertial_2}
with corresponding empirical measure 
\begin{align}
		\sigma_N(t)\coloneqq \frac 1 N \sum_i \delta_{X_i(t)}.
\end{align}
Assume that the initial data satisfy for some $p\in[1,\infty]$ the assumptions
\begin{align} 
    \label{ass:bin1}  \tag{D1}
    &\lim_{N \to \infty} \bra{\W_p( \sigma_N(0),\sigma^0)  + \frac 1 {\lambda_N}}^{\frac{p(d-\alpha-1)}{d+pq'}}\left(1 + \frac 1 {N^{\frac{1}d} \dmin(0)}\right)^{\alpha+1} = 0, \quad  \\[10pt]
    \label{ass:bin2} \tag{D2}
    &\forall N \in \N: ~ \forall\, 1 \leq i,j \leq N: ~ |V_i^0 - V_j^0| \leq \frac 12 \lambda_N |X_i^0 - X_j^0|,   \\[7pt] 
    \label{ass:bin3}\tag{D3}
    &\exists C_V > 0: ~ \forall N \in \N: ~ \frac 1 N  |V^0|_p^p \leq C_V.
\end{align}    
Then, for all $T>0$ and all $N$ sufficiently large (depending on $T$) there is no collision up to time $T$ and it holds for all $t\in [0,T]$ that:
\begin{equation}\label{eq:bin_Wass_est}
	\W_p(\sigma_N(t),\sigma(t))\le C\bra{\W_p(\sigma_N(0),\sigma^0)+\frac 1{\lambda_N}}e^{Ct}.
\end{equation}
	where $C$ depends only on $d$, $\alpha$, $q$, $p$,  the constants in \eqref{eq:C_alpha} and \eqref{ass:bin3}, and  $\|\sigma^0\|_{L^q}$.
% \begin{align*}
%     \tilde V_i=\frac 1N\sum_{j\neq i} K(X_i-X_j).
% \end{align*}
\end{theorem}
    We refer to \cite{CarilloChoi21} for a result on a related system with interaction kernels that are either Lipschitz continuous or given through repulsive Coulomb or Riesz potentials. We postpone the proof of Theorem \ref{th:binary} to Appendix \ref{sec:binary}.
}
The main strategy of the proof \rs{of Theorem~\ref{th:Hauray}} is identical to the original proof by Hauray in \cite{Hauray09}: We derive differential inequalities for both $\eta$ 
and $\dmin$ which eventually allow to control both quantities by a buckling argument. As mentioned in Section \ref{sec:strategy}, the key observation in the argument is that the singular sums 

\begin{align}\label{eq:S}
S_\beta\coloneqq \sup_i\sum_{j\neq i}\frac{1}{\abs{X_i-X_j}^\beta}
\end{align}
 can be controlled through the limiting density $\sigma$ as long as $\dmin$ is not too small compared to $\eta$ (as specified in \eqref{eq:initial_cond_Hauray}). 

We state this crucial estimate in the following Lemma. Its proof is an adaptation of \cite{Hauray09} and \cite[Lemma 3.1]{Hofer&Schubert} where the $\infty$-Wasserstein distance is treated. The generalization to any $p$-distance requires more care to estimate \enquote{bad sets} (where mass accumulates) since these sets cannot be localized anymore through a bound on the $p$-Wasserstein distance.

\begin{lem} \label{lem:sums.Wasserstein}
Let $\beta \in (0,d)$, $p \in [1,\infty]$, and $q > \frac d {d-\beta}$ with the dual exponent $q'$ such that $\frac 1q+\frac 1{q'}=1$ (and $q'=1$ if $q=\infty$). Furthermore let $X_i \in \R^d$, $1 \leq i \leq N$, and $\sigma_N = \frac 1 N \sum \delta_{X_i}$, as well as  $\sigma \in \mP(\R^d)\cap L^q(\R^d) $. Then, $S_\beta$ is estimated by
\begin{align}
 \label{eq:sums.Wasserstein.p}
	\frac 1NS_\beta\lesssim &\|\sigma\|_{L^q}^{\frac{\beta q'}d}  + 
 \frac{\|\sigma\|_{L^q}^{\frac{d-\beta}d \frac {p q'} {d+p q'}}}{N^{\beta/d} \dmin^{\beta}} (\W_p(\sigma_N,\sigma))^{\frac{(d-\beta)p}{d+pq'}} \\
	&+ \left(\frac{\|\sigma\|_{L^q}^{\frac{d-\beta}d }}{N^{\beta/d} \dmin^{\beta}}\right)^{\frac{(\beta+p)q'}{d - \beta + \beta q'+pq'}} (\W_p(\sigma_N,\sigma))^{\frac{(d-\beta)p}{d - \beta + \beta q'+pq'}}.
\end{align}

In the case $p = \infty$, this should be understood as
\begin{align} \label{eq:sums.Wasserstein}
	\frac 1NS_\beta
	\lesssim \|\sigma\|_{L^q}^{\frac{\beta q'}{d}} +  \frac{\|\sigma\|_{L^q}^{\frac{d-\beta}d}}{N^{\beta/d} \dmin^{\beta}}(\W_\infty(\sigma_N,\sigma))^{\frac{d-\beta}{q'}}.
\end{align}
\end{lem}

Above and in the following we write $A\ls B$ to mean $A\le CB$ for some global constant $C<\infty$ that does not depend on the specific particle configuration or $N$. 
However,
we will sometimes allow $C$ to depend on the initial density for the macroscopic system $\rho^0$ or $\sigma^0$ when indicated in the statements. 

Both in the proof of Lemma~\ref{lem:sums.Wasserstein} and further below, we will make use of the following fact (cf. \cite{Villani21}). Assume $\sigma_i \in \mP(\R^d), i=1,2$ with $\W_p(\sigma_1,\sigma_2)<\infty$ and that $\sigma_1$ is continuous with respect to the Lebesgue measure. Then, there exists an optimal transport map $T:\R^d\to \R^d$ such that $\sigma_2=T\# \sigma_1$, i.e. $\sigma_1(U)=\sigma_2(T(U))$ for all measurable sets $U\subset \R^3$, and
 \begin{align}
		\W_p(\sigma_1,\sigma_2)= \bra{\int_{\R^d} \abs{T(x)-x}^p \sigma_1(t,x)\dd x}^{1/p}, \quad \text{ for } p<\infty.
\end{align} 
\begin{proof}
We only consider $p< \infty$ here as the case $p = \infty$, $q=\infty$ has been proved in \cite[Lemma~3.1]{Hofer&Schubert} and the case $p = \infty$, $q<\infty$ is an interpolation between the two proofs. We fix $1 \leq i \leq N$ and adopt the short notation $\eta \coloneqq \mathcal W_p(\sigma_N,\sigma) $.\\

\noindent\textbf{Step 1:} \emph{Splitting of the sum.} 
If $\W_p(\sigma_N,\sigma) < \infty$, let $T$ be an optimal transport map for $\sigma_N,\sigma$, i.e.,
a map $T:\R^3\to \R^3$ such that
$\sigma_N = T\# \sigma$ and 
\begin{align}
	\eta=\W_p(\sigma_N,\sigma) = \bra{\int_{\R^d} |T(x) - x|^p \dd \sigma(x)}^{1/p}.
\end{align}
If $\W_p(\sigma_N,\sigma) =\infty$ the assertion of the lemma is trivial.

We fix $1\leq i \leq N$, without loss of generality assume $X_i=0$ and estimate 
\begin{align}
    \frac 1 N \sum_{j \neq i} \frac 1 {|X_j|^\beta}.
\end{align}
 For $r > 0$ to be chosen later, we split the set 
\begin{align}
    E&\coloneqq \{x \in \R^3 : T x \neq 0 \} = E_1 \cup E_2 \cup E_3, \\
     E_1 &\coloneqq \biggl\{x \in \R^3 : T x \neq 0,  |Tx | \geq \frac 1 2 |x| \biggr\}, \\
     E_2 &\coloneqq \biggl\{x \in \R^3 : T x \neq 0, r \leq  |Tx | \leq \frac 1 2 |x| \biggr\},\\
    E_3 &\coloneqq \biggl\{x \in \R^3 : T x \neq 0, |Tx | \leq r \biggr\}.
\end{align}
Then,
\begin{align} \label{est:split.E}
    \frac 1 N \sum_{j \neq i} \frac 1 {|X_j|^\beta} &=  \int_{E} \frac {1} {|T x|^\beta} \dd \sigma(x) 
     \leq \int_{E_1\cup E_2\cup E_3} \frac {1} {|T x|^\beta} \dd \sigma(x) .
\end{align}
\noindent\textbf{Step 2:} \emph{Estimate on $E_1$ and $E_2$.}
For all $\psi \in L^1(\R^3) \cap L^q(\R^3)$ and $\beta<\tfrac{d}{q'}$ we have
\begin{align} \label{eq:fractional.convolution}
	 	\||\cdot|^{-\beta} \ast \psi \|_{\infty} \lesssim \|\psi\|_{q}^{\frac{\beta q'}d}\|\psi\|_{1}^{\frac{d-\beta q'}d},
	 \end{align}
where here and in the following, for $p\in [0,\infty]$, we denote $\|\cdot\|_{q}=\|\cdot\|_{L^q(\R^3)}$. Indeed, for $R >0$ and $\beta q' < d$, i.e. $q > \frac d {d-\beta}$, by H\"older's inequality
\begin{align}
    \abs{\int_{B_R(0)} |x|^{-\beta}  \psi(x) \dd x} \leq \left( \int_{B_R(0)} |x|^{-\beta q'} \right)^{\frac 1 {q'}} \|\psi\|_{q} \lesssim R^{\frac{d - \beta q'}{q'}}  \|\psi\|_{q}.
\end{align}
Moreover, 
\begin{align}
    \int_{\R^d \setminus B_R(0)} |x|^{-\beta} \psi(x) \dd x\leq R^{-\beta} \|\psi\|_{1}.
\end{align}
Choosing
\begin{align}
    R = \left(\frac{\|\psi\|_{1}}{\|\psi\|_{q}}\right)^{\frac {q'}d}
\end{align}
yields \eqref{eq:fractional.convolution}. Applying this, we deduce	 
\begin{align} \label{est:E_1}
    \int_{E_1} \frac {1} {|T x|^\beta} \dd \sigma(x) \lesssim \int_{\R^3} \frac {1} {|x|^\beta} \dd \sigma(x) \lesssim \|\sigma\|_{q}^{\frac{\beta q'}d}.
\end{align}
We observe that in $E_2$ we have 
\begin{align}
|Tx - x| \geq |x| - |T x| \geq \frac 12 |x|\ge r.
\end{align}
Thus,
\begin{align} \label{est:E_3}
    \int_{E_2} \frac {1} {|T x|^\beta} \dd \sigma(x) \lesssim  r^{-(\beta+p)} \int_{E_2} |x - Tx|^p \dd \sigma(x) \lesssim   r^{-(\beta+p)} \eta^p.
\end{align}

\noindent\textbf{Step 3:} \emph{Estimate on $E_3$.}
Let 
$\J\coloneqq\{j \neq i : |X_j| \leq r\}$ and
\begin{align}
	 \psi = \frac{2^d}{N\dmin^{d}\abs{B_1(0)}} \sum_{j \in \J} \1_{B_{\dmin/2}(X_j)}. 
\end{align} 
Then, since $\abs{X_j}\ge \frac 23\abs{y}$ for all $y\in B_{\dmin/2}(X_j)$, and thanks to \eqref{eq:fractional.convolution} for $q=\infty$ we obtain
	 \begin{align} \label{est:E_2.0}
  \begin{aligned}
	 \hspace{-1cm}\int_{E_3} \frac {1} {|T x|^\beta} \dd \sigma(x) &\leq \frac 1 N \sum_{j \in \J} \frac 1 {|X_j|^{\beta}} \ls \||\cdot|^{-\beta} \ast \psi \|_{\infty}
	 \lesssim \frac{\left(N^{-1} |\J|\right)^{\frac{d-\beta}d}}{N^{\beta/d} \dmin^{\beta}}.
  \end{aligned}
	 \end{align}
  We claim that
\begin{align} \label{est:|J|}
   \frac{|\J|}{N} \lesssim r^{\frac d {q'}} \|\sigma\|_{q} + \|\sigma\|_{q}^{\frac{pq'}{pq'+d}} \mathcal \eta^{\frac{dp}{pq'+d}}.
\end{align}
Inserting \eqref{est:|J|} into \eqref{est:E_2.0}  yields
\begin{align} \label{est:E_2}
	 \int_{E_3} \frac {1} {|T x|^\beta} \dd \sigma(x) \lesssim \frac{1}{N^{\beta/d} \dmin^{\beta}} \left(r^{\frac d{q'}} \|\sigma\|_{q} + \|\sigma\|^{\frac{pq'}{d+pq'}}_{q} \mathcal \eta^{\frac{dp}{d+pq'}}\right)^{\frac {d-\beta}{d}}.
\end{align}

The remainder of this step is devoted to the proof of \eqref{est:|J|}. We introduce the set 
\begin{align}
    A \coloneqq \left\{x \in (B_{2r}(0))^c : T x \in B_r(0)\right\}
\end{align} 
and observe
\begin{align} \label{est:|J|.0} 
  \sigma(A) \geq \sigma_N(B_r(0)) - \sigma(B_{2r}(0))\geq \frac{|\J|}{N} - |B_{2r}(0)|^{\frac 1{q'}} \|\sigma\|_{q}.
\end{align}
In particular, this implies \eqref{est:|J|} if
$\sigma(A) \leq |\J|/(2N)$. 
It therefore remains to show \eqref{est:|J|} in the case $\sigma(A) \geq |\J|/(2N)$. We introduce 
\begin{align}
    \tilde \sigma := \sigma \1_A.
\end{align}
Moreover, we fix $\alpha >0$ which satisfies
\begin{align}
    \frac{(\alpha p)' q'} \alpha < d
\end{align}
where $(\alpha p)'$ is the H\"older dual of $\alpha p$. Note that the left-hand side converges to $0$ as $\alpha \to \infty$ and thus such an $\alpha$ exists.
We then estimate
\begin{align}
  \|\tilde \sigma\|_1 \leq \int_{\R^3} \frac{|x|^{\frac 1 \alpha}}{|x|^{\frac 1 \alpha}}\dd \tilde \sigma 
    &\lesssim  \left(\int_{\R^3} \frac{1}{|x|^{\frac {(\alpha p )'} \alpha}}\dd \tilde \sigma \right)^{\frac 1 {(\alpha p)'}} 
    \left(\int_{\R^3} |x|^{p} \dd \tilde \sigma \right)^{\frac 1 {\alpha p}} \\
    & \lesssim \|\tilde \sigma\|_q^{\frac{q'}{\alpha d}} \|\tilde \sigma\|_1^{\frac {d - \frac {(\alpha p)' q'}{\alpha} }{d(\alpha p)'}} \left(\int_{\R^3} |x|^{p} \dd \tilde \sigma \right)^{\frac 1 {\alpha p}}
\end{align}
Thus, using that
\begin{align}
    \alpha \left( 1 - \frac {d - \frac {(\alpha p)' q'}{\alpha} }{d(\alpha p)' }\right) = \frac 1 p + \frac {q'}{d},
\end{align}
we obtain
\begin{align}
  \eta \geq  \left(\int_{\R^3} |x|^{p} \dd \tilde \sigma \right)^{\frac 1 { p}}
  \geq \|\tilde \sigma\|_q^{-\frac{q'}{d}} \|\tilde \sigma\|_1^{\frac 1 p + \frac {q'}{d}} \gtrsim  \| \sigma\|_q^{-\frac{q'}{d}} \left( \frac{|\J|}{N} \right)^{\frac 1 p + \frac {q'}{d}},
\end{align}
which implies \eqref{est:|J|}.

\noindent\textbf{Step 4:} \emph{Conclusion.}
Inserting the estimates \eqref{est:E_1}, \eqref{est:E_3} and \eqref{est:E_2} into \eqref{est:split.E} yields
\begin{align}
     \frac 1 N \sum_{i \neq j} &\frac 1 {|X_j|^\beta} \\
     & \lesssim \|\sigma\|_{q}^{\frac{\beta q'}d} +\frac{1}{N^{\beta/d} \dmin^{\beta}} \left(r^{\frac d {q'}} \|\sigma\|_{q} + \|\sigma\|_q^{\frac{pq'}{d+pq'}} \mathcal \eta^{\frac{dp}{d+pq'}}\right)^{\frac {d-\beta}{d}} +  r^{-(\beta+p)} \eta^p.
\end{align}
Choosing  
\begin{align}
    r = \left(\frac{N^{\frac\beta d} \dmin^\beta \eta^p}{\|\sigma\|_{q}^{\frac{d-\beta}d}} \right)^\frac{q'}{d - \beta + \beta q' +p q'}
\end{align}
concludes the proof.
\end{proof}

\medskip

\begin{proof}[Proof of Theorem \ref{th:Hauray}]
We start with the observation that, since $\sigma$ satisfies a conservation law, we have $\norm{\sigma(t)}_1= \norm{\sigma^0}_1$ for all $t>0$ and $\sigma(t)\in \mP(\R^d)$ for all $t>0$. Moreover, since $\dv K=0$, equation \eqref{eq:K_eq} can be rewritten as transport equation $\partial_t\sigma+(K\ast\sigma)\cdot\nabla \sigma = 0$. Therefore $\norm{\sigma(t)}_q\le \norm{\sigma^0}_q$ for all $t>0$. We will frequently use that $\norm{\sigma}_q+\norm{\sigma}_1\ls 1$.  \\

\noindent\textbf{Step 1:} \emph{Estimate for the transport distance.}
Let $T_0$ be an optimal transport plan, such that
\begin{align}
    \W_p(\sigma_N(0),\sigma^0)= \bra{\int_{\R^d} \abs{T_0(x)-x}^p \sigma^0(t,x)\dd x}^{1/p}.
\end{align}
Both $\sigma_N$ and $\sigma$ can be written as 
\begin{align}
	\sigma(t,x) &= \sigma^0(Y(0,t,x)),\\
    \sigma_N(t,x) &= \sigma_N(0)(Y_N(0,t,x)),
\end{align}
where the second equation has to be understood in the sense $\sigma_N(t)=Y_N(0,t)\#\sigma_N(0)$ with a $\sigma_N$-almost everywhere defined map $Y_N$. Here, $Y_N,Y$ are the flow maps and satisfy
\begin{align}
	\partial_t Y(t,s,x) &= u(t,Y(t,s,x)),& ~ Y(t,t,x) &= x,\\
    \partial_t Y_N(t,s,x) &= u_N(t,Y_N(t,s,x)), & ~ Y_N(t,t,x) &= x,
\end{align}
where $u=K\ast\sigma$ and $u_N(X_i)=\frac 1 N \sum_{j \neq i} K(X_i - X_j)$. For the microscopic system this is immediate from \eqref{eq:X_i}, for the macroscopic system this follows by the method of characteristics. The flow maps exist as long as the underlying systems are well-posed.
Consider the map $T_t=Y_N(t,0,\cdot)\circ T_0\circ Y(0,t,\cdot )$, where $Y_N$ and $Y$ are the flow maps corresponding to $\sigma_N$ and $\sigma$, respectively. The map 
$T_t$ is well-defined because $T_0$ maps $\sigma(0)$-a.e. point into the support of $\sigma_N(0)$. We define 
\begin{align}\label{eq:fdef}
	\bar \eta(t)\coloneqq \sup_{0\le s\le t} \bra{\int_{\R^d} \abs{T_s(x)-x}^p \sigma(s,x)\dd x}^{1/p}.
\end{align}
Note that $\eta(t)\le \bar \eta(t)$. By \cite{Hofer&Schubert}[Lemma 3.2], with the convention $K(0)=0$, it holds that
\begin{align}
\begin{aligned}
		&\bar \eta(t)-\bar \eta(0) \\
  &\leq   \int_0^t \bra{\int_{\R^d}\left| \int_{\R^d} (K(T_s(x) - T_s(y)) - K(x-y)) \sigma(s,y) \dd y \right|^p\sigma(s,x)\dd x}^{1/p}\!\!\!\dd s. \quad
  \end{aligned}\label{eq:f_est}
\end{align}	
Since $K$ satisfies \eqref{eq:C_alpha}, we have the estimate
\begin{align}
	|K(x) - K(y)| \leq C |x-y| \left( \frac 1 {|x|^{\alpha+1}} + \frac 1 {|y|^{\alpha+1}} \right). \label{eq:K.Lipschitz} 
\end{align}
In the following we drop the index $s$ for $T_s$ and we use \eqref{eq:K.Lipschitz} to estimate the integrand for every $s$ by 
\begin{align}
\begin{aligned}
    &\abs{K(T(x)-T(y))-K(x-y)}\ls\\
    &\hspace{1cm}\bra{\frac{1}{\abs{x-y}^{\alpha+1}}+\frac{\1_{\{T(x) \neq T(y)\}}}{\abs{T(x)-T(y)}^{\alpha+1}}}\bra{\abs{T(x)-x}+\abs{T(y)-y}},
    \end{aligned}\label{eq:integrand_est}
    \end{align}
where we use the convention $\frac 0 0 = 0$. 
We bound both terms of the sum separately. First, we estimate the term in \eqref{eq:integrand_est} containing $\abs{T(x)-x}$ through
\begin{align}
    \abs{\int_{\R^d}\abs{T(x)-x}\bra{\frac{1}{\abs{x-y}^{\alpha+1}}+\frac{\1_{\{T(x) \neq T(y)\}}}{\abs{T(x)-T(y)}^{\alpha+1}}}\sigma(y)\dd y} \\
    \hspace{2cm}\ls \bra{1+\frac 1N S_{\alpha+1}}\abs{T(x)-x},
\end{align}
where we recall the notation $S_\beta$ from \eqref{eq:S} and where we have used that \eqref{eq:fractional.convolution} implies the estimate
$\int_{\R^d}\abs{x-y}^{-\alpha-1}\sigma(y)\dd y \lesssim \norm{\sigma}_q+\norm{\sigma}_1\ls 1$. Consequently
\begin{align}\label{eq:Wp_x}
\begin{aligned}
    &\bra{\int_{\R^d}\abs{\int_{\R^d}\abs{T(x)-x}\bra{\frac{1}{\abs{x-y}^{\alpha+1}}+\frac{\1_{\{T(x) \neq T(y)\}}}{\abs{T(x)-T(y)}^{\alpha+1}}}\sigma(y)\dd y}^p\!\!\sigma(x)\dd x}^{1/p}\\
    &\qquad\ls \bra{1+\frac 1N S_{\alpha+1}}\bar \eta.
    \end{aligned}
\end{align}
Second, for the term in \eqref{eq:integrand_est} containing $\abs{T(y)-y}$, we use a convolution estimate with weight, i.e. we employ H\"older's inequality after splitting the term in the parenthesis into powers $1/p$ (which is paired with $\abs{T(y)-y}$) and $1/p'=1-1/p$. Hence,
\begin{align}\label{eq:int_hoelder}
\begin{aligned}
    &\abs{\int_{\R^d}\abs{T(y)-y}\bra{\frac{1}{\abs{x-y}^{\alpha+1}}+\frac{\1_{\{T(x) \neq T(y)\}}}{\abs{T(x)-T(y)}^{\alpha+1}}}\sigma(y)\dd y}\\
    &\ls \bra{1+\frac 1N S_{\alpha+1}}^{1/p'} \\
    & \quad  \times \bra{\int_{\R^d}\abs{T(y)-y}^p\abs{\frac{1}{\abs{x-y}^{\alpha+1}}+\frac{\1_{\{T(x) \neq T(y)\}}}{\abs{T(x)-T(y)}^{\alpha+1}}}\sigma(y)\dd y}^{1/p}, \quad 
    \end{aligned}
\end{align}
where we have reasoned for the term with power $1/p'$ as above. Using estimate \eqref{eq:int_hoelder} and then integrating first with respect to $x$ yields again
\begin{align} \label{W_p.S_alpha+1}
\begin{aligned}
    &\bra{\int_{\R^d}\abs{\int_{\R^d}\abs{T(y)-y}\bra{\frac{1}{\abs{x-y}^{\alpha+1}}+\frac{\1_{\{T(x) \neq T(y)\}}}{\abs{T(x)-T(y)}^{\alpha+1}}}\sigma(y)\dd y}^p\!\!\sigma(x)\dd x}^{1/p}\\
    &\qquad\ls \bra{1+\frac 1N S_{\alpha+1}}\bar \eta.
    \end{aligned}
\end{align}
Combining \eqref{eq:integrand_est}, \eqref{eq:Wp_x} and \eqref{W_p.S_alpha+1} with \eqref{eq:f_est}, we arrive at 
\begin{align}\label{eq:gronwall_f}
		\bar \eta(t)-\bar \eta(0)\ls   \int_0^t \bra{1+\frac 1N S_{\alpha+1}(s)}\bar \eta(s) \dd s.
\end{align}	

\noindent\textbf{Step 2:} \emph{Estimate for the minimal distance.}
For fixed $i \neq j$ we use \eqref{eq:K.Lipschitz} to estimate 
	\begin{align}
		\abs{\dot{X_i}-\dot{X_j}}&=\frac 1 N\abs{\sum_{k \neq i} K(X_i - X_k)  - \sum_{k \neq j} K(X_j - X_k))} \\	
		& \hspace{-1cm}\lesssim |X_i - X_j| \frac 1 N  \bra{\sum_{k \not \in \{i,j\}}  
		\left( \frac{1}{|X_i - X_k|^{\alpha+1}} +  \frac{1}{|X_j - X_k|^{\alpha+1}} 	\right)+\dmin^{-(\alpha+1)}}.
	\end{align}	 
	Thus, 
	\begin{align}\label{eq:gronwall_d}
		\frac \dd { \dd t} \dmin \gtrsim - \dmin \frac 1N S_{\alpha+1}.
	\end{align}		

\noindent\textbf{Step 3:} \emph{Buckling argument.} We observe that by Lemma~\ref{lem:sums.Wasserstein} in combination with condition \eqref{eq:initial_cond_Hauray} it holds for all $N$ sufficiently large that
\begin{align}
\frac 1N S_{\alpha+1}(0)\leq 2C_\ast\|\sigma^0\|_q^{\frac{(\alpha+1)q'}{d}},
\end{align}
where $C_\ast$ is the implicit constant from Lemma~\ref{lem:sums.Wasserstein}. To see that the last term in the right-hand side of \eqref{eq:sums.Wasserstein.p} vanishes as $N\to \infty$, it is enough to raise it to the power $\frac{d-\beta+\beta q'+pq'}{d+pq'}$, and to see that the resulting power $\frac{(\beta+p)q'}{d+pq'}$ of $N^{-\beta/d}\dmin^{-\beta}$ is smaller than $1$ by the condition $q>\frac{d}{d-\alpha-1}$ for $\beta=\alpha+1$.
We then set $T_\ast(N)$ the maximal time up to which
\begin{align}\label{eq:bootstrap_hauray}
\frac 1N \bra{S_{\alpha+1}}< 4 C_\ast\|\sigma^0\|_q^{\frac{(\alpha+1)q'}{d}}.
\end{align}
  By continuity we have $T_\ast(N)>0$. Based on \eqref{eq:gronwall_f} and \eqref{eq:gronwall_d}, a Gronwall argument yields 
\begin{align}
\bar \eta(t)\le e^{CC_\ast t}\bar \eta(0),\\
\dmin(t)\ge e^{-CC_\ast t}\dmin(0),
\end{align}
for all $t\le T_\ast(N)$. This in turn implies 
\begin{align}\label{eq:dmin_buckling}
\frac 1N\dmin^{-{\alpha+1}}\le \frac 1N\dmin(0)^{-{\alpha+1}}e^{CC_\ast t}
\end{align}
and
\begin{align}
\begin{aligned}
   &\eta^{\frac{p(d-\alpha-1)}{d+pq'}}\bra{1+\dmin^{-(\alpha+1)}N^{-(\alpha+1)/d}} \\
   & \hspace{2cm}\le  \eta(0)^{\frac{p(d-\alpha-1)}{d+pq'}}\bra{1+\dmin(0)^{-(\alpha+1)}N^{-(\alpha+1)/d}} e^{CC_\ast t}. 
    \end{aligned}\label{eq:eta_buckling}
\end{align}
Due to \eqref{eq:initial_cond_Hauray}, the factors in front of the exponential both in \eqref{eq:dmin_buckling} and in \eqref{eq:eta_buckling} can be made arbitrarily small by choosing $N$ large enough. Hence it follows together with Lemma~\ref{eq:sums.Wasserstein}, that $T_\ast(N)\to \infty$ as $N\to \infty$.
\end{proof}

%% file: 3.ProofMain.tex
\section{Proof of the main theorem} \label{sec:proof.main}
\subsection{Notation}\label{subsec:notation}

We will use the following notation throughout. We also recall some previously introduced notation to collect everything in one place.

\begin{itemize}
    \item We write $A\ls B$ to mean $A\le CB$ for some global constant $C<\infty$ that does not depend on the specific particle configuration or $N$. However we will sometimes allow $C$ to depend on the initial density for the macroscopic system $\rho^0$ if this is indicated in the statement.

    \item We write $\Phi$ for the Oseen tensor, the fundamental solution of the Stokes equation, i.e.
    \begin{align}  \label{Oseen}
        \Phi(x) = \frac 1 {8 \pi} \left( \frac 1 {|x|} + \frac{x \otimes x} {|x|^3} \right).
    \end{align}
   We will use repeatedly that $\abs{\Phi(x)}+\abs{x}\abs{\nabla \Phi(x)}+\abs{x^2}\abs{\nabla^2\Phi(x)}\le \abs{x}^{-1}$ and that $\dv \Phi=0$. Moreover we us the convention that $\Phi(0)=0$.
    
    \item We denote $\dot H^1(\R^3)=\{w\in L^6(\R^3):\nabla w\in L^2(\R^3)\}$. This is a Hilbert space when endowed with the $L^2$ scalar product of the gradients.  We denote its dual by $H^{-1}(\R^3)$.
    We sometimes slightly abuse notation by letting context make clear what the target space is, i.e. whether the functions are scalar or vector valued. We will explicitly write the target space when we deem it necessary for clarity.

    \item For a fluid velocity $v$ that satisfies the Stokes equation in some domain, we write
    \begin{align}
        \sigma[v] = 2 e v- p \Id
    \end{align}
    for the fluid stress tensor where $ev = \frac 1 2 \left(\nabla v + (\nabla v)^T\right)$ is the strain and $p$ is the pressure associated to $v$ that we view as a Lagrange multiplier for the constraint $\dv v = 0$. For this reason we will also abusively use the same symbol $p$ for different pressures that can be recovered from corresponding fluid velocity fields.
    If we construct $\sigma$ from $v,p$ where $v,p$ do not satisfy the homogeneous Stokes equation, we write $\sigma[v,p]$.

     \item For a generic particle configuration $Y \in  \bra{\R^3}^N$ we write 
    \begin{align}
        \dmin \coloneqq \min_{i \neq j} |Y_i - Y_j|, \\
        d_{ij} \coloneqq \begin{cases}
                |Y_i - Y_j|, & \quad  \text{for} ~ i \neq j, \\
                \dmin & \quad  \text{for} ~ i = j.
                \end{cases}
    \end{align}
    
    \item For $\beta \in \R$ and $\dmin>0$, we denote
    \begin{align} \label{def.S}
        S_\beta \coloneqq \sup_i \sum_{j\neq i} \frac{1}{d_{ij}^\beta}.
    \end{align}
    \rh{Note that for $N \geq 2$, 
    \begin{align}
         \sup_i \sum_{j = 1}^N \frac{1}{d_{ij}^\beta} \leq 2 S_\beta
    \end{align}}
    We will usually write $\sum_j$ in short for $\sum_{j=1}^N$. If no other particle configuration $Y \in \bra{\R^3}^N$ is specified, $\dmin$, $d_{ij}$ and $S_\beta$ refer to the particle configuration $X$ given through the dynamics \eqref{eq:acceleration} --\eqref{eq:fluid.micro}. In particular, the quantities $\dmin$, $d_{ij}$ and $S_k$
    are then time dependent.

    \item For given $N\in \N$ and for times $t \geq 0$, we write $F \in (\R^3)^N$ for the forces exerted by the particles on the fluid, i.e. 
    \begin{align} \label{def:F}
        F_i = -\int_{\partial B_i} \sigma[u_N] n \dd \mathcal H^2,
    \end{align}
   where $u_N$ is the solution to \eqref{eq:fluid.micro} corresponding to the specific particle configuration $X$. The sign is due to the orientation of the normal, which is the \emph{inner} normal of the fluid domain.

\item For a vector $W\in \bra{\R^3}^N$ we write
\begin{align} \label{euclidean}
    \abs{W}_2^2=\sum_{i=1}^N \abs{W_i}^2
\end{align}
and
\begin{align} \label{maximum}
   \abs{W}_\infty=\max_{i=1,\dots,N}\abs{W_i}.
\end{align}
Furthermore we denote by $\norm{\cdot}$ the operator norm for linear operators $(\R^3)^N\to (\R^3)^N$, where $(\R^3)^N$ is equipped with the Euclidean norm $|\cdot|_2$. 

% The empirical measure is given by
% \begin{align}
%     \rho_N=\frac 1N \sum_i \delta_{X_i}.
% \end{align}
% We denote by $\mfk X$ the positions of the inertialess particles and the corresponding empirical density by $\varrho_N$.

\item We recall that $g \in S^2$ denotes a constant vector specifying the direction of gravity, and we denote 
\begin{align}
    \bar g = (g,\dots,g) \in (\R^3)^N.
\end{align}

\item We denote the Hausdorff measure restricted to $\partial B_i$ and normalized to weight one by $\delta_{\partial B_i} \coloneqq \frac {\mH^2|_{\partial B_i}}{\mH^2(\partial B_i)}$. At several points we will refer to the fact that for $G\in \R^3$ the solution \rs{$w\in \dot H^1(\R^3)$} to the problem
\begin{align}
        - \Delta w + \nabla p =  \delta_{\partial B_R(0)} G, ~ \dv  w &=0 \quad \text{in} ~ \R^3,\\
\end{align}
is given by
\begin{align}\label{eq:single_sol}
    w(x)=\begin{cases}\frac{G}{6\pi R},\quad &\text{for}~\abs{x}\le R,\\
    \bra{\Phi(x) + \frac {R^2} 6 \Delta \Phi(x)}G,\quad &\text{for}~\abs{x}>R.    
    \end{cases}
\end{align}

\item For a given configuration of particle positions, $Y\in (\R^3)^N$ with $\dmin>2R$ and given velocities $W\in (\R^3)^N$, let $w\in \dot H^1(\R^3)$ be the solution to 
\begin{align}
\left\{\begin{array}{rl}
        - \Delta w + \nabla p = 0, ~ \dv  w &=0 \quad \text{in} ~ \R^3 \setminus \bigcup_i B_i, \\
	    w &= W_i \quad \text{in} ~ B_i.
	\end{array}\right.\label{eq:fru}
\end{align}
where we denote, slightly abusing the notation, $B_i = \overline{B_R(Y_i)}$.
We denote the corresponding forces by
\begin{align}
	G_i\coloneqq -\int_{\partial B_i} \sigma[w] n \dd  \mathcal{H}^2.
\end{align}
 The so-called resistance matrix $\mR(Y)$ is the linear operator that maps the velocities $W=(W_i)_{1 \leq i \leq N}$ to the forces $G = (G_i)_{1 \leq i \leq N}$, i.e. 
\begin{align}
	(\mR(Y)W)_i=G_i.
\end{align}
The operator $\mR(Y)$ is symmetric and positive definite, which is a consequence of the fact that for solutions $w_1,w_2$ of \eqref{eq:fru} corresponding to $W_1,W_2\in (\R^3)^N$, respectively, we have by an application of the divergence theorem, 
\begin{align}\label{eq:R_bilin}
\begin{aligned}
	W_2\cdot\mR(Y)W_1&=-\sum_{i}\int_{\partial B_i}w_2\cdot \sigma[w_1]n\dd \cH^2 \\
 &=\int_{\R^3\setminus \cup_i B_i}\nabla w_2\cdot \sigma[w_1]\dd x \\
    &=2\int_{\R^3\setminus \cup_i B_i}\nabla w_2\cdot e w_1\dd x \\
    &=2\int_{\R^3}\nabla w_2\cdot e w_1\dd x=\int_{\R^3}\nabla w_2\cdot \nabla w_1 \dd x,
    \end{aligned}
\end{align}
where we use $\cdot$ both for the scalar product of vectors and of matrices $A\cdot B=\sum_{i,j=1}^3 A_{ij}B_{ij}$. Since $\mR$ is positive definite, it is in particular invertible.

We denote by $c_{\mR(X)}$ and $C_{\mR(X)}$ the smallest and largest eigenvalue of $\mR(X)$, respectively. Then we have
\begin{align}
	C_{\mR(Y)}&=\|\mR(Y)\| = \bra{\min_{|G|_2=1} G \cdot \mR^{-1}(Y) G}^{-1},  \label{eq:norm_C_R}\\
	c_{\mR(Y)}&=\min_{|W|_2=1} W \cdot \mR(Y) W = \|\mR^{-1}(Y)\|^{-1}.\label{eq:norm_c_R} 
\end{align}
\end{itemize}

\subsection{Modulated energy argument}

We recall the definition \eqref{eq:V_tilde} of the velocities
\begin{align}\label{eq:def_tildeV}
	\tilde V(t) \coloneqq \mR^{-1}(X(t)) \frac{\bar g}{N}.
\end{align}
The velocities $\tilde V$ are the velocities that inertialess particles would assume if they were in the position $X$ of the particles \emph{with} inertia. 

We remark first that \eqref{eq:norm_c_R} and \eqref{eq:def_tildeV} \ad{and $|g| = 1$} immediately imply 
    \begin{align}\label{eq:tildeV_l2}
        |\tilde V|_2\le N^{-1/2}c_{\mR(X)}^{-1}.
     \end{align} 

In the following we state and prove the central ingredient to the proof of Theorem~\ref{th:diagonal}, which gives control of the difference of $V$ and $\tilde V$ by means of a modulated energy argument.\\

In the following we denote 
\begin{align}
    c_{\mR,T}\coloneqq \inf_{0\le t\le T}c_{\mR(X(t))}.   
\end{align}

\begin{prop}\label{prop:energy}
    Let $T>0$ be given. Assume that 
    \begin{align}\label{eq:lambda_cond}
        \lambda_N \ge \frac{2}{N^2 c_{\mR,T}}\sup_{0\le t\le T}\|\nabla\mR^{-1}(X(t))\bar{g}\|.    
    \end{align}
    Then for all $t\in [0,T]$, the following estimate holds 
    \begin{align}\label{eq:E_control}
	    |V-\tilde{V}|_2(t) &\leq  |V-\tilde V|_2(0) e^{-\frac 12 \lambda_N N  c_{\mR,T} t} \\
     &\hspace{1cm}+ \frac{2 \sqrt 2 \bra{1 - e^{-\frac 12 \lambda_N N  c_{\mR,T} t}}}{\lambda_N N^{5/2}c_{\mR,T}^2}\sup_{0\le t\le T} \|\nabla \mR^{-1}(X)\bar g\|. \qquad \,
    \end{align}
\end{prop}

\begin{proof}
We start by introducing the modulated energy
\begin{align} \label{eq:mod.energy}
	E \coloneqq \frac{1}{2N}|V - \tilde V|_2^2.
\end{align}
By using \eqref{eq:acceleration} and the representation \eqref{eq:def_tildeV} we compute the time derivative of $E$ and estimate

\begin{align}
\begin{aligned}
	\frac {\dd E} {\dd t} &= \frac 1N(V - \tilde  V) \cdot \left(\lambda_N N \left( \frac {\bar g} N - \mR(X) V\right) - (V \cdot \nabla) \mR^{-1}(X) \frac {\bar g} N\right) \\
	&= -\lambda_N (V - \tilde  V) \cdot \mR(X) (V - \tilde V) \\
 &\hspace{1cm} -\frac{1}{N^2}(V - \tilde V) \cdot \bra{\bra{(V - \tilde  V) +  \tilde V} \cdot \nabla} \mR^{-1}(X) \bar g \\
	& \leq - \lambda_N c_{\mR(X)} |V - \tilde V|_2^2 + \frac{1}{N^2} \|\nabla \mR^{-1}(X)\bar{g}\||V - \tilde V|^2_2 \\
 & \hspace{1cm}+\frac{1}{N^2}\|\nabla \mR^{-1}(X)\bar g\||\tilde V|_2|V - \tilde V|_2. \\
	& \leq - 2 \lambda_N N c_{\mR(X)} E + \frac 2N  \|\nabla \mR^{-1}(X)\bar{g}\|\left(E + N^{-1/2}|\tilde V|_2 E^{1/2} \right).
 \end{aligned} \label{eq:dtE} 
\end{align}
We use \eqref{eq:tildeV_l2} and \eqref{eq:lambda_cond} to rewrite \eqref{eq:dtE} as 
\begin{align}
	\frac {\dd E} {\dd t} & \leq -  \lambda_N N c_{\mR(X)} E + \frac{2}{N^2c_{\mR(X)}}  \|\nabla \mR^{-1}(X)\bar{g}\|E^{1/2}.
\end{align}
This implies
\begin{align}
	\sqrt E(t) \leq  \sqrt E(0) e^{-\frac 1 2 \lambda_N N  c_{\mR,T} t} + 2 \frac{1 - e^{-\frac 12 \lambda_N N  c_{\mR,T} t}}{\lambda_NN^3 c_{\mR,T}^2}\sup_{0\le t\le T} \|\nabla \mR^{-1}(X)\bar g\|^2.
\end{align}
The statement follows directly.
\end{proof}

\subsection{Uniform estimates on the resistance matrix \texorpdfstring{$\mathcal R$}{mR} and the minimal distance \texorpdfstring{$\dmin$}{dmin}}

It is clear from the estimate \eqref{eq:E_control}, that it is crucial to control both $c_{\mR}$ and $\norm{\nabla \mR^{-1}(X)\bar g}$ in order to profit from Proposition~\ref{prop:energy}. This control is provided by the following two statements, which are proved in Section~\ref{sec:EstR}.
\begin{prop}\label{pro:resistance.estimates}
Let $Y\in (\R^3)^N$ satisfy $\dmin \ge 4 R$. Then, the following estimates hold.
	\begin{align}
		c_{\mR(Y)} &\gtrsim \frac{R}{1+R S_1}, \label{eq:c_R} \\
		C_{\mR(Y)} &\lesssim R. \label{eq:C_R}
	\end{align}
\end{prop}
\begin{prop}\label{pro:gradient.resistance}
 There exists $\delta > 0$ with the following property. Assume that $Y\in (\R^3)^N$ satisfies $\dmin \ge 4 R$ and
\begin{align}  \label{delta}
    R^3 S_3 \leq \delta.
\end{align}
Then
\begin{align}
    \|\nabla \mR^{-1}(Y) \bar g \| \lesssim S_2\bra{1+R^2S_2} .
\end{align}
\end{prop}

In both Proposition~\ref{pro:resistance.estimates} and Proposition~\ref{pro:gradient.resistance} the minimal distance of the particles plays a crucial role. We are able to retain control of the minimal distance with the help of the following two statements, which are proved in Section~\ref{sec:MinDist}. This control is also crucial in order to adapt the argument in the proof of Theorem \ref{th:Hauray} to the setting of particles with small inertia.

\begin{lem}\label{lem:dmin_control}
Let $t\ge 0$ and let $X$ be a solution to \eqref{eq:acceleration} and \eqref{eq:fluid.micro} until time $t$. Assume $\inf_{0\le s\le t} \dmin(s)\ge 8R$ as well as \eqref{ass:V.Lipschitz}. Then there exists a universal constant $C<\infty$ such that the following bound is satisfied.
\begin{align} \label{est:dmin.0}
    d_{\min}(t)\ge \frac 12d_{\min}(0)-C\int_0^t \dmin \bra{\abs{F}_\infty S_2 +R^3 S_2 S_6^{1/2}\abs{F}_2} \dd s.
\end{align}
\end{lem}

\begin{lem} \label{lem:VF.infty}
Let $t\ge 0$ and let $X$ be a solution to \eqref{eq:acceleration} and \eqref{eq:fluid.micro} until time $t$. Assume $\inf_{0\le s\le t} \dmin(s)\ge 8R$. Let $|F|_{2,t} = \sup_{0\le s \le t} |F|_2$ and $S_{\beta,t}=\sup_{0\le s \le t}S_\beta(s)$. Then the following bound holds true. 
\begin{align}
    \label{eq:F_infty}
    \hspace{0.3cm} \abs{F_i(t)} + R \abs{V_i(t)} \ls R\left(\frac 1 {\gamma_N}  + |V_i^0| e^{-\gamma_N \lambda_N t} + |F|_{2,t} \left( S_{2,t}^{\frac 1 2}  +R^{3}S_{2,t} S_{4,t}^{\frac 1 2}\right)\right) . \qquad 
\end{align}
\end{lem}

\subsection{Proof of Theorem~\ref{th:diagonal}}

We perform a buckling argument similar as in the proof of Theorem \ref{th:Hauray} in order to show \eqref{W_2.thm}--\eqref{dmin.thm}. Estimate \eqref{u.thm} will then be shown in the last step of the proof.
We denote $\eta(t) = \mathcal W_2(\rho_\ast(t),\rho_N(t))$.

\medskip 

\noindent \textbf{Step 1:} \emph{Setup of the buckling argument.}
By Lemma \ref{lem:sums.Wasserstein}, using that $\|\rho_\ast(t)\|_{q} \leq \|\rho^0\|_{q} \leq C$, we have for all $t > 0$ for which the particle dynamics is well-defined and as long as $\eta \leq 1$, that
\begin{align} \label{S_1+S_2}
     \frac{S_1 + S_2}N \leq \frac {C} 2 \left(1 + \frac{1}{N^{2/3} \dmin^{2}} \eta^{\frac 2 {3+2q'}} \right)
\end{align}
for some $C \geq 1$ that only depends on $\|\rho^0\|_{q}$.

Then, for any $N>0$, we consider $T_\ast = T_\ast(N) > 0$ the largest time for which it holds that 
\begin{align}
     R^2\dmin^{-3} &\leq 1 \qquad \text{for all } t\leq T_\ast, \label{bootstrap.dmin} \\
    \frac{S_1 + S_2}N + \frac{S_3}{N^2} &\leq C_1 \qquad  \text{for all } t\leq T_\ast,\label{bootstrap.S}
\end{align}
\ad{for some $C_1 \ls 1$ that will be chosen later}. We observe that \eqref{bootstrap.dmin} and \eqref{bootstrap.S} imply in particular, using \eqref{ass:gamma}, that for any $\delta>0$ we can choose $N$ sufficiently large such that
\begin{align}
        R^3 S_3 &< \delta \qquad \text{for all } t\leq T_\ast \label{eq:bootstrap.S_3}, \\
        \dmin &\geq 8 R\qquad \text{for all } t\leq T_\ast. \label{eq:bootstrap.dmin.R}
\end{align} 

We check that, for all $N$ sufficiently large, the inequalities in the buckling assumptions \eqref{bootstrap.dmin}--\eqref{bootstrap.S} are satisfied at $t=0$ which implies $T_\ast >0 $ by a continuity argument: First, \eqref{bootstrap.dmin} holds for large $N$ at $t=0$ due to \eqref{ass:gamma} and \eqref{dmin.threshold}.
Moreover, by \eqref{eq:sum.3}, we have
\begin{align} \label{S_3.dmin}
    \frac{S_3}{N^2} \leq C \frac{\log N}{N^2 \dmin^3}.
\end{align}
Using again \eqref{dmin.threshold}, the right-hand side above tends to $0$ as $N \to \infty$ at $t=0$. Combining this with \eqref{S_1+S_2}, again  \eqref{dmin.threshold}, and \eqref{ass:Wasserstein.00} yields that \eqref{bootstrap.S} also holds at $t=0$ for all $N$ sufficiently large.

\medskip

\noindent \textbf{Step 2:} \emph{Estimate of $\eta$.}
Analogous to the proof of Theorem \ref{th:Hauray}, we introduce 
\begin{align}\label{eq:fdef1}
	\bar \eta(t)\coloneqq \sup_{0\le s\le t} \bra{\int_{\R^3} \abs{T_s(x)-x}^2 \rho_\ast(s,x)\dd x}^{1/2},
\end{align}
where $T_s = Y_N(t,0,\cdot) \circ T_0 \circ Y(0,t,\cdot)$, $T_0$  with $\rho_N^0=T_0\#\rho^0$ is an optimal transport map, and $Y_N$ and $Y$ are flow maps for $\rho_N$ and $\rho_\ast$ respectively. As above, we will drop the index $s$ for $T_s$ in the following.
Then, we have again $\eta \leq \bar \eta$, $\eta(0) = \bar \eta(0)$, and by the same computation as in \cite[Lemma 3.2]{Hofer&Schubert} (cf. \eqref{eq:f_est}),
we obtain
\begin{align}\label{eq:split.f}
    \bar \eta(t) - \bar \eta(0) \leq& \int_0^t \left( \int_{\R^3} \left|u_N(s,T(x)) - \left(u_\ast(s,x) + \frac g {6 \pi \gamma_N}\right)\right|^2 \dd \rho_\ast(s,x)  \right)^{\frac 1 2} \dd s.
\end{align}
We split
\begin{align}
\begin{aligned}
    &\left( \int_{\R^3} \left|u_N(s,T(x)) - \left(u_\ast(s,x) + \frac g {6 \pi \gamma_N}\right)\right|^2 \dd \rho_\ast(s,x)  \right)^{\frac 1 2} \\
    &\qquad\qquad\leq  \left( \int_{\R^3} |u_N(s,T(x)) - v_N(s,T(x))|^2 \dd \rho_\ast(s,x)  \dd x \right)^{\frac 1 2}  \\
& \qquad\qquad \quad  +  \left( \int_{\R^3} \left|v_N(s,T(x)) - \left(u_\ast(s,x) + \frac g {6 \pi \gamma_N}\right)\right|^2 \dd \rho_\ast(s,x)   \right)^{\frac 1 2} ,\quad
\end{aligned}
\label{split.eta}
\end{align}
where $v_N$ is the solution to \eqref{eq:V_tilde}. The first term in \eqref{split.eta} is estimated by
\begin{align} \label{E/N}
    \left( \int_{\R^3} |u_N(s,T(x)) - v_N(s,T(x))|^2 \dd \rho_\ast(s,x)   \right)^{\frac 1 2} \dd s = \frac{|V-\tilde V|_2}{\sqrt N}.
\end{align}
To estimate the second term in \eqref{split.eta} we introduce another intermediate velocity field $w_N$ which is the solution to
\begin{align} \label{w_N}
 -\Delta w_N + \nabla p_N = \frac g N \sum_i \delta_{\partial B_i} ,  ~ \dv w_N = 0 \quad \text{in} ~ \R^3 ,
\end{align}
where $\delta_{\partial B_i}$ is the normalized uniform measure on $\partial B_i$. We use now  \cite[Proposition 3.12]{Hofer18MeanField}. Note that the assumption $\phi_N \alpha_3 \leq \delta$ from \cite[Proposition 3.12]{Hofer18MeanField} is satisfied because in the notation in \cite{Hofer18MeanField} (cf. \cite[Equation (29)]{Hofer18MeanField}), $\phi_N = NR^3$ and $\alpha_3 = S_3/N$. Thus,  $\phi_N \alpha_3 \leq \delta$ is exactly \eqref{eq:bootstrap.S_3}.
Then by \cite[Proposition 3.12]{Hofer18MeanField}, it holds that
\begin{align} \label{w_N-v_N}
    \|v_N - w_N\|_\infty  \ls\frac 1 {N} S_2(  R^3 S_2 + R).
\end{align}
By \eqref{eq:single_sol} we have
\begin{align}
    w_N(X_i) = \frac g {6 \pi \gamma_N} + \frac 1  N \sum_{j \neq i}  \left(\Phi(X_i - X_j) + \frac {R^2} 6 \Delta \Phi(X_i - X_j) \right)  g.
\end{align}
Hence, using the decay of $\Delta \Phi$ for the last term in the sum, we can modify \eqref{w_N-v_N} to obtain
\begin{align}\label{eq:v_n_w_n}
   \sup_{i} \left| v_N(s,X_i) - \frac g {6 \pi \gamma_N} - \frac 1 N \sum_{j \neq i} \Phi(X_i - X_j) g \right| \ls   \frac {S_2} {N}  (  R^3 S_2  + R)+\frac{R^2S_3}{N}. \quad
\end{align}
Using the convention $\Phi(0)=0$ we have $\frac 1N \sum_{j\neq i}(\Phi g)(X_i-X_j)=((\Phi g)\ast \rho_N)(X_i)$. Thus \eqref{eq:v_n_w_n} and $u_\ast=(\Phi g)\ast\rho$ yield 
\begin{align}
\begin{aligned}
 &\left( \int_{\R^3} \left|v_N(s,T(x)) - \left(u_\ast(s,x) + \frac g {6 \pi \gamma_N}\right)\right|^2 \dd \rho_\ast(s,x)  \dd x \right)^{\frac 1 2} \\
&\qquad\ls \left( \frac {S_2} {N}  (  R^3 S_2  + R)+\frac{R^2S_3}{N}\right)  \\
&\qquad\quad +  \left( \int_{\R^3} \left| \int_{\R^3} (\Phi(T(x) - T(y)) -  \Phi(x - y))g \dd\rho_\ast(s,y) \right|^2 \dd \rho_\ast(s,x) \right)^{\frac 1 2} .\qquad
\end{aligned}\label{u.tilde.u.ast}
\end{align}
Arguing as in Step 1 of the proof of Theorem~\ref{th:Hauray} we obtain 
\begin{align} \label{Phi.eta}
    &\left( \int_{\R^3} \left| \int_{\R^3} (\Phi(T(x) - T(y)) -  \Phi(x - y))g \dd\rho_\ast(s,y) \right|^2 \dd \rho_\ast(s,x) \right)^{\frac 1 2} \\
    &\hspace{9cm}\lesssim \left(1+\frac{S_2}N\right) \bar \eta.
\end{align}
Inserting \eqref{Phi.eta} into \eqref{u.tilde.u.ast} and the result together with \eqref{E/N} into \eqref{split.eta} yields
\begin{align}\label{eq:est_uN_u}
    &\left( \int_{\R^3} \left|u_N(s,T(x)) - \left(u_\ast(s,x) + \frac g {6 \pi \gamma_N}\right)\right|^2 \dd \rho_\ast(s,x)  \right)^{\frac 1 2}\\
    &\qquad\qquad\qquad\qquad\ls \frac{|V-\tilde V|_2}{\sqrt N} +  \frac 1 {N} S_2(  R^3 S_2  + R) + \frac{R^2S_3}{N}+ \left(1+\frac{S_2}N\right) \bar \eta.
\end{align}
To further estimate the right-hand side, we use Proposition \ref{prop:energy}. Using the buckling assumptions \eqref{bootstrap.S}--\eqref{eq:bootstrap.dmin.R}
as well as assumption \eqref{ass:gamma}, the bounds  from  Proposition \ref{pro:resistance.estimates} and Proposition \ref{pro:gradient.resistance} imply $c_{\mR,T} \gs 1/N$ and $\|\nabla \mR^{-1}(X)\bar g\| \ls N$. \rh{Note that this in particular implies that \eqref{eq:lambda_cond} is satisfied for $N$ large enough since \eqref{ass:Wasserstein.00} implies $\lambda_N \to \infty$.}
 Thus by Proposition \ref{prop:energy} we obtain that for all $N$ 
 large enough, we have
\begin{align} \label{V-V.tilde(s)}
    \frac{|V-\tilde V|_2(s)}{\sqrt N} \ls \frac{|V-\tilde V|_2(0)}{\sqrt N} e^{- \frac{\lambda_N}{C}s} +  \min\left\{\frac 1 {\lambda_N},s\right\}.
\end{align}
Inserting \eqref{V-V.tilde(s)} into \eqref{eq:est_uN_u}, and using again \eqref{bootstrap.S} and \eqref{ass:gamma}, delivers
\begin{align}\label{eq:est_uN_u2}
    &\left( \int_{\R^3} \left|u_N(s,T(x)) - \left(u_\ast(s,x) + \frac g {6 \pi \gamma_N}\right)\right|^2 \dd \rho_\ast(s,x)  \right)^{\frac 1 2}\\
    &\qquad\qquad\qquad\qquad\ls \frac{|V-\tilde V|_2(0)}{\sqrt N} e^{- \frac{\lambda_N}{C}s} +  \min\left\{\frac 1 {\lambda_N},s\right\} +R+ \bar \eta.
\end{align}
Integration in time of the exponential term yields 
\begin{align} \label{int.E}
  \int_0^t  \frac{|V-\tilde V|_2(s)}{\sqrt N} \dd s \ls   \min\left\{\frac 1 {\lambda_N},t\right\}\left(\frac{|V-\tilde V|_2(0)}{\sqrt N} +  t \right).
\end{align}
Thus, inserting \eqref{eq:est_uN_u2} into \eqref{eq:split.f}, and using \eqref{int.E}, we arrive at
\begin{align}
    \bar \eta(t) \leq \bar \eta(0) + C \min\left\{\frac 1 {\lambda_N},t\right\}\left( \frac{|V-\tilde V|_2(0)}{\sqrt N} +    t \right) + C R t  + C \int_0^t \bar \eta \dd s,
\end{align}
which implies by Gronwall's inequality and recalling $\eta \leq \bar \eta$, $\eta(0) = \bar \eta(0)$
\begin{align} \label{eta.final}
   \hspace{0.5cm}  \eta(t) \leq \bar \eta(t) \leq\left(\eta(0) + C \min\left\{\frac 1 {\lambda_N},t\right\}\left(\frac{|V-\tilde V|_2(0)}{\sqrt N} +    t \right) + C R t\right) e^{C t}.  \hspace{0.5cm}
\end{align}
This estimate corresponds to \eqref{W_2.thm}, once we have shown that $T_\ast \to \infty$ as $N \to \infty$. Moreover, we can simplify to obtain \eqref{W_2.simplified}. More precisely, we observe that
\begin{align}\label{eq:R_est}
    R \lesssim N^{-1} \leq N^{-1/3} \lesssim \W_2 (\rho_N(0),\rho^0)
\end{align} 
where we used \eqref{ass:gamma} in the first inequality and \eqref{Wasserstein.lower.bound} for the last inequality. On the other hand, 
\begin{align} \label{V-V.tilde(0)}
\begin{aligned}
    \frac{|V-\tilde V|_2(0)}{\sqrt N} &\leq \frac{|V|_2(0)}{\sqrt N} +  \frac{|\tilde V|_2(0)}{\sqrt N}  \\
    &\ls \bra{ 1 + \frac 1N c^{-1}_{\mR(X^0)}} \ls \bra{ 1 + \frac{1 + RS_1(0)}{NR}} \ls 1 ,
    \end{aligned}
\end{align}
where we used \eqref{ass:V_infty}, \eqref{eq:tildeV_l2}, \eqref{eq:c_R} and, in the last estimate, \eqref{bootstrap.S}.
Hence, 
\begin{align} \label{eta.final.simplified}
    \eta(t) \leq  \bar \eta(t) \ls \left(\eta(0) +\frac{1}{\lambda_N}\right) e^{C t}.
\end{align}

\noindent \textbf{Step 3:} \emph{Estimate of $\dmin$.}
To estimate $\dmin$ we first prove  the pointwise bound \eqref{Propagation.velocities}. 
We use that $F = \mR V$ and  $\tilde V = \mathcal R^{-1} \bar g/N$ to deduce   \rh{for $N$} large enough \rh{(such that \eqref{eq:bootstrap.dmin.R} holds)}
\begin{align}
\begin{aligned}
    |F|_2 &\leq C_{\mathcal R} |V|_2 \leq C_{\mathcal R} \left( |V - \tilde V|_2 + |\tilde V|_2 \right) \leq  C_{\mathcal R} \left( |V - \tilde V|_2 + N^{-\frac 1 2} c_{\mathcal R}^{-1}  \right)\\
    &\ls R \left( |V - \tilde V|_2 + N^{\frac 1 2}\right),
    \end{aligned} \label{eq:F_2_est}
\end{align}
where we used \eqref{eq:tildeV_l2} in the penultimate inequality and Proposition \ref{pro:resistance.estimates},  equation \eqref{bootstrap.S}, and \eqref{ass:gamma} in the last inequality. 
Combining this with \eqref{V-V.tilde(s)} and \eqref{V-V.tilde(0)} yields
\begin{align} \label{F_2}
    |F|_2 &\ls R N^{\frac 1 2}. 
\end{align}
Inserting this in \eqref{eq:F_infty} and using \eqref{bootstrap.S}, \eqref{eq:sum.4+}, \eqref{bootstrap.dmin} and \eqref{ass:gamma}
yields for all $1 \leq i \leq N$
\begin{align} \label{F_infty}
    |N F_i(t)| + \abs{V_i(t)} \ls \left(1 + |V_i^0| e^{- \lambda_N \gamma_N t } \right).
\end{align}
Inserting \eqref{F_2} and \eqref{F_infty} into  \eqref{est:dmin.0}, using \eqref{eq:sum.4+}, the buckling assumptions, and  \eqref{ass:gamma}
yields  \rh{for $N$} large enough
\begin{align}
    d_{\min}(t)\ge\frac 12d_{\min}(0)-C \int_0^t \dmin\left(1 +|V^0|_\infty e^{-\gamma_N\lambda_N s}\right)  \dd s.
\end{align}
Assumptions \eqref{ass:gamma} and \eqref{ass:V_infty} and Gronwall's inequality imply
\begin{align} \label{dmin.final}
    d_{\min}(t)\ge \frac 12d_{\min}(0) e^{-C (t + 1)}.
\end{align}

\noindent \textbf{Step 4:} \emph{Conclusion of the buckling argument.} 
It remains to show that $T_\ast(N) \to \infty$ as $N \to \infty$. Indeed, the assertions \eqref{W_2.thm} and \eqref{dmin.thm} then follow immediately from \eqref{eta.final} and \eqref{dmin.final}. 

We recall the  inequalities in \eqref{bootstrap.dmin} and \eqref{bootstrap.S} that define $T_\ast$. Combining \eqref{dmin.threshold} with \eqref{ass:gamma} and \eqref{dmin.final} yields that for all $t \leq T_\ast$ and all $N$ sufficiently large,
\begin{align} \label{boostrap.aposteriori.1}
     R^2 \dmin^{-3} \to 0.
\end{align} 
Moreover, a combination of \eqref{S_1+S_2},  \eqref{eta.final.simplified}, \eqref{dmin.final} and \eqref{ass:Wasserstein.00} 
 for the first term and a combination of \eqref{S_3.dmin}, \eqref{dmin.final}, and \eqref{dmin.threshold} for the second term, yields
\begin{align} \label{boostrap.aposteriori.2}
    \frac{S_1 + S_2}N + \frac{S_3}{N^2} &\leq \frac{C_1} 2,
\end{align} 
 for all  $t \leq T_\ast$ and all $N$ sufficiently  large \ad{provided $C_1$ is chosen sufficiently large}. From \eqref{boostrap.aposteriori.1}--\eqref{boostrap.aposteriori.2} we conclude $T_\ast(N) \to \infty$.

\medskip

\noindent \textbf{Step 5:} \emph{Proof of the estimate on the phase space densities in \eqref{u.thm}.} We observe that 
\begin{align}\label{eq:phase_spatial}
    &\W_2(f_N(t),\rho_\ast(t)\otimes\delta_{(u_\ast(t)+g/(6\pi\gamma_N))})\\
    &\le \W_2(\rho_N(t),\rho_\ast(t))+\left( \int_{\R^3} \left|u_N(s,T(x)) - \left(u_\ast(s,x) + \frac g {6 \pi \gamma_N}\right)\right|^2 \!\!\dd \rho_\ast(s,x)  \right)^{\frac 1 2}\!\!.
\end{align}
The first term on the right-hand side is estimated by \eqref{W_2.simplified}. For the second term on the right-hand side of \eqref{eq:phase_spatial} we refer to \eqref{eq:est_uN_u2}. Using the already obtained bound for $\bar\eta$ from \eqref{eta.final.simplified}, and absorbing the terms proportional to $R$ like in \eqref{eq:R_est}, finishes the proof.

\medskip

\noindent \textbf{Step 6:} \emph{Proof of the estimate on the fluid velocities in \eqref{u.thm}.}
In the following we fix $z\in \R^3$ and $K\coloneqq B_1(z)$. As in Step 2, let $v_N$ be the solution to \eqref{eq:V_tilde}. 
Then, by \eqref{eq:R_bilin}, the definition of $\mR$ and \eqref{eq:C_R}, we have
\begin{align}
    \|(u_N - v_N)(t)\|^2_{\dot H^1(\R^3)} &=  (V- \tilde V)\cdot\mR (V - \tilde V) \lesssim R |V - \tilde V|_2^2 \\
    & \lesssim \frac{|V-\tilde V|_2^2(0)}{N} e^{- \frac{\lambda_N}{C }t} +   \min\left\{\frac 1 {\lambda_N^2},t^2\right\},
\end{align}
where we used \eqref{V-V.tilde(s)} together with \eqref{ass:gamma} in the last step. 

Next, we recall the definition of $w_N$ from \eqref{w_N}  and by using \eqref{w_N-v_N} and Sobolev embedding we deduce
\begin{align} \label{u_N-w_N}
&\|(u_N - w_N)(t)\|_{L^6(K)} \\
& \hspace{2cm}\lesssim \frac{|V-\tilde V|_2(0)}{\sqrt N} e^{- \frac{\lambda_N t}{C}}+ \min\left\{\frac 1 {\lambda_N},t\right\} +\frac 1 {N} S_2(  R^3 S_2 + R). \qquad
\end{align}
It remains to compare $w_N$ to $u_\ast$. To this end, we first observe that for all $x \in \R^3$ that satisfy $|x - X_i| \geq \dmin$ for all $1 \leq i \leq N$, it holds for all $\beta \geq 0$ that
\begin{align} \label{sum.reference.x}
    \sum_i \frac 1 {|x-X_i|^\beta} \lesssim S_\beta.
\end{align}
Indeed, for given $x\in \R^3$, let $X_j$ be the minimizer of $|x-X_i|$. Then, we can estimate $|x-X_j| \geq \dmin$ and, for $i \neq j$
\begin{align}
    |x-X_i| \geq \frac 1 2 | X_i - X_j| - \frac 1 2 |x - X_j| + \frac 1 2 |x - X_i| \geq  \frac 1 2 | X_i - X_j|.
\end{align}
In the following we estimate the $L^p(K)$-norm of $w_N - u_\ast$ for $p \in [2,3)$. We will set $p=2$ at the end and draw further conclusions for $p>2$ in Remark \ref{p>2}. 
The implicit constant in $\lesssim$ might depend on $p$. 
We recall from 
\eqref{eq:single_sol} the explicit form of $w_N$. Thus, using the decay of $\Phi$ and $\Delta \Phi$, as well as $|x-X_i| \gtrsim |x - X_i| + R$ for all $x \in \R^3 \setminus B_i$ and \eqref{sum.reference.x}, we have
\begin{align} 
    \|w_N &- \Phi g  \ast \rho_N\|^p_{L^p(K)} \\
    &\lesssim  \frac 1 {N^p} \sum_i \int_{B_i} \frac 1 {|x - X_i|^p} + \frac 1 {R^p} \dd x +  \int_{K} \left(\frac 1 N \sum_i \frac {R^2} {R^3 + |x-X_i|^3} \right)^p \dd x \\
    &\lesssim N^{1-p} R^{3-p}+ \frac 1 {N^p} R^{2p} \sum_i \int_{B_{\dmin(X_i)}} \frac 1 {R^{3p} + |x-X_i|^{3p}} \dd x  + R^{2 p}  N^{-p} S_3^p \\
  &\lesssim N^{1-p} R^{3-p} + R^{2 p} N^{-p} S_3^p. \label{w_N-Phi.g.ast.rho_N}
\end{align}
It remains to estimate $u_\ast - \Phi g  \ast \rho_N = \Phi g \ast (\rho_\ast - \rho_N)$.
For an optimal transport plan $T$, we compute 
\begin{align} \label{split.u_ast-Phi.g.ast.rho_N}
\begin{aligned}
   & \|u_\ast - \Phi g\ast \rho_N\|^p_{L^p(K)} \\
    &= \int_K \abs{\int_{\R^3} ((\Phi g)(x-T(y)) - (\Phi g)(x-y)) \rho_\ast(y) \dd y }^p \dd x \\
    &\lesssim \int_K \!\left(\int_{\R^3} \1_{\{\abs{x-T(y)}>\dmin\}}|T(y) - y| \left(\frac 1 {|x-T(y)|^2} + \frac 1 {|x-y|^2}\right) \rho_\ast(y) \dd y \right)^p \!\!\dd x \\
   &\quad+ \int_K \left(\int_{\R^3} \1_{\{\abs{x-T(y)}
   <\dmin\}} \left(\frac 1 {|x-T(y)|} + \frac 1 {|x-y|}\right) \rho_\ast(y) \dd y \right)^p \dd x .
   \end{aligned}
\end{align}
For the second term on the right-hand side, we estimate firstly
\begin{align} \label{split.u_ast-Phi.g.ast.rho_N.1}
    \int_K \left(\int_{\R^3} \1_{\{\abs{x-T(y)} < \dmin\}} \frac 1 {|x-y|} \rho_\ast(y) \dd y \right)^p \dd x 
    & \lesssim N^{-\frac{2p}{3}},
\end{align}
where we used \eqref{eq:fractional.convolution} together with $\|\rho_\ast\|_\infty \lesssim 1$ and, for all $x \in \R^3$, 
\begin{align} \label{particles.in.dmin.ball}
    \|\1_{\{\abs{x-T(\cdot)} < \dmin\}} \rho_\ast\|_1 = \frac 1 N \#\{i : |x - X_i| \leq \dmin\} \lesssim 
\frac 1 N,
\end{align}
by definition of $\dmin$.
Secondly, using H\"older's inequality with $1=1/p+1/p'$, \eqref{particles.in.dmin.ball} and Fubini
\begin{align} \label{split.u_ast-Phi.g.ast.rho_N.2}
\begin{aligned}
   & \int_K \left(\int_{\R^3} \1_{\{\abs{x-Ty} < \dmin\}} \frac 1 {|x-T(y)|} \rho_\ast(y) \dd y \right)^p \dd x \\
    & \lesssim N^{-\frac{p}{p'}} \int_{\R^3}\int_K  \1_{\{\abs{x-Ty} < \dmin\}} \frac 1 {|x-T(y)|^p}  \rho_\ast(y) \dd x \dd y  
    \lesssim N^{-\frac{p}{p'} } \dmin^{3-p}.
    \end{aligned}
\end{align}
For the first term on the right-hand side of \eqref{split.u_ast-Phi.g.ast.rho_N} we proceed similarly as in the proof of Theorem~\ref{th:Hauray}. We denote $s(x,y)=\1_{\{\abs{x-T(y)}>\dmin\}}\abs{\frac{1}{\abs{x-y}^{2}}+\frac{1}{\abs{x-T(y)}^{2}}}$ and apply H\"older's inequality with three factors to the product $$\abs{T(y)-y}s(x,y)=\abs{T(y)-y}^{1-2/p}\bra{\abs{T(y)-y}^{2/p}s(x,y)^{2/(2+p)}}s(x,y)^{p/{2+p}}$$ with exponents $2p/(p-2)$, $p$ and $2$, respectively
\begin{align}
\begin{aligned}
    &\abs{\int_{\R^3}\abs{\Id-T}s(x,\cdot)\rho_\ast\dd y}\\
    & \hspace{0.5cm}\ls \bra{\int_{\R^3}\abs{\Id-T}^2\rho_\ast\dd y}^{\frac{p-2}{2p}}\bra{\int_{\R^3}\abs{\Id-T}^2s(x,\cdot)^{\frac{2p}{2+p}}\rho_\ast\dd y}^{\frac 1p} \\
   & \hspace{7cm} \times \bra{\int_{\R^3}s(x,\cdot)^{\frac{2p}{2+p}}\rho_\ast\dd y}^{\frac 12}. \quad
    \end{aligned} \label{eq:trihoelder}
\end{align}
In view of the fact that $p<3$, we have by \eqref{sum.reference.x}
\begin{align}
   \int_{\R^3}s(x,y)^{\frac{2p}{2+p}}\rho_\ast(y)\dd y\ls 1+\frac 1N S_{4p/(2+p)},
\end{align}
and thus \eqref{eq:trihoelder} implies
\begin{align}
    &\int_{\R^3}\abs{\Id-T}s(x,\cdot)\rho_\ast\dd y 
    \\
    &\hspace{0.5cm}\ls \W_2(\rho_\ast,\rho_N)^{\frac{p-2}{p}}\bra{1+\frac {S_{4p/(2+p)}}N }^{\frac 12}\bra{\int_{\R^3}\abs{\Id-T}^2s(x,\cdot)^{\frac{2p}{2+p}}\rho_\ast\dd y}^{\frac 1p}. \qquad \label{eq:est_trih}
\end{align}
Using  \eqref{eq:est_trih} and integrating we arrive at
\begin{align}\label{split.u_ast-Phi.g.ast.rho_N.3}
   \hspace{4mm}  \int_K\abs{\int_{\R^3}\abs{\Id-T}s(x,\cdot)\rho_\ast\dd y}^p\dd x
    & \ls \W_2(\rho_\ast,\rho_N)^p\bra{1+\frac 1N S_{4p/(2+p)}}^{1+\frac p2}. \hspace{4mm}
\end{align}
Inserting \eqref{split.u_ast-Phi.g.ast.rho_N.1}, \eqref{split.u_ast-Phi.g.ast.rho_N.2} and \eqref{split.u_ast-Phi.g.ast.rho_N.3} into \eqref{split.u_ast-Phi.g.ast.rho_N} yields
\begin{align}
    &\|u_\ast - \Phi g\ast \rho_N\|_{L^p(K)} \\
    &\hspace{1cm} \lesssim \W_2(\rho_\ast,\rho_N) \bra{1+\frac 1N S_{4p/(2+p)}}^{\frac 1 2 +\frac 1 p} + N^{-\frac{1}{p'} } \dmin^{\frac{3-p}{p}} + N^{-\frac{2}{3}}.
\end{align}
Combining this with \eqref{u_N-w_N} and \eqref{w_N-Phi.g.ast.rho_N} yields
\begin{align}
    &\|u_\ast - u_N\|_{L^p(K)} \\
    \lesssim  & \frac{|V-\tilde V|_2(0)}{\sqrt N} e^{- \frac{\lambda_N t}{C}}+ \min\left\{\frac 1 {\lambda_N},t\right\}+ \W_2(\rho_\ast,\rho_N) \bra{1+\frac 1N S_{4p/(2+p)}}^{\frac 1 2 +\frac 1 p} \\
    & +\frac 1 {N} S_2(  R^3 S_2 + R) +  N^{1-p} R^{3-p} + R^{2p} N^{-p} S_3^p + N^{-\frac{1}{p'} } \dmin^{\frac{3-p}{p}} + N^{-\frac{2}{3}}.
\end{align}
It remains to absorb some lower order terms. Indeed, using \eqref{bootstrap.S}, \eqref{eq:sum.3}, \eqref{ass:gamma}, \eqref{dmin.threshold}, $\dmin\ls N^{-1/3}$ (which is implied by the fact that $\rho_N$ has a non-trivial limit), and \eqref{dmin.thm}, all the terms in the second line above are $O(N^{-2/3})$ and $N^{-2/3} \leq \W_2(\rho_\ast,\rho_N)$ by  \eqref{Wasserstein.lower.bound}.
Thus,
\begin{align} \label{u_ast-u.p}
   & \|u_\ast - u\|_{L^p(K)} \\
    & \hspace{.2cm} \lesssim \frac{|V-\tilde V|_2(0)}{\sqrt N} e^{- \frac{\lambda_N t}{C}}+  \min\left\{\frac 1 {\lambda_N},t\right\}+ \W_2(\rho_\ast,\rho_N) \bra{1+\frac 1N S_{4p/(2+p)}}^{\frac 1 2 +\frac 1 p}\!\!\!. \qquad
\end{align}
Using once more \eqref{bootstrap.S} and \eqref{W_2.thm}, together with \eqref{V-V.tilde(0)} and $R \lesssim N^{-1} \lesssim \W_2(\rho_\ast,\rho_N)$ as above, we conclude for $p=2$.

\begin{rem} \label{p>2}
For $p \in (2,3)$ and $x\in \R^3$, we have  
\begin{align}
\|u_N(t) - &u_\ast(t)\|_{L^p(B_1(x))} \\
    &\leq C\left( \W_2 (\rho_N(0),\rho^0) + \frac{|V-\tilde V|_2(0)}{\sqrt N} e^{- \frac{\lambda_N t}{C}} +\min\left\{\frac 1 {\lambda_N},t\right\} \right) e^{Ct}  
\end{align}
 under the strengthened assumption 
\begin{align}
    \lim_{N \to \infty} \bra{\W_2(\rho^0, \rho_N(0))  + 1/\lambda_N}^{\frac {2\left(3 - \frac{4p}{2 +p}\right)}{3+2q'}}\left(1 + N^{-\frac{4p}{3(2+p)}} \dmin^{-\frac{4p}{2+p}}(0)\right) = 0.
\end{align}
This follows from \eqref{u_ast-u.p} together with Lemma \ref{lem:sums.Wasserstein}, \eqref{W_2.simplified}, \eqref{dmin.thm} and \eqref{dmin.threshold}.
\end{rem}

%% file: 4.EstimatesR.tex
\section{Estimates for the resistance matrix \texorpdfstring{$\mR$}{mR}}\label{sec:EstR}

Let $c>0$ be fixed. This section is devoted to the proof of the estimates on the three quantities $c_{\mR(Y)}$, $C_{\mR(Y)}$ and $|\nabla \mR^{-1}(Y)\bar g|$ stated in Proposition \ref{pro:resistance.estimates} and Proposition \ref{pro:gradient.resistance}. 
The proof of Proposition \ref{pro:resistance.estimates} follows from standard variational principles for the Stokes equation.
The proof of Proposition \ref{pro:gradient.resistance} is based on the method of reflections similarly  to a related result in \cite{HoeferLeocataMecherbet22}.

\rh{Throughout this section, we will consider generic particle positions $Y_i$, velocities $W_i$ and forces $G_i$ (instead of the ones corresponding to the dynamics \eqref{eq:acceleration}--\eqref{eq:fluid.micro}). To  avoid heavy notation, we will continue to use $B_i$ for the balls corresponding to these generic configurations, i.e. $B_i = \overline{B_R(Y_i)}$.}

\begin{proof}[Proof of Proposition~\ref{pro:resistance.estimates}]
	\textbf{Proof of \eqref{eq:C_R}.} Let $W \in (\R^3)^N$ and let $v$ be the solution to \eqref{eq:fru}. Then, on the one hand, $v$ is the minimizer of the Dirichlet energy among all solenoidal fields equal to $W_i$ on $B_i$ for all $i=1,\dots,N$. On the other hand, the form $W\cdot \mR(Y)W$ is equal to the Dirichlet energy (equation \eqref{eq:R_bilin}). Combining both, we have 
	\begin{align}
		W \cdot \mR W =  \| \nabla v \|^2_2 \leq \|w\|_{\dot H^1}^2, 
	\end{align}
	for all $w \in \dot H^1(\R^3)$ with $\dv w = 0$ and $w = W_i$ in $B_i$. To prove \eqref{eq:C_R} it is thus enough to construct a field $w$ with $w=W_i$ on $B_i$ and such that
		\begin{align}
		\|w\|^2_{\dot H^1(\R^3)} \leq C R |W|_2^2.
	\end{align}
	This is possible under the assumption $|Y_i - Y_j| \geq 4 R$. For the construction we refer to \cite[Lemma 4.2]{HoeferJansen20}.\\
	
    \noindent\textbf{Proof of \eqref{eq:c_R}.} Let now $v$ be the solution to the homogeneous Stokes equation with given forces $G \in (\R^3)^N$, i.e.
	\begin{align}
    \left\{\begin{array}{rl}
	    - \Delta  v + \nabla p = 0, ~ \dv   v &=0 \quad \text{in} ~ \R^3 \setminus \bigcup_i  B_i, \\
		  \nabla   v &= 0 \quad \text{in} ~  \cup_i B_i, \\
		 -\int_{\partial  B_i} \sigma[ v] n \dd  \mathcal{H}^2  &= G_i, \quad \text{for all} ~ 1 \leq i \leq N.
   \end{array}\right.
	\end{align}
	Then, denoting $W_i[G] = v(Y_i)$, we have  $G = \mR W[G]$, and thus
	\begin{align} \label{char.c_R}
	    c^{-1}_{\mR(Y)} = \sup_{|G| = 1} G \cdot \mR^{-1} G = \sup_{|G| = 1} G \cdot W[G] = \sup_{|G| = 1} \| v\|^2_{\dot H^1}.
	\end{align}
	On the other hand, we have
	\begin{align} \label{projection}
		\| v\|_{\dot H^1} \leq \|w\|_{\dot H^1},
	\end{align}
	where $w \in \dot H^1$ is given as the unique solution to the Stokes problem
	\begin{align}
	 -\Delta  w + \nabla  p =   \sum_i G_i \delta_{\partial B_i}, ~ \dv  w = 0 \quad \text{in} ~ \R^3,
	\end{align}
    and we recall that $\delta_{\partial B_i}$ is the normalized Hausdorff measure restricted to $\partial B_i$. Indeed, \eqref{projection} follows from the observation that
	\begin{align}
	    \int_{\R^3} \nabla v \cdot \nabla ( v -  w) \dd x= -\sum_i \int_{\partial  B_i}   v\cdot (\sigma[v -  w] n) \dd  \mathcal{H}^2 = 0 ,
	\end{align}
	because $ v$ is constant in $B_i$ and $\int_{\partial  B_i}   \sigma[v -  w] n \dd \mathcal{H}^2= 0$.	
Referring to \eqref{eq:single_sol} we have
	\ad{\begin{align}
		w(x) &= \sum_i w_i(x), \\
  		w_i(x) &\coloneqq\begin{cases} \frac{G_i}{6 \pi R} & \text{ in } B_i, \\
								(\Phi(x - Y_i) - \frac{R^2}{6} \Delta \Phi(x-Y_i)) G_i & \text { in } \R^3 \setminus B_i,
				\end{cases}
	\end{align}}
and where $\Phi$ is the fundamental solution of the Stokes equation (cf. \eqref{Oseen}). Thus, an explicit calculation reveals that, recalling the definition of $S_1$ from \eqref{def.S} and using Young's inequality,
	\begin{align*}
		\| v\|^2_{\dot H^1} &\leq \|w\|^2_{\dot H^1}=-\sum_i \int_{\partial B_i}w\cdot(\sigma[\ad{w_i}]n) \dd \mathcal{H}^2 \\
 & \lesssim  \sum_i \frac{|G_i|^2 }{R} + \sum_i \sum_{j \neq i} \frac{G_i G_j}{|Y_i - Y_j|} \lesssim |G|^2 (R^{-1} + S_1).
	\end{align*}
	Combining this with \eqref{char.c_R} concludes the proof.
\end{proof}
\subsection{Estimate for the gradient of the resistance matrix}

This subsection is devoted to the proof of Proposition \ref{pro:gradient.resistance}. We start with the following observation.
\begin{lem}
Let $A=(a_{ij})_{ij}\subset \R^{N\times N}$. Then the operator norm of $A$ with respect to the Euclidean metric satisfies the following estimate:
\begin{align}
\norm{A}\le \left(\sup_{i} \sum_j|a_{ij}|\sup_{k} \sum_l|a_{lk}|\right)^{\frac 1 2}.\label{eq:infty_1}
    % \norm{A}_{\op}&\le \sqrt{\norm{A}_{1,\infty}\norm{A}_{\infty,1}}\label{eq:infty_1}\\
    % \norm{A}_{\op}&\le \sqrt{\norm{AA^T}_{2}}\label{eq:Frobenius}\\
\end{align}
\end{lem}
This easily extends to the situation, where $A$ is a linear map $(\R^3)^N\to(\R^3)^N$.

\begin{proof}
Let $v\in \R^N$. Then we compute
\begin{align}\label{eq:norm_Av}
    |Av|_2^2&=\sum_i \left(\sum_j a_{ij}v_j\right)^2\\
    &=\sum_i\sum_{j_1}\sum_{j_2}a_{ij_1}a_{ij_2}v_{j_1}v_{j_2}.
\end{align}
We use $v_{j_1}v_{j_2}\le \frac 12 (v_{j_1}^2+v_{j_2}^2)$ and observe that both resulting sums are the same to obtain
\begin{align}
    |Av|_2^2&\le\sum_i\sum_{j_1}\sum_{j_2}a_{ij_1}a_{ij_2}v_{j_1}^2\\
    &\le \sup_i \sum_{j_2} \abs{a_{ij_2}}\sum_i\sum_{j_1}a_{ij_1}v_{j_1}^2\\
    &\le \sup_i \sum_{j_2} \abs{a_{ij_2}}\sup_{j_1}\sum_i\abs{a_{ij_1}}\sum_{j_1}v_{j_1}^2.
    %&=\norm{A}_{\infty,1}\norm{A}_{1,\infty}\norm{v}_2^2.
\end{align}
 This proves \eqref{eq:infty_1}. 
\end{proof}

In order to estimate $\nabla \mathcal R^{-1}\bar{g}$, we rely on the method of reflections, similarly as in \cite[Section 4.2]{HoeferLeocataMecherbet22}. 
For a particle configuration $Y \in (\R^3)^N$ as in the statement of Proposition \ref{pro:gradient.resistance}, we introduce the function $ w[Y] \in \dot H^1(\R^3)$ as the solution to 
\begin{align}
    	\label{eq:inertialess.intermediate}
\left\{\begin{array}{rl}
		- \Delta   w[Y] + \nabla   p[Y] = 0, ~ \dv    w[Y] &=0 \quad \text{in} ~ \R^3 \setminus \bigcup_i B_i, \\
		  \nabla  w[Y] &= 0 \quad \text{in} ~ \bigcup_i  B_i, \\
		- \int_{\partial B_i^l} \sigma[w[Y]] n \dd  \mathcal{H}^2  &= g \quad \text{for all} ~ 1 \leq i \leq N.
\end{array}\right.
\end{align}
Then, by definition of $\mR(Y)$, we have
\begin{align}
    (\mR^{-1}(Y) \bar g)_i &=   w[Y](Y_i), \\
    \nabla_{Y_j}(\mR^{-1}(Y) \bar g)_i &= \nabla_{Y_j}   w[Y](Y_i) + \delta_{ij} \nabla_{x}   w[Y](Y_i) =\nabla_{Y_j}   w[Y](Y_i)\label{eq:nabla_X.R},
\end{align}
provided these gradients exist and where the second term after the first identity in the second line vanishes because of \eqref{eq:inertialess.intermediate}.

Furthermore, we define, $  w_0[Y] \in \dot H^1(\R^3)$ as the solution to 
\begin{align}
    -\Delta   w_0[Y] + \nabla   p_0 [Y]  = \sum_i \delta_{\partial B_i} g,
\end{align}
where, as before, $\delta_{\partial B_i}$ is the normalized surface measure on $\partial B_i$.
By linearity we can write $ w_0 = \sum_i  w^{(i)}_0$, where $-\Delta w_0^{(i)}+\nabla p=\delta_{\partial B_i}g$ and it is well-known that the functions have the explicit form
    \begin{align} \label{eq:w_0.explicit}
          w^{(i)}_0(x) = \begin{cases} 
          \frac g {6 \pi R} &\qquad \text{in } B_i, \\  \left(1 - \frac {R^2}{6} \Delta \right)\Phi(x - Y_i) g & \qquad \text{otherwise}.
        \end{cases}
\end{align}
The strategy is now to show that $   w_0[Y]$ satisfies the desired estimates and to show that $  w[Y] -   w_0[Y] $ is very small.
\begin{lem} \label{lem:nabla.w_0}
    Under the assumptions of Proposition \ref{pro:gradient.resistance} we have
    \begin{align}
        \left\|\left(\nabla_{Y_j}   w_0[Y](Y_i) \right)_{ij}\right\| \le  S_2.
    \end{align}
\end{lem}
\begin{proof}
    By the explicit form \eqref{eq:w_0.explicit}, it holds that
    \begin{align}
       |\nabla_{Y_j}   w_0[Y](Y_i)| \le \frac{1}{d_{ij}^2}.
       %\deleted{|\nabla_{x}   w_0[Y](Y_i)| \le \sum_{k \neq i} \frac{1}{d_{ik}^2}}. 
    \end{align}
    We conclude by \eqref{eq:infty_1}.
\end{proof}

To estimate the difference $  w -   w_0$, we define for $k \in \N$ the functions $  w_k[Y]$ obtained through the method of reflections from $  w_0[Y]$.
More precisely, following the notation from \cite{Hofer18MeanField}, we introduce $Q_i$ as the solution operator that maps a  function $\varphi \in H^1_{\sigma}(B_i)$ to the solution $\psi \in \dot H^1_{\sigma}(\R^3)$ of 
\begin{equation} \label{eq:Q}
\left\{
\begin{array}{rl}
	- \Delta \psi + \nabla p = 0, ~\dv \psi &=  0 \quad \text{in} ~ \R^3 \setminus B_i, \\
	\nabla \psi &= \nabla \varphi \quad \text{in} ~ B_i, \\
\int_{\partial B_i} \sigma[\psi] n\dd \mathcal{H}^2&=0.
	\end{array}\right.
\end{equation} 
Here the index $\sigma$ in $H^1_\sigma(B_i)$ and $\dot H_\sigma^1(\R^3)$ denotes the subspace of divergence free functions in $H^1(B_i)$ and $\dot H^1(\R^3)$, respectively.
Note that the operator $Q_i$ depends only on the position of $Y_i$.

Then, we know from \cite[Proposition 3.12]{Hofer18MeanField} that 
\begin{align} \label{eq:MOR}
	  w_k[Y] \coloneqq (1 - \sum_i Q_i)^k   w_0[Y] \to   w[Y] \quad \text{in} ~ \dot H^1_\sigma(\R^3) \cap L^\infty(\R^3).
\end{align}
Note that the assumption $\phi_N \alpha_3 \leq \delta$ from \cite[Proposition 3.12]{Hofer18MeanField} is satisfied because in the notation in \cite{Hofer18MeanField} (cf. \cite[Equation (29)]{Hofer18MeanField}) $\phi_N = NR^3$ and $\alpha_3 = S_3/N$. Thus,  $\phi_N \alpha_3 \leq \delta$ is exactly \eqref{delta}.

We recall the following decay estimates:
\begin{lem}[{\cite[Lemma 3.10 and Proposition 3.12]{Hofer18MeanField}}]\label{lem:decay.Q}
Let $\varphi \in H^1_{\sigma}({B}_i) $ such that $ \nabla \varphi \in L^\infty({B}_i)$. Then,  for all $x \in \R^3 \setminus B(Y_i,2R)$, it holds that
\begin{align} \label{est:Q_i.pointwise}
	|\nabla^l (Q_i \varphi)(x)| \lesssim \frac { R^3}{|x - Y_i|^{l+2}} \|\nabla \varphi\|_{L^\infty(B_i)} .
\end{align}
Moreover, 
	\begin{align} \label{eq:Q_i.average}
		Q_i \varphi = \varphi - \fint_{\partial B_i} \varphi \dd \mathcal{H}^2 \quad \text{in} ~ B_i.
	\end{align}
\end{lem}

In order to  estimate the derivatives with respect to the particle positions $Y_i$ we first estimate the derivative of $Q_i$ with respect to $Y_i$.

\begin{lem} \label{lem:decay.nabla.Q}
Let $\varphi \in H^2_{\sigma}(B(Y_i,2R))$ such that $ \nabla \varphi \in W^{1,\infty}(B(Y_i,2R)) $.
Then, \\ $\nabla_{Y_i} (Q_i \varphi)\in L^2_\loc(\R^3)$ and for all $x \in \R^3 \setminus B(Y_i,2R)$ and all $l \in \N$, it holds that
\begin{align}  \label{est:nabla.Q_i.pointwise}
	|\nabla^l  \nabla_{Y_i} (Q_i \varphi)(x)| \lesssim \frac {C R^3}{|x - Y_i|^{l+2}} \left(\|\nabla^2 \varphi\|_{L^\infty(B_i)} + \frac 1 {|x - Y_i|} \|\nabla \varphi\|_{L^\infty(B_i)} \right).
	\end{align}
\end{lem}

\begin{proof}
	We write $Q_i = Q[Y_i]$ to denote the dependence on $Y_i$ more explicitly.
	By considering the defining equation for $Q$, \eqref{eq:Q}, we observe that
	\begin{align} \label{eq:Q[xi]}
		(Q[Y_i] \varphi)(x) = (Q[0] \varphi(\cdot + Y_i)) (x - Y_i).
	\end{align}
	Taking the gradient with respect to $Y_i$ and using Lemma \ref{lem:decay.Q} yields the desired estimates.
\end{proof}

To estimate the shape derivative $\nabla_Y   w_n[Y]$, we rewrite the series expansion \eqref{eq:MOR} as
\begin{multline} \label{eq:MOR.expansion}
	  w_n[Y] =   w_0[Y] - \sum_{i_1} Q_{i_1}   w_0[Y] + \sum_{i_1} \sum_{i_2 \neq i_1} Q_{i_2} Q_{i_1}   w_0[Y] \\
    - \sum_{i_1} \sum_{i_2 \neq i_1} \sum_{i_3 \neq i_2}  Q_{i_3} Q_{i_2} Q_{i_1}   w_0[Y] + \dots \\
	+ 
	(-1)^n \sum_{i_1} \sum_{i_2 \neq i_1} \dots \sum_{i_n \neq i_{n-1}}  Q_{i_n} Q_{i_{n-1}} \dots Q_{i_1}   w_0[Y].
\end{multline}
This representation can be directly deduced from \eqref{eq:MOR} by just using
$Q_i Q_i = Q_i$ (see \cite[Section 2]{HoferVelazquez18} for details).
As explained in \cite[Section 4.2]{HoeferLeocataMecherbet22}, this representation is more convenient for estimating the shape derivatives because it avoids terms of the form $Q_i \nabla_{Y_i} Q_i$ which are a priori not well-defined.

We recast this into the more compressed form
\begin{align} 
	  w_n[Y] &= \sum_{k=1}^N \sum_{m=0}^n  \psi^{(k)}_m[Y] \label{eq:w_n.sum.psi}, \\
	\psi^{(k)}_m[Y] &\coloneqq \sum_{\underline l \in \Pi_m^{(k)}} (-1)^m Q_{l_1} \dots Q_{l_m}    w^{(k)}_0, \\
	\Pi_m^{(k)} &\coloneqq \! \left\{ \underline l\! =\! (l_1, \dots l_m)\! \in\! \{1,\dots,N\}\colon\! l_m \neq k,  l_i \neq l_{i+1}, \! \text{ for }i\!=\!1,\dots,m-1 \right\}.
\end{align}
where we used that $Q_k    w^{(k)}_0 = 0$ to obtain the condition $l_m \neq k$ in the definition of $\Pi_m^{(k)}$.

\begin{lem} \label{lem:est.nabla.psi}
      There exists a universal constant $C$ such that under the assumptions of Proposition \ref{pro:gradient.resistance} we have for all $n \geq 1$, all $1 \leq i \leq N$ and all $x \in B_i$
    \begin{align}
     \hspace{3mm} \sum_k |\nabla_{Y_j} \psi^{(k)}_n[Y](x)|\lesssim   \frac {n} {d_{i j}^2}  \left(CR^3 S_3\right)^{n-2}\bra{R^3 S_3+R^2 S_2(\1_{n\ge 2}+R^3S_3)}. \hspace{3mm} \label{est:nabla_Y.psi} 
    \end{align}
\end{lem}
\begin{proof} Fix $1 \leq i \leq N$ and denote
    \begin{align}
            \Pi_m^{(i,k)} \coloneqq \{\underline l = (l_1, \dots l_m) \in \Pi_m^{(k)} : l_1 \neq i \}.
    \end{align}

    We compute
    \begin{align} \label{eq:split.nabla.psi}
    \begin{aligned}
        \nabla_{Y_j} \psi^{(k)}_n[Y] =& \delta_{kj} \sum_{\underline l \in \Pi_n^{(k)}} (-1)^n Q_{l_1} \dots Q_{l_n} \nabla_{Y_j}    w^{(k)}_0 \\
        & + \sum_{m=1}^{n} \sum_{\substack{\underline l \in \Pi_n^{(k)} \\ l_m = j}} (-1)^n Q_{l_1} \dots \nabla_{Y_j} Q_j \dots Q_{l_n}    w^{(k)}_0,
        \end{aligned}
    \end{align}
    where the last term has to be interpreted in the obvious sense if $m=1$ or $m=n$.

    We observe that by \eqref{eq:Q_i.average} we have $|Q_i \varphi (x)| \lesssim R \|\nabla \varphi\|_{L^\infty(B_i)}$ for all $x \in B_i$.
    Relying also on \eqref{eq:w_0.explicit} to compute $\nabla_{Y_j}    w^{(k)}_0$ and using  \eqref{est:Q_i.pointwise}, we estimate the first term on the right-hand side in \eqref{eq:split.nabla.psi} for $x \in B_i$
    \begin{align}
       & \biggl| \sum_{\underline l \in \Pi_n^{(j)}} (-1)^n \left(Q_{l_1} \dots Q_{l_n} \nabla_{Y_j}   w^{(j)}\right)(x)\biggr|\\
        & \leq \biggl| \sum_{\underline l \in \Pi_n^{(i,j)}} (-1)^n \left(Q_{l_1} \dots Q_{l_n} \nabla_{Y_j}   w^{(j)}\right)(x)\biggr| \\ & \hspace{2mm} + \biggl| \sum_{\underline l \in \Pi_{n-1}^{(i,j)}} (-1)^n \left(Q_i Q_{l_1} \dots Q_{l_{n-1}} \nabla_{Y_j}   w^{(j)}\right)(x)\biggr|  \\
         &\leq C^n \sum_{\underline l \in \Pi_n^{(i,j)}} \frac {R^3} {d^2_{i l_1}} \frac {R^3} {d^3_{l_1 l_2 }}
         \dots \frac {R^3} {d^3_{l_{n-1} l_n}} \frac {1} {d^3_{l_{n} j}} \\ & \hspace{2mm}  + C^{n-1} R \sum_{\underline l \in \Pi_{n-1}^{(i,j)}} \frac {R^3} {d^3_{i l_1}} \frac {R^3} {d^3_{l_1 l_2}}
         \dots \frac {R^3} {d^3_{l_{n-2} l_{n-1}}} \frac {1} {d^3_{l_{n-1} j}}  \\      
         &\lesssim  \frac {1} {d_{i j}^2}\left(CR^3 S_3 \right)^n  +\frac {R} {d_{i j}^3} \left(CR^3 S_3\right)^{n-1}\\
         &\ls\frac {1} {d_{i j}^2} \left(CR^3 S_3\right)^{n-1},
\label{eq:mor_1}
    \end{align}
    where we used \eqref{eq:sum.contraction.2}--\eqref{eq:sum.contraction.3} in the penultimate estimate. In the last estimate we used \eqref{delta} in combination with $\dmin>R$ .
    
    Similarly, regarding the second term on the right-hand side in \eqref{eq:split.nabla.psi}, we obtain, using in addition Lemma \ref{lem:decay.nabla.Q}, the estimate

    \begin{align}
        \sum_k &  \biggl| \sum_{m=1}^{n} \sum_{\substack{\underline l \in \Pi_n^{(k)} \\ l_m = j}} (-1)^n \left(Q_{l_1} \dots \nabla_{Y_j} Q_j \dots Q_{l_n}    w^{(k)}_0\right)(x)\biggr| \\
       \leq & C^n
         \sum_k  \sum_{m=1}^{n} \sum_{\substack{\underline l \in \Pi_n^{(i,k)} \\ l_m = j}} \frac {R^3} {d^2_{i l_1}} \frac {R^3} {d^3_{l_1 l_2}}
         \dots \\
         & \hspace{2cm} \times \left(\frac {R^3} {d^3_{l_{m-1} j}}  \frac {R^3} {d^4_{j l_{m+1}}} + \frac {R^3} {d^4_{l_{m-1} j}}  \frac {R^3} {d^3_{j l_{m+1}}} \right) \dots \frac {R^3} {d^3_{l_{n-1} l_{n}}} \frac {1} {d^2_{l_{n} k}} \\
         &+ C^{n-1}R \sum_k  \sum_{m=1}^{n-1} \sum_{\substack{\underline l \in \Pi_{n-1}^{(i,k)} \\ l_m = j}} \frac {R^3} {d^3_{i l_1}} \frac {R^3} {d^3_{l_1 l_2}}
         \dots \\
         & \hspace{2cm} \times \left(\frac {R^3} {d^3_{l_{m-1} j}}  \frac {R^3} {d^4_{j l_{m+1}}} + \frac {R^3} {d^4_{l_{m-1} j}}  \frac {R^3} {d^3_{j l_{m+1}}} \right) \dots \frac {R^3} {d^3_{l_{n-2} l_{n-1}}} \frac {1} {d^2_{l_{n-1} k}} \\
          \lesssim & \frac{n }{\dmin} \frac {1} {d_{i j}^2} \left(CR^3 S_3 \right)^{n-1} R^3 S_2 + \1_{n \geq 2} \frac{n }{\dmin} \frac {R} {d_{i j}^3}  \left(CR^3 S_3\right)^{n-2} R^3 S_2\\
          \le& \frac{n}{\dmin}\frac{1}{d_{ij}^2}R^2S_2(CR^3S_3)^{n-2}\bra{R^3S_3+\1_{n\ge2}}.\label{eq:mor_2}
    \end{align}

    Here, to be more precise, for  $m=1$ and $n \geq 2$ the term in the second line reads 
    \begin{align}
      \sum_k  \sum_{\substack{\underline l \in \Pi_n^{(i,k)} \\ l_m = j}}  \left(\frac {R^3} {d^2_{i j}}  \frac {R^3} {d^4_{j l_{2}}} + \frac {R^3} {d^3_{i j}}  \frac {R^3} {d^3_{j l_{2}}} \right)  \frac {R^3} {d^3_{l_2 l_3}}
         \dots  \dots \frac {R^3} {d^3_{l_{n-1} l_{n}}} \frac {1} {d^2_{l_{n} k}},
    \end{align}
    and for $m=n$ and $n \geq 2$ it reads
    \begin{align}
        \sum_k  \sum_{\substack{\underline l \in \Pi_n^{(i,k)} \\ l_n = j}} \frac {R^3} {d^2_{i l_1}} \frac {R^3} {d^3_{l_1 l_2}}
         \dots \left(\frac {R^3} {d^3_{l_{n-1} j}}  \frac {1} {d^3_{j k}} + \frac {R^3} {d^4_{l_{n-1} j}}  \frac {1} {d^2_{j k}} \right) ,
    \end{align}
    while for $n=m=1$ it reads
        \begin{align}
        \sum_{k \neq j}  \left(\frac {R^3} {d^2_{ij}}  \frac {1} {d^3_{k j}} + \frac {R^3} {d^3_{ij}}  \frac {1} {d^2_{k j}} \right),
    \end{align}
    and analogous statements hold for the terms corresponding to $m=1$ and $m=n-1$ in the third line.

    Combining \eqref{eq:mor_1} and \eqref{eq:mor_2} yields \eqref{est:nabla_Y.psi}.
\end{proof}

\begin{proof}[Proof of Proposition \ref{pro:gradient.resistance}]
    Let $\Xi_\delta
    \subset (\R^3)^N$ be the subset of particle configurations for which the assumptions of Proposition~\ref{pro:gradient.resistance} hold. We choose $\delta<C^{-1}$ so that $n(C\delta)^n$ is summable where $C$ is the constant from Lemma \ref{lem:est.nabla.psi}. Then \eqref{eq:w_n.sum.psi} and Lemma \ref{lem:est.nabla.psi} and \cite[Proposition 3.12]{Hofer18MeanField} imply that the map $Y \mapsto   w_n[Y](x)$ for $x\in B_i$ is a Cauchy sequence in $W^{1,\infty}(\Xi_\delta)$. Therefore, by \eqref{eq:nabla_X.R}, $\nabla_{Y_j} (\mR(Y)\bar g)_i$ exists and is estimated by
    \begin{align} \label{eq:operator.norm.sum}
        |\nabla_{Y_j} (\mR(Y)\bar g)_i| \lesssim & \left|\nabla_{Y_j}   w_0[Y](Y_i)\right|+\sum_{n=1}^\infty  \sum_k  |  \nabla_{Y_j} \psi^{(k)}_n[Y](Y_i) |.
    \end{align}
    The first term on the right-hand side of \eqref{eq:operator.norm.sum} satisfies the desired estimate due to Lemma \ref{lem:nabla.w_0}. 
    For the remaining part of the right-hand side of \eqref{eq:operator.norm.sum} we use \eqref{eq:infty_1}. Then, Lemmas \ref{lem:est.nabla.psi} and \ref{le:sums} yield
    \begin{align}
        \sum_{n=1}^\infty \left\| \left(  \sum_k \nabla_{Y_j} \psi^{(k)}_n[Y](Y_i) \right)_{ij}\right\| &\lesssim  S_2  n\left(C\delta\right)^{n-2}\bra{\delta+R^2 S_2(\1_{n\ge 2}+\delta)}\\
        &\ls S_2\bra{1+R^2S_2}.     
    \end{align}
    Inserting these estimates in \eqref{eq:operator.norm.sum} concludes the proof.
\end{proof}

%% file: 5.MinDist.tex
\section{Control of the minimal distance}\label{sec:MinDist}

In this Section, we prove Lemma \ref{lem:dmin_control} and Lemma \ref{lem:VF.infty}. Although some of the statements may be stated for general particle configurations $Y$ that satisfy certain assumptions, we will use the notation $X, V$, and $F$, while still stating all assumptions that are needed for $X$.

We start with the following lemma that yields a convenient representation of the forces $F_i$. 
In particular it singles out the self-interaction term $6\pi RV_i$, which is the force that a single particle exerts on the fluid Stokes flow.

\begin{lem} \label{le:force.representation}
Let $r>0$, and $d \geq 4 r$.
Let $A = B_d(0) \setminus B_{d/2}(0)$ and $\tilde A = B_d(0) \setminus B_{r}(0)$.
Then, there exists a weight $\omega \colon A \to \R^{3\times3}$ with
\begin{align} \label{omega}
    \fint_A \omega \dd x = \Id, \qquad    \|\omega\|_\infty \leq C 
\end{align}
for some universal constant $C$ and such that the following holds true.
For all $(w,p) \in H^1(\tilde A) \times L^2(\tilde A)$ with  
\begin{align}\label{eq:Stokes_A}
    - \Delta w + \nabla p = 0, ~ \dv w = 0 \quad \text{in} ~ \tilde A
\end{align}
we have
\begin{align}\label{eq:br}
    -\int_{\partial B_r(0)} \sigma[w] n \dd \mathcal H^2 = 6 \pi r\left( \fint_{\partial B_r(0)} w \dd \mathcal H^2 - \fint_A \omega w \dd x\right).
\end{align}
\end{lem}
\begin{rem} \label{rem:force.repr}
A direct consequence is that, as long as $\dmin \geq 8 R$,
we have
\begin{align} \label{force.repr}
    F_i = 6 \pi R\left(V_i-\fint_{A_i} \omega_i u \dd x\right) \eqqcolon 6 \pi R(V_i-(u)_i).
\end{align}
where 
\begin{align}
    A_i &\coloneqq B_{\dmin/2}(X_i) \setminus B_{\dmin/4}(X_i), \\
    \omega_i(x) &\coloneqq \omega(x - X_i).
\end{align}
\end{rem}
\begin{proof}
    For $W \in \R^3$ let
    \begin{align}
        \Psi[W](x) &\coloneqq \left(\frac {3r}{2 |x|} - \frac{r^3}{2 |x|^3}  \right) \frac{W \times x}{2}, \\
        \mathcal U[W] &\coloneqq  \nabla \times \Psi[W], \\
        \mathcal P[W](x) &\coloneqq \nabla \left(\frac{1}{|x|} \right) r W.
    \end{align}
    A straightforward computation shows (see e.g. \cite[Section 2]{Hillairet21}) that
   \begin{align}
        \mathcal U[W] &= W \quad \text{in} ~ B_r(0), \\
        - \Delta \mathcal U[W] + \nabla \mathcal P[W] &= 0 \quad \text{in} ~ \R^3 \setminus B_r(0), \\
        -\sigma[\mathcal U [W]] n &= \frac{6 \pi r W}{|\partial B_r0)|} \quad \text{on} ~ \partial B_r(0).
    \end{align}
    Let $\eta \in C_c^\infty(B_d(0))$ be a cut-off function with
    $\eta = 1$ in $B_r(0)$ and,    for $k = 0,1,2,3$, 
    \begin{align} \label{est:eta}
        \|\nabla^k \eta\|_{\infty} \leq \frac{C}{d^k}.
    \end{align}
    We test \eqref{eq:Stokes_A} with $\nabla\times (\eta \Psi[W])$ (which is solenoidal) and after some integration by parts we deduce for $w$ as in the statement of the lemma
    \begin{multline}
        -W \cdot \int_{\partial B_r(0)} \sigma[w] n  \dd \mathcal H^2 \\
        = 6 \pi r W \cdot \fint_{\partial B_r(0)} w \dd \mathcal H^2 - \int_A w  \cdot \dv \sigma[ \nabla \times (\eta  \Psi[W]),\eta \mathcal P[W]] \dd x.
    \end{multline}
    Note that we explicitly denoted the dependence on the cut-off pressure here, since the test function fails to satisfy the homogeneous Stokes equation in $A$. Thus,  defining 
    \begin{align}
        \omega = \frac{|A|}{6 \pi r} \sum_i e_i \otimes \dv \sigma[ \nabla \times (\eta  \Psi[e_i]),\eta \mathcal P[e_i]]
    \end{align}
    yields \eqref{eq:br} because of linearity in $W$.  Moreover, the identity in \eqref{omega} follows immediately from integration by parts and the estimate in \eqref{omega} follows from  \eqref{est:eta} and the explicit form of $\mathcal U, \mathcal P$.
\end{proof}

%----------------------------------------------------------------------

Considering the particle configuration $X$ with associated velocity field $u$, by \eqref{omega}, the term
\begin{align}
    (u)_i = \fint_{A_i}\omega_i u\dd x,
\end{align}
in \eqref{force.repr} can be viewed as a suitable average  fluid velocity 
in the vicinity of the particle. Then, \eqref{force.repr} resembles the classical Stokes law $F_i = 6 \pi R(V_i-(u)_i)$. By the choice $d \sim \dmin \gg R$ this average fluid velocity $(u)_i$ represents the ``mean-field'' velocity  at $X_i$ that is generated by the other particles, since on this scale  the contribution of the particle $i$ itself becomes negligible.

The following lemma provides helpful estimates on $(u)_i$.
\begin{lem} \label{lem:(u)_i}
Assume that $\dmin \geq 8 R$ (at some time $t \geq 0)$. Then,
\begin{enumerate} 
\item[(i)] For all $1\leq i \leq N$
\begin{align} \label{est:omega.u}
    \abs{ \fint_{A_i}\omega_i u\dd x}\ls \sum_j \frac{\abs{F_j}}{d_{ij}}+R^3 S_2 S_4^{1/2} |F|_2 .
\end{align}

\item[(ii)] For all $i \neq j$
\begin{align} 
    \abs{ \fint_{A_i}\omega_i u\dd x-\fint_{A_j}\omega_j u\dd x}&\ls d_{ij}\bra{\sum_k \abs{F_k}\bra{\frac{1}{d_{ik}^2}+\frac 1{d_{jk}^2}}+R^3 S_2 S_6^{1/2} \abs{F}_2} \\
    & \lesssim d_{ij}\bra{|F|_\infty S_2+R^3  S_2 S_6^{1/2} \abs{F}_2}.
\label{est:omega.u.diff}
\end{align}
\end{enumerate}
\end{lem}

For the proof we will rely on the following lemma adapted from \cite[Proposition 3.9]{Hofer18MeanField}.
\begin{lem}\label{l:tilde_u}
    Assume that $\dmin \geq 4 R$.
    Let $w_0$ be the solution to 
    \begin{align}\label{eq:tilde_u}
        -\Delta w_0 + \nabla p = \sum_i \delta_{\partial B_i} F_i, ~ \dv w_0 = 0 \quad \text{in} ~ \R^3.
    \end{align}
    Then, there exists a universal constant $C<\infty$ such that
    \begin{align} \label{est.u-utilde}
        \frac 1 C \|w_0 - u\|^2_{\dot H^1(\R^3)} \leq  \| \nabla w_0\|^2_{L^2(\cup_i B_i)} \leq  C R^3  |F|_2^2 S_2^2.
    \end{align}
\end{lem}
\begin{proof}
    By standard arguments, $v = w_0 - u$ is the minimizer of $\|\varphi\|_{\dot H^1(\R^3)}$ among all $\varphi \in \dot H^1(\R^3)$ that are divergence free and satisfy $\nabla \varphi =  \nabla w_0$ in $\cup_i B_i$. By using Lemma \ref{lem:extension} below, there exists a $\varphi \in \dot H^1(\R^3)$ with
    \begin{align}
        \|\varphi\|_{\dot H^1(\R^3)}^2 \lesssim  \|\nabla w_0\|^2_{L^2(\cup_i B_i)} \lesssim R^3 \sum_i \|\nabla w_0\|^2_{L^\infty(B_i)}.
    \end{align}
    We use the explicit form of $w_0$ analogous to \eqref{eq:w_0.explicit} together with \eqref{eq:sum.contraction.1} to conclude
    \begin{align}
        \|w_0 - u\|^2_{\dot H^1(\R^3)} &\lesssim  R^3 \sum_i \left( \sum_{j \neq i} \frac{|F_{j}|}{d_{ij}^2} \right)^2 = R^3 \sum_i \sum_{j \neq i} \sum_{k \neq i} \frac{|F_{j}|}{d_{ij}^2} \frac{|F_{k}|}{d_{ik}^2}  \\
        &\lesssim R^3 S_2  \sum_{k} \sum_{j} \frac {|F_j|}{d_{jk}}\frac{|F_k|}{d_{jk}}  
        \lesssim R^3 |F|_2^2 S_2^2,
    \end{align}
    where we used Young's inequality in the last inequality.
\end{proof}

\begin{lem}[{\cite[Lemma 3.5]{Hofer18MeanField}}] \label{lem:extension}
	Let $x \in \R^3$, $r> 0$ and assume that $\varphi \in H^1(B_r(x))$ is divergence free.
 Then, there exists a divergence free function  $\psi \in H^1_{0}(B_{2 r}(x))$
	such that $ \nabla  \psi = \nabla \varphi$ in $B_r(x)$ and
	\begin{align}
		\|\nabla \psi\|_{L^2(B_{2 r}(x))} \leq C \|\nabla \varphi \|_{L^2(B_r(x))},
	\end{align}
 for some universal constant $C$.

 Moreover, if in addition $\int_{\partial B_r(x)} \varphi = 0$, then $\psi$ can be chosen to satisfy in addition $\psi = \varphi$ in $B_r(x)$.
\end{lem}

\begin{proof}[Proof of Lemma \ref{lem:(u)_i}]
\textbf{(i)}: For given $W\in \R^3$ consider the solution $\varphi$ to
\begin{align}
    -\Delta \varphi+\nabla p= \omega_i^T W \frac{1}{\abs{A_i}}\1_{A_i}, ~ \dv \varphi=0 \quad \text{in} ~\R^3.
\end{align}
We compute
\begin{align}
\begin{aligned}
    W\cdot\fint_{A_i}\omega_i u\dd x&=\fint_{A_i}u\cdot \omega_i^T W \dd x
    =\int_{\R^3} u\cdot (-\Delta \varphi+\nabla p)\dd x\\
    &= 2 \int_{\R^3} \nabla u\cdot e\varphi\dd x
    = 2 \int_{\R^3} eu\cdot\nabla\varphi\dd x
    \\
    &=-\sum_j \int_{\partial B_j} \varphi\cdot (\sigma[u]n)\dd \mathcal{H}^2\\
    &=-\sum_j \int_{\partial B_j} (\varphi-\varphi_j)\cdot (\sigma[u]n)\dd \mathcal{H}^2+\sum_j \varphi_j\cdot F_j ,
    \end{aligned}
    \label{eq:comp_test}
\end{align}
where $\varphi_j\coloneqq\fint_{\partial B_j} \varphi\dd \mathcal{H}^2$.
Using that $\varphi= \Phi\ast \omega_i^T W \frac{1}{\abs{A_i}}\1_{A_i}$, that $\dist(A_i,B_j)\gs d_{ij}$, the second property from \eqref{omega}, and the decay of $\Phi$, we have 
\begin{align}
    \abs{\varphi_j}\ls \frac{\abs{W}}{d_{ij}}.
\end{align}
Statement (i) follows from this and 
\begin{align} \label{eq:double_diff}
     \left|\sum_j \int_{\partial B_j} (\varphi-\varphi_j)\cdot (\sigma[u]n)\dd \mathcal{H}^2 \right|\ls R^3 S_2 S_4^{1/2}\abs{F}_2\abs{W},
\end{align}
which we prove in the following. Here we capitalize on the fact that the mean of $\varphi-\varphi_j$ over $\partial B_j$ vanishes  whence we can renormalize $\sigma[u]n$ by a constant. We start by noting that since $e u = 0$ in $B_j$ for all $j$ we have
\begin{align} \label{dipole.0}
    -\sum_j \int_{\partial B_j} (\varphi-\varphi_j)\cdot (\sigma[u]n)\dd \mathcal{H}^2= 2 \int_{\R^3 \setminus \cup_j B_j} \nabla \psi \cdot e   u \dd x = 2 \int_{\R^3} \nabla \psi \cdot e  u \dd x
\end{align}
for all  divergence free functions $\psi \in \dot H^1(\R^3)$ such that for all $1 \leq j \leq N$, $\psi = \varphi - \varphi_j$ in $B_j$. Moreover, using the weak formulation for $w_0$ defined as in \eqref{eq:tilde_u}, we have
\begin{align}
    \int_{\R^3} \nabla \psi \cdot e   w_0 \dd x = 0.
\end{align}
Hence we deduce
\begin{align} \label{dipole.3}
    \left|\sum_j \int_{\partial B_j} (\varphi-\varphi_j)\cdot (\sigma[u]n)\dd \mathcal{H}^2\right|  &= \left|2 \int_{\R^3} \nabla \psi \cdot e  (u - w_0) \dd x\right|  \\
    &\ls \norm{\nabla(u-w_0)}_2\norm{\nabla \psi}_2. \qquad 
\end{align}
By Lemma \ref{lem:extension} and using that   $\nabla \varphi$ decays like $ \nabla \Phi$ we can choose $\psi$  such that
\begin{align}
    \norm{\nabla \psi}_{L^2(\R^3)}^2\ls \sum_j \norm{\nabla\varphi}_{L^2(B_j)}^2\ls \sum_j R^3\frac{\abs{W}^2}{d_{ij}^4}\ls  R^3  S_4 \abs{W}^2.
\end{align}
Inserting this in \eqref{dipole.3} and using \eqref{est.u-utilde} yields \eqref{eq:double_diff}.

\medskip

\noindent\textbf{(ii)}:  For $W\in \R^3$ consider the solution $\varphi$ to
\begin{align}
    -\Delta \varphi+\nabla p= \left(\omega_i^T \frac{1}{\abs{A_i}}\1_{A_i} -\omega_j^T \frac{1}{\abs{A_j}}\1_{A_j}\right)W , ~ \dv \varphi=0 \quad \text{in} ~ \R^3.
\end{align}
Since $\varphi=\Phi\ast \bra{\omega_i^T \frac{1}{\abs{A_i}}\1_{A_i} -\omega_j^T \frac{1}{\abs{A_j}}\1_{A_j}}W$, we have
\begin{align}\label{eq:phi_k}
    \abs{\varphi_k}\ls \abs{W}d_{ij}\bra{\frac 1{d_{ik}^2}+\frac 1{d_{jk}^2}}.
\end{align}
As in \eqref{eq:comp_test} we compute 
\begin{multline}
    W\cdot\bra{\fint_{A_i}\omega_i u\dd x-\fint_{A_j}\omega_i u\dd x} \\
    =\sum_k \int_{\partial B_k} (\varphi-\varphi_k)\cdot (\sigma[u]n)\dd \mathcal{H}^2+\sum_k \varphi_k\cdot F_j.
\end{multline}
Using \eqref{eq:phi_k} for the second term it remains to show that
\begin{align} \label{dipole.ii.1}
   \left|\sum_k \int_{\partial B_k} (\varphi-\varphi_k)\cdot (\sigma[u]n)\dd \mathcal{H}^2 \right|\ls R^{3}d_{ij}S_2 S_6^{1/2} \abs{F}_2 \abs{W}.
\end{align}
As in the proof of (i), using \eqref{est.u-utilde}, we can estimate 
\begin{align} \label{dipole.ii.2}
    \left|\sum_k \int_{\partial B_k} (\varphi-\varphi_k)\cdot (\sigma[u]n)\dd \mathcal{H}^2 \right| \lesssim  R^{3/2}  S_2 |F|_2  \norm{\nabla \psi}_2
\end{align}
for any divergence free function $\psi \in \dot H^1(\R^3)$ with $\psi = \varphi - \varphi_k$ in $B_k$.
Relying again on Lemma~\ref{lem:extension} and the decay of $\nabla \varphi$ like $\nabla^2 \Phi$ we find such a $\psi$ with
\begin{align}
    \norm{\nabla \psi}_2^2\ls \sum_j \norm{\nabla\varphi}_{L^2(B_j)}^2\ls \sum_k R^3d_{ij}^2\abs{W}^2\bra{\frac{1}{d_{ik}^6}+\frac{1}{d_{jk}^6}}\ls R^3d_{ij}^2 S_6\abs{W}^2.
\end{align}
Inserting this estimate in \eqref{dipole.ii.2} yields \eqref{dipole.ii.1}.
\end{proof}

\begin{proof}[Proof of Lemma \ref{lem:dmin_control}]
Let $i,j$ be given. By definition of $d_{ij}$, we have
\begin{align}\label{eq:basic_dij_est}
   d_{ij}(t)\ge d_{ij}(0)-\abs{\int_0^t V_i(s)-V_j(s)\dd s}.
\end{align}
To estimate the second term we consider the differential equality \eqref{eq:acceleration} which we combine with \eqref{force.repr} to obtain
\begin{align}
   \dot V_i=\lambda_N\left(g-NF_i\right)
   =\lambda_N\left(g+6\pi\gamma_N\left(\fint_{A_i}\omega_i u\dd x- V_i\right)\right).
\end{align}
Integration yields
\begin{align}\label{eq:V_integrated}
   V_i(s)=e^{-6\pi\gamma_N\lambda_N t}V_i(0)+\int_0^s e^{6\pi\gamma_N \lambda_N(\tau-s)}\lambda_N\left(g+6\pi\gamma_N(u)_i(\tau)\right)\dd \tau. \qquad 
\end{align}
We use \eqref{eq:V_integrated} and \eqref{ass:V.Lipschitz} to estimate
\begin{align} \label{est:int.V_i-V_j}
   & \abs{\int_0^t V_i(s)-V_j(s)\dd s}\\
   & \hspace{2mm}\le \frac{1}{6\pi\gamma_N\lambda_N}\abs{V_i-V_j}(0)+6\pi\gamma_N\lambda_N\int_0^t\int_0^s e^{6\pi\gamma_N\lambda_N(\tau-s)}\abs{(u)_i-(u)_j}(\tau)\dd \tau\dd s\\
    &\hspace{2mm} \le \frac 12 d_{ij}(0)+\int_0^t \abs{(u)_i-(u)_j}(\tau)\dd \tau.
\end{align}
By \eqref{est:omega.u.diff} we conclude \eqref{est:dmin.0}. 
\end{proof}

The representation formula from Lemma~\ref{le:force.representation} allows to control the $l_\infty$-norm of $V$ in terms of the $l_2$-norm (which is in turn controlled by the modulated energy argument Proposition~\ref{prop:energy}). 
Due to Lemma~\ref{le:force.representation}, $V_i$ has a damping effect on itself while the influence of the other particles through the forces $F_k$ for $k\neq i$ is modulated with a linear decay in the distance (Lemma~\ref{lem:(u)_i}), which allows us to estimate the effect of the background velocity with Cauchy-Schwarz. 

\begin{proof}[Proof of Lemma \ref{lem:VF.infty}]
We start with \eqref{eq:acceleration} and use \eqref{force.repr} to write for any $i=1,\dots,N$ that
\begin{align}\label{eq:Vidot}
    \dot{V_i}=\lambda_N(-6\pi\gamma_N V_i+g+\gamma_N H_i),
\end{align}
where
\begin{align}
    H_i\coloneqq 6\pi\fint_{A_i}\omega_i u\dd x.
\end{align}
This implies
\begin{align}
    |V_i|(t) \lesssim |V^i_0|e^{-\lambda_N \gamma_N t} + \gamma_N^{-1} + \sup_{s \leq t}|H_i(t)|.  
\end{align}
By \eqref{est:omega.u} we have
\begin{align} \label{est:R_i}
    |H_i| &\lesssim |F|_2 \left( S_2^{\frac 1 2}  +R^{3}S_2 S_4^{\frac 1 2}\right),
\end{align}
which finishes the estimate of $V_i$ in \eqref{eq:F_infty}.
    
For the estimate of $F_i$, we use again Lemma \ref{le:force.representation}, the estimate of $V_i$ in \eqref{eq:F_infty} that we have just shown, and \eqref{est:R_i} to deduce
\begin{align}
    |F_i(t)| &\lesssim R(|V_i(t)| + |H_i(t)|) \\
    &\lesssim  R\left(\frac 1 \gamma_N  + |V_i^0| e^{-\gamma_N \lambda_N t} + |F|_{2,t} \left( S_{2,t}^{\frac 1 2}  +R^{3}S_{2,t} S_{4,t}^{\frac 1 2}\right)\right)
\end{align}
as claimed.
\end{proof}

%% file: Appendix.tex
 \section{Estimates of  sums of inverse particle distances}\label{sec:sums}

\begin{lem} \label{le:sums}
The following estimate holds: 
\begin{align} \label{eq:sum.3}
	S_3 = \sup_i  \sum_{j} \frac 1 {d_{ij}^3} \lesssim \frac{\log N}{\dmin^{3}}.
\end{align}
In addition, for $n \geq 4$, we have
\begin{align} \label{eq:sum.4+}
	S_n = \sup_i  \sum_{j} \frac 1 {d_{ij}^n} \lesssim \frac{1}{\dmin^{n}}.
\end{align}
Moreover, for all $0\le j,k\le N$, the following expressions can be estimated
\begin{align}%\label{eq:sum.contraction.0}
    %\sum_{i} \frac{1}{d^2_{ij} d_{ik}^2} &\lesssim  \frac{1}{\dmin^{3}} \frac 1 {d_{jk}}, \\
        \label{eq:sum.contraction.1}
        \sum_{i} \frac{1}{d^2_{ij} d_{ik}^2} &\lesssim  \frac {S_2} {d_{jk}^2}, \\
        \label{eq:sum.contraction.2}
        \sum_{i} \frac{1}{d^2_{ij} d_{ik}^3} &\lesssim   \frac { S_3} {d_{jk}^2},  \\ \label{eq:sum.contraction.3}
        \sum_{i} \frac{1}{d^3_{ij} d_{ik}^3}& \lesssim \frac { S_3}  {d_{jk}^3}.
    \end{align}
\end{lem}
\begin{proof} 
The proof of \eqref{eq:sum.3} and \eqref{eq:sum.4+} \noeqref{eq:sum.3} can be found in \cite[Lemma 4.8]{NiethammerSchubert19}. 
It remains to prove \eqref{eq:sum.contraction.1}--\eqref{eq:sum.contraction.3}. Due to the triangle inequality $d_{jk} \leq d_{ij} + d_{ik}$ (which continues to hold despite our convention $d_{ll} = \dmin$), in the case that $d_{ij} \leq d_{ik}$ it holds that $d_{ik} \geq d_{jk}/2$ while in the other case $d_{ik} < d_{ij}$ we have $d_{ij} \geq d_{jk}/2$. Thus,
\begin{align}
     \frac{1}{d^2_{ij} d_{ik}^2} &\lesssim \frac{1}{d_{jk}^2} \left(\frac{1}{d^2_{ij}} +  \frac{1}{d^2_{ik}} \right), \\
     \frac{1}{d^2_{ij} d_{ik}^3} &\lesssim \frac{1}{d_{jk}^2} \left(\frac{1}{d^3_{ij}} +  \frac{1}{d^3_{ik}} \right), \\
     \frac{1}{d^3_{ij} d_{ik}^3} &\lesssim \frac{1}{d_{jk}^3} \left(\frac{1}{d^3_{ij}} +  \frac{1}{d^3_{ik}} \right).
\end{align}
Taking the sum in $i$ immediately yields \eqref{eq:sum.contraction.1}--\eqref{eq:sum.contraction.3}.
\end{proof}

\rs{\section{Proof of an Hauray type result for second-order binary systems with asymptotically vanishing inertia}
\label{sec:binary}

\begin{proof}[Proof of Theorem~\ref{th:binary}]
    We will argue for $p<\infty$ here, the arguments for $p=\infty$ are analogous. We set $\tilde V_i=\frac 1N\sum_{j\neq i} K(X_i-X_j)$.
    
    \noindent \textbf{Step 1:} \emph{modulated energy argument.} We define the modulated energy
    \begin{align}
        E_p\coloneqq \frac 1{pN}|V-\tilde V|_p^p.
    \end{align}
    Computing the time derivative yields
    \begin{align}
        \frac {\dd E_p} {\dd t}=& \frac 1N \sum_{i}\Bigg(|V_i-\tilde V_i|^{p-2}\\
        &\left.\times(V_i-\tilde V_i)\cdot \bra{\lambda_N(-V_i+\tilde V_i)-\frac 1N \sum_{j\neq i}(V_i-V_j)\cdot\nabla K(X_i-X_j)}\right)\\
        =&-p\lambda_N E_p \\
        &-\frac 1{N^2} \sum_i \left(|V_i-\tilde V_i|^{p-2}
        (V_i-\tilde V_i)\cdot \sum_{j\neq i}(V_i-V_j)\cdot \nabla K(X_i-X_j)\right).
    \end{align}
    For the second right-hand side term, we split $V_i - V_j = (V_i-\tilde V_i)+(\tilde V_i-\tilde V_j)+(\tilde V_j-V_j)$. Then, using the decay of $K$ from \eqref{eq:C_alpha}, firstly
    \begin{align}
        &\sum_i |V_i-\tilde V_i|^{p-2}(V_i-\tilde V_i)\cdot \sum_{j\neq i}(V_i-\tilde V_i)\cdot \nabla K(X_i-X_j)\\
        &\quad\ls S_{\alpha+1}\sum_i |V_i-\tilde V_i|^p.
    \end{align}
    Secondly, using that
    \begin{align}
        |\tilde V_i-\tilde V_j| &\ls \frac 1Nd_{ij}^{-\alpha}+\frac 1N\sum_{k\not\in \{i,j\}}|K(X_i-X_k)-K(X_j-X_k)| \\
        & \ls \frac 1N d_{ij}^{-\alpha}+\frac 1Nd_{ij}\sum_{k\not\in \{i,j\}}(d_{ik}^{-\alpha-1}+d_{j_k}^{-\alpha-1}) \\
        & \ls d_{ij} \frac 1 N S_{\alpha +1}
    \end{align}
    we find 
    \begin{align}
        &\sum_i |V_i-\tilde V_i|^{p-2}(V_i-\tilde V_i)\cdot \sum_{j\neq i}(\tilde V_i-\tilde V_j)\cdot \nabla K(X_i-X_j)\\
        &\quad\ls \sum_i |V_i-\tilde V_i|^{p-1} d_{ij} \frac 1 N S_{\alpha +1} \sum_{j\neq i} |\nabla K(X_i-X_j)| \\
        & \quad \ls  \frac 1 N S_{\alpha +1} S_\alpha  |V-\tilde V|_{p-1}^{p-1} \leq \frac 1 N S_{\alpha +1} S_\alpha  N^{1/p}|V-\tilde V|_{p}^{p-1}.
    \end{align}
        Finally, using Young's inequality
        \begin{align}
        &\sum_i |V_i-\tilde V_i|^{p-2}(V_i-\tilde V_i)\cdot \sum_{j\neq i}(\tilde V_j-V_j)\cdot \nabla K(X_i-X_j)\\
        & \quad \ls \sum_i \sum_{j\neq i} (|V_i - \tilde V_i|^p + |V_j - \tilde V_j|^p) |\nabla K(X_i-X_j)| \\ 
        &\quad\ls |V-\tilde V|_p^{p}S_{\alpha+1}
    \end{align}
    This yields
    \begin{align}
        \frac {\dd E_p} {\dd t}&\le -p\lambda_N E_p+C\frac 1N S_{\alpha+1}pE_p+C(1+\frac 1N S_{\alpha+1})\frac 1N S_\alpha (pE_p)^{(p-1)/p}\\
        &\le (-p\lambda_N+C\frac 1N S_{\alpha+1}) E_p+C(1+\frac 1N S_{\alpha+1})\frac 1N S_\alpha (pE_p)^{(p-1)/p}\label{eq:energy_bin}
    \end{align}
    \noindent \textbf{Step 2:} \emph{Setup of the buckling argument.}
    As before, we denote $\eta(t)=\W_p(\sigma_N(t),\sigma(t))$. We have $\|\sigma(t)\|_q\le\|\sigma^0\|_q\le C$ for all $t>0$. Lemma~\ref{lem:sums.Wasserstein} with Assumption \eqref{ass:bin1} yields (cf. Step 3 in the proof of Theorem \ref{th:Hauray})
    \begin{align}\label{eq:Salpha_bin}
        \frac 1N (S_\alpha+S_{\alpha+1})\ls 1+N^{-(\alpha+1)/d}\dmin^{-(\alpha+1)}\eta^{\frac{(d-\alpha-1)p}{d+pq'}}.
    \end{align}
    We consider $T_\ast=T_\ast(N)$ the largest time such that 
    \begin{align}
        \frac 1N (S_\alpha+S_{\alpha+1})\le C_1
    \end{align}
    for some constant $C_1$ that will be chosen later.\\
    
    \noindent \textbf{Step 3:} \emph{Estimate of $\eta$.} We write 
    \begin{align}\label{eq:etabin}
	\bar \eta(t)\coloneqq \sup_{0\le s\le t} \bra{\int_{\R^3} \abs{T_s(x)-x}^2 \sigma(s,x)\dd x}^{1/2},
\end{align}
where $T_s$ is the canonical transport map constructed via characteristics from an optimal map at $t=0$. (Cf. Step 1 of the proof of Theorem~\ref{th:Hauray}. The only difference here is that $u_N(X_i) = V_i$). As usual $\eta\le \bar \eta, \eta(0)=\bar\eta(0)$ and
\begin{align}
    \bar \eta(t) - \bar \eta(0) \leq& \int_0^t \left( \int_{\R^3} \left|V(s,T(x)) - (K\ast\sigma)(s,x)\right|^p \dd \sigma(s,x)  \right)^{\frac 1 p} \dd s\\
    \leq& \int_0^t \left( \int_{\R^3} \left|V(s,T(x)) - \tilde V(s,T(x))\right|^p\dd \sigma(s,x)  \right)^{\frac 1 p}\dd s\\
    &+\int_0^t\left( \int_{\R^3} \left|\tilde V(s,T(x))-(K\ast\sigma)(s,x)\right|^p \dd \sigma(s,x)  \right)^{\frac 1 p} \dd s.\label{eq:split.bin}
\end{align}
with the slightly abusive notation $V(s,x) = V_i$ and $\tilde V(s,x) = \tilde V_i$ for $x = X_i(s)$. Note that the second right-hand side term can be estimated exactly as in Step 1 of the proof of Theorem~\ref{th:Hauray} starting at \eqref{eq:f_est} and thus 
\begin{align}
    &\int_0^t\left( \int_{\R^3} \left|\tilde V(s,T(x))-(K\ast\sigma)(s,x)\right|^p \dd \sigma(s,x)  \right)^{\frac 1 p} \dd s\\
    &\ls \int_0^t (1+\frac 1N S_{\alpha+1}(s))\bar \eta(s)\dd s.\label{eq:split.bin1}
\end{align}
The first term satisfies
\begin{align}
    &\int_0^t \left( \int_{\R^3} \left|V(s,T(x)) - \tilde V(s,T(x))\right|^p\dd \sigma(s,x)  \right)^{\frac 1 p}\dd s\\
    &=\int_0^t N^{-1/p}|V-\tilde V|_p(s)\dd s.\label{eq:split.bin2}
\end{align}
Estimate \eqref{eq:energy_bin} yields for $N$ (and thus by \eqref{ass:bin1} $\lambda_N$) large enough
\begin{align}\label{eq:V_bin}
    N^{-1/p}|V-\tilde V|_p(s)\ls e^{-\frac{\lambda_N}{C}t}|V-\tilde V|_p(0)+C\min\{\lambda_N^{-1},s\}. 
\end{align}
Inserting \eqref{eq:V_bin} in \eqref{eq:split.bin2} and using this with \eqref{eq:split.bin1} in \eqref{eq:split.bin} yields, with computations analogous to \eqref{int.E} and following and with the help of \eqref{ass:bin3} to bound $|V-\tilde V|_p(0)$,
\begin{align}\label{eq:eta_bin}
    \eta(t)\le\bar\eta(t)\ls\bra{\eta(0)+\frac 1{\lambda_N}}e^{Ct}.
\end{align}
\noindent \textbf{Step 4:} \emph{Estimate of $\dmin$.} 
We use $\dot V_i=\lambda_N(-V_i+\tilde V_i)$ to write
\begin{align}
    V_i(s)=e^{-\lambda_Nt}V_i(0)+\int_0^s e^{\lambda_N(\tau-s)}\lambda_N\tilde V_i(\tau)\dd \tau.
\end{align}
Thus, using \eqref{ass:bin2}, we have 
\begin{align}
    d_{ij}(t)&\ge d_{ij}(0)-\left|\int_0^t V_i(s)-V_j(s)\dd s \right|\\
    &\ge d_{ij}(0)-\frac{1}{\lambda_N}|V_i-V_j|(0)-\lambda_N\int_0^t\int_0^s e^{\lambda_N(\tau-s)}|\tilde V_i-\tilde V_j|(\tau)\dd \tau\dd s\\
    &\ge  d_{ij}(0)-\frac{1}{2}d_{ij}(0)-\lambda_N\int_0^t |\tilde V_i-\tilde V_j|(\tau)\dd \tau
\end{align}
The same computation as in Step 2 of the proof of Theorem~\ref{th:Hauray} yields
\begin{align}
    d_{ij}(t)&\ge \frac 12 d_{ij}(0)-\lambda_N\int_0^t \frac 1N S_{\alpha+1}(\tau)\dmin(\tau)\dd \tau
\end{align}
and thus
\begin{align}\label{eq:dmin_bin}
    \dmin(t)\ge \frac 12\dmin(0)e^{-C(t+1)}.
\end{align}
\noindent \textbf{Step 5:} \emph{Conclusion.} Combining \eqref{eq:Salpha_bin}, \eqref{eq:eta_bin} and \eqref{eq:dmin_bin} with \eqref{ass:bin1} shows  that $T_\ast(N)\to \infty$ as $N\to \infty$ provided $C_1$ is chosen large enough. The desired estimate is then \eqref{eq:eta_bin}.
\end{proof}}

\rs{
\section{Adaptations for rotating particles}\label{sec:rot}

We show how to adapt the proof of the main result to include particle rotations. The main changes concern Propositions \ref{prop:energy}, \ref{pro:resistance.estimates} and \ref{pro:gradient.resistance}. We give modified versions of these Propositions and their proofs below. The rest of the proof of the main theorem generalizes straightforwardly to the case with rotations.

For a given configuration of particle positions, $Y\in (\R^3)^N$ with $\dmin>2R$ and given translational and angular velocities $(W,Z)\in (\R^3)^{N} \times (\R^3)^{N}$, let $w\in \dot H^1(\R^3)$ be the solution to 
\begin{align}
\left\{\begin{array}{rl}
        - \Delta w + \nabla p = 0, ~ \dv  w &=0 \quad \text{in} ~ \R^3 \setminus \bigcup_i B_i, \\
	    w(x) &= W_i + Z_i \times (x-Y_i)\quad \text{in} ~ B_i.
	\end{array}\right.\label{eq:fru.rot}
\end{align}
We denote the corresponding forces and torques by
\begin{align}
	G_i\coloneqq -\int_{\partial B_i} \sigma[w] n \dd  \mathcal{H}^2, \\
 H_i \coloneqq -\int_{\partial B_i} (x - Y_i) \times \sigma[w] n   \dd  \mathcal{H}^2.
\end{align}
 The resistance matrix $\mR(Y)$ is then the linear operator that maps the velocities $(W,Z)$ to the forces and torques $(G,H)$.
The operator $\mR(Y)$ is still symmetric and positive definite, which is a consequence of the fact that for solutions $w_1,w_2 \in \dot H^1(\R^3)$ of \eqref{eq:fru.rot} corresponding to $(W_1,Z_1),(W_2,Z_2)\in (\R^3)^{N} \times (\R^3)^N$, respectively, we have by an application of the divergence theorem, 
\begin{align}\label{eq:R_bilin.rot}
\begin{aligned}
	(W_2,Z_2) \cdot\mR(Y)(W_1,Z_1)&=-\sum_{i}\int_{\partial B_i}w_2\cdot \sigma[w_1]n\dd \cH^2 \\
 &=\int_{\R^3\setminus \cup_i B_i}\nabla w_2\cdot \sigma[w_1]\dd x \\
    &=2\int_{\R^3\setminus \cup_i B_i}\nabla w_2\cdot e w_1\dd x \\
    &=2\int_{\R^3}\nabla w_2\cdot e w_1\dd x=\int_{\R^3}\nabla w_2\cdot \nabla w_1 \dd x.
    \end{aligned}
\end{align}

Let 
\begin{align}
    L &:= \frac 2 5 R^2, \qquad
    \m M := \begin{pmatrix}
        \Id & 0 \\
        0 & L^{-1/2} \Id
    \end{pmatrix}
\end{align}
where $\m M \in \R^{6N \times 6N}$.\footnote{Note that $L$ is the ratio between the moment of inertia and the mass of a ball of radius $R$ with constant mass density.}
We then denote
\begin{align}
     \tilde \mR &:= \m M \mR \m M, \\
    C_{\tilde \mR(Y)}&:= \max_{|A|_{2}=1} A \cdot \tilde \mR^{-1}(Y) A, \\
    c_{\tilde \mR(Y)}&:=\min_{|A|_{2}=1} A \cdot \tilde \mR(Y) A. \label{tilde.c}
\end{align}
Accordingly, we denote
\begin{align}
    c_{\tilde \mR,T}\coloneqq \inf_{0\le t\le T}c_{\tilde \mR(X(t))}.   
\end{align}

We recall the definition \eqref{eq:V_tilde.rot} of the velocities
\begin{align}\label{eq:def_tildeV.rot}
	(\tilde V(t),\tilde \Omega(t)) \coloneqq \mR^{-1}(X(t)) \frac{(\bar g,0)}{N}.
\end{align}

We remark first that \eqref{tilde.c} and \eqref{eq:def_tildeV.rot} and $|g| = 1$ immediately imply 
    \begin{align}\label{eq:tildeV_l2.rot}
        |\tilde V|_2\le N^{-1/2}c_{\tilde \mR(X)}^{-1}.
     \end{align} 

\begin{prop}\label{prop:energy.rot}
    Let $T>0$ be given. Assume that 
    \begin{align}\label{eq:lambda_cond.rot}
        \lambda_N \ge \frac{2}{N^2 c_{\tilde \mR,T}}\sup_{0\le t\le T}\|\nabla \tilde \mR^{-1}(X(t))(\bar{g},0)\|.    
    \end{align}
    Then for all $t\in [0,T]$, the following estimate holds 
    \begin{align}\label{eq:E_control.rot}
	    |(V- \tilde V, L^{1/2} (\Omega- \tilde \Omega)|_2(t) &\leq  |(V- \tilde V, L^{1/2} (\Omega- \tilde \Omega))|_2(0)  e^{-\frac 12 \lambda_N N  c_{\mR,T} t} \\
     &\hspace{-1cm} +\frac{2 \sqrt 2 \bra{1 - e^{-\frac 12 \lambda_N N  c_{\tilde \mR,T} t}}}{\lambda_N N^{5/2}c_{\tilde \mR,T}^2}\sup_{0\le t\le T} \|\nabla  \tilde \mR^{-1}(X)(\bar g,0)\|. \qquad \,
    \end{align}
\end{prop}

\begin{proof}
We start by introducing the modulated energy
\begin{align} \label{eq:mod.energy.rot}
	E \coloneqq  \frac 1 {2N} \left( |V - \tilde V|_2^2 + L |\Omega - \tilde \Omega|^2_2 \right).
\end{align}
By using \eqref{eq:acceleration} and the representation \eqref{eq:def_tildeV.rot} we compute the time derivative of $E$ and estimate
\begin{align}
\begin{aligned}
	\frac {\dd E} {\dd t} &= \frac 1N((V,L \Omega) - (\tilde  V,L \tilde \Omega)) \cdot \left(\lambda_N N \left( \frac {(\bar g,0)} N - \m M^2 \mR(X) (V,\Omega) \right) \right.\\
 &\hspace{6cm}\left.- (V \cdot \nabla) \mR^{-1}(X) \frac {(\bar g,0)} N\right) \\
	&= -\lambda_N ((V, L^{1/2} \Omega)  - (\tilde  V,L^{1/2} \tilde \Omega)) \cdot \tilde \mR(X) \left((V, L^{1/2} \Omega)  - (\tilde  V,L^{1/2}  \tilde \Omega)\right) \\
 &\hspace{1cm} -\frac{1}{N^2}((V,L \Omega) - (\tilde V, L \tilde \Omega)) 
\cdot  \bra{\bra{(V - \tilde  V) +  \tilde V} \cdot \nabla} \mR^{-1}(X) (\bar g,0) \\
	& \leq - 2  \lambda_N c_{\tilde \mR(X)} E + \frac{1}{N^2} \|\nabla \tilde \mR^{-1}(X)(\bar{g},0)\||(V,L^{1/2} \Omega) - (\tilde V,L^{1/2} \tilde \Omega)|^2_2 \\
 & \hspace{1cm}+\frac{1}{N^2}\|\nabla \tilde \mR^{-1}(X)(\bar g,0)\||\tilde V|_2|(V,L^{1/2}\Omega) - (\tilde V,L^{1/2} \tilde \Omega)|_2. \\
	& \leq - 2 \lambda_N N c_{\tilde \mR(X)} E + \frac 2N  \|\nabla \tilde\mR^{-1}(X)(\bar{g},0)\|\left(E + N^{-1/2}|\tilde V|_2 E^{1/2} \right).
 \end{aligned} \label{eq:dtE.rot} 
\end{align}
We use \eqref{eq:tildeV_l2} and \eqref{eq:lambda_cond.rot} to rewrite \eqref{eq:dtE.rot} as 
\begin{align}
	\frac {\dd E} {\dd t} & \leq -  \lambda_N N c_{\tilde \mR(X)} E + \frac{2}{N^2c_{\tilde \mR(X)}}  \|\nabla \tilde\mR^{-1}(X)(\bar{g},0)\|E^{1/2}.
\end{align}
This implies
\begin{align}
	\sqrt E(t) \leq  \sqrt E(0) e^{-\frac 1 2 \lambda_N N  c_{\tilde \mR,T} t} + 2 \frac{1 - e^{-\frac 12 \lambda_N N  c_{\tilde \mR,T} t}}{\lambda_NN^3 c_{\tilde \mR,T}^2}\sup_{0\le t\le T} \|\nabla \tilde \mR^{-1}(X)(\bar g,0)\|^2.
\end{align}
The statement follows directly.
\end{proof}

    \begin{prop}\label{pro:resistance.estimates.rot}
Let $Y\in (\R^3)^N$ satisfy $\dmin \ge 4 R$. Then, the following estimates hold.
	\begin{align}
		c_{\tilde \mR(Y)} &\gtrsim \frac{R}{1+R S_1}, \label{eq:c_R.rot} \\
		C_{\tilde \mR(Y)} &\lesssim R. \label{eq:C_R.rot}
	\end{align}
\end{prop}
\begin{proof}
     \textbf{Proof of \eqref{eq:C_R.rot}.} Let $(W,Z) \in (\R^3)^N \times (\R^3)^N$ and let $v$ be the solution to \eqref{eq:fru.rot} with $(W,L^{1/2} Z)$. Then, on the one hand, $v$ is the minimizer of the Dirichlet energy among all solenoidal fields   $w$ satisfying $w(x) = W_i + L^{-1/2} Z_i \times (x-Y_i)$ for all $x \in B_i$ for all $i=1,\dots,N$. On the other hand, the form $(W,Z)\cdot \tilde \mR(Y)(W,Z) = (W,L^{-1/2} Z)\cdot \mR(Y)(W, L^{-1/2} Z)$ is equal to the Dirichlet energy (equation \eqref{eq:R_bilin.rot}). Combining both, we have 
	\begin{align}
		(W,Z)\cdot \tilde \mR(Y)(W,Z) =  \| \nabla v \|^2_2 \leq \|w\|_{\dot H^1}^2, 
	\end{align}
	for all $w \in \dot H^1(\R^3)$ with $\dv w = 0$ and $w = W_i + Z_i \times (x- Y_i)$ in $B_i$. To prove \eqref{eq:C_R} it is thus enough to construct a field $w$ with $w(x) = W_i + L^{-1/2} Z_i \times (x-Y_i)$ for all $x \in B_i$ and such that
		\begin{align}
		\|w\|^2_{\dot H^1(\R^3)} \leq C R |(W,Z)|_{2}^2.
	\end{align}
	This is possible under the assumption $|Y_i - Y_j| \geq 4 R$. For the construction we refer to \cite[Lemma 4.2]{HoeferJansen20}.

   \noindent\textbf{Proof of \eqref{eq:c_R.rot}.} Let now $v$ be the solution to the homogeneous Stokes equation with given forces and torques $(G,L^{1/2} H) \in (\R^{3})^N \times (\R^{3})^N$, i.e.
	\begin{align}
    \left\{\begin{array}{rl}
	    - \Delta  v + \nabla p = 0, ~ \dv   v &=0 \quad \text{in} ~ \R^3 \setminus \bigcup_i  B_i, \\
		  e  v &= 0 \quad \text{in} ~  \cup_i B_i, \\
		 -\int_{\partial  B_i} \sigma[ v] n \dd  \mathcal{H}^2  &= G_i,  \quad \text{for all} ~ 1 \leq i \leq N, \\
   -\int_{\partial  B_i} (x-Y_i) \times \sigma[ v] n \dd  \mathcal{H}^2  &= L^{1/2} H_i ,  \quad \text{for all} ~ 1 \leq i \leq N
   \end{array}\right.
	\end{align}
	Then, denoting $W_i[G,L^{1/2} H] = v(Y_i)$, we have  $(G,L^{1/2} H) = \mR W[G,L^{1/2} H]$, and thus
	\begin{align} \label{char.c_R.rot}
	    c^{-1}_{\mR(Y)} &= \sup_{|A|_2 = 1}  A\cdot \tilde \mR^{-1} A \\
     &= \sup_{|(G,H)|_2 = 1} (G,L^{1/2} H) \cdot W[G,L^{1/2} H] = \sup_{|(G,H)| = 1} \| v\|^2_{\dot H^1}.
	\end{align}
	On the other hand, we have
	\begin{align} \label{projection.rot}
		\| v\|_{\dot H^1} \leq \|w\|_{\dot H^1},
	\end{align}
	where $w = \sum_i w_i \in \dot H^1$ is given  through the unique solutions to the Stokes problems
	\begin{align}
    \left\{\begin{array}{rl}
	    - \Delta  w_i + \nabla p_i = 0, ~ \dv   w_i &=0 \quad \text{in} ~ \R^3 \setminus   B_i, \\
		  e  w_i &= 0 \quad \text{in} ~ B_i, \\
		 -\int_{\partial  B_i} \sigma[ v] n \dd  \mathcal{H}^2  &= G_i,   \\
   -\int_{\partial  B_i} (x-Y_i) \times \sigma[ v] n \dd  \mathcal{H}^2  &= L^{1/2} H_i .
   \end{array}\right.
	\end{align}
  Indeed, \eqref{projection.rot} follows from the observation that
	\begin{align}
	    \int_{\R^3} \nabla v \cdot \nabla ( v -  w) \dd x= -\sum_i \int_{\partial  B_i}   v\cdot (\sigma[v -  w] n) \dd  \mathcal{H}^2 = 0 ,
	\end{align}
	because $ e v = 0$ in $B_i$ and $\int_{\partial  B_i}   \sigma[v -  w] n \dd \mathcal{H}^2= 0 = \int_{\partial  B_i}  (x- Y_i) \times\sigma[v -  w] n \dd \mathcal{H}^2$.	
By a standard decay estimate, for all $x \in \R^3 \setminus B_i$
\begin{align}
    |w_i(x)| + |x-Y_i|(|\nabla w_i(x)| + |p_i(x)|) \lesssim \frac{|G_i|}{|x-Y_i|} + \frac{R |H_i|}{|x-Y_i|^2} \lesssim \frac{|G_i| + |H_i|}{|x-Y_i|}.
\end{align}
 Thus, an explicit calculation reveals that, recalling the definition of $S_1$ from \eqref{def.S} and using Young's inequality,
	\begin{align*}
		\| v\|^2_{\dot H^1} &\leq \|w\|^2_{\dot H^1}= 2 \sum_i  \int_{\R^3 \setminus B_i} |e w_i|^2 \dd x  - \sum_i \sum_{j 
  \neq i} \int_{\partial B_i}w_j\cdot(\sigma[w_i]n) \dd \mathcal{H}^2 \\
 & \lesssim  \sum_i \frac{|G_i|^2 + |H_i|^2 }{R} + \sum_i \sum_{j \neq i} \frac{(|G_i| + |H_i|)(|G_j| + |H_j|)}{|Y_i - Y_j|} \\
 &\lesssim |(G,H)|^2 (R^{-1} + S_1).
	\end{align*}
	Combining this with \eqref{char.c_R.rot} concludes the proof.
\end{proof}

\begin{prop}\label{pro:gradient.resistance.rot}
 There exists $\delta > 0$ with the following property. Assume that $Y\in (\R^3)^N$ satisfies $\dmin \ge 4 R$ and
\begin{align}  \label{delta.rot}
    R^3 S_3 \leq \delta.
\end{align}
Then
\begin{align}\label{eq:blub}
    \|\nabla \tilde \mR^{-1}(Y) (\bar g,0) \| \lesssim S_2\bra{1+R^2S_2} .
\end{align}
\end{prop}

We remark that the proof of Proposition~\ref{pro:gradient.resistance.rot} is largely similar to the proof of Proposition~\ref{pro:gradient.resistance}. In the interest of keeping the appendix short, we only comment on parts, where the proof differs.\\

For a particle configuration $Y \in (\R^3)^N$ as in the statement of Proposition \ref{pro:gradient.resistance.rot}, we introduce the function $ w[Y] \in \dot H^1(\R^3)$ as the solution to 
\begin{align}
    	\label{eq:inertialess.intermediate.rot}
\left\{\begin{array}{rl}
		- \Delta   w[Y] + \nabla   p[Y] = 0, ~ \dv    w[Y] &=0 \quad \text{in} ~ \R^3 \setminus \bigcup_i B_i, \\
		  e w[Y] &= 0 \quad \text{in} ~ \bigcup_i  B_i, \\
		- \int_{\partial B_i^l} \sigma[w[Y]] n \dd  \mathcal{H}^2  &= g \quad \text{for all} ~ 1 \leq i \leq N, \\
  		 -\int_{\partial B_i^l} (x - X_i) \times \sigma[w[Y]] n \dd  \mathcal{H}^2  &= 0 \quad \text{for all} ~ 1 \leq i \leq N.
\end{array}\right.
\end{align}
It suffices to show the estimate in Proposition \ref{pro:gradient.resistance.rot} when $\tilde \mR$ is replaced by $\bar \mR$ defined as
\begin{align}
    \bar \mR &: = \bar {\m M} \mR \bar {\m M}, \\
    \bar  {\m M} &:= \begin{pmatrix}
        \Id & 0 \\
        0 & \frac 1 {2 R} \Id
    \end{pmatrix}
\end{align}
since $L$ is equal $R^2$ up to a fixed multiplicative constant.
Then, by definition of $\bar \mR(Y)$, we have
\begin{align}
    \bar \mR(Y)^{-1} (\bar g ,0) = (w[Y](Y_1), \dots, w[Y](Y_N), R \curl w[Y](Y_1), \dots ,  R \curl w[Y](Y_N)).
\end{align}
Rearranging the entries, we conveniently write
\begin{align}
    (\bar \mR^{-1}(Y) (\bar g ,0))_i &=   (w[Y](Y_i),R \curl w[Y](Y_i))
\end{align}
such that
\begin{align}
    &\nabla_{Y_j}(\bar\mR^{-1}(Y) (\bar g,0))_i \\ &= \nabla_{Y_j}   (w[Y](Y_i), R \curl_x w[Y](Y_i)) + \delta_{ij} \nabla_{x}    (w[Y](Y_i), R \curl w[Y](Y_i)) \\
 &= \nabla_{Y_j}   (w[Y](Y_i), R\curl_x w[Y](Y_i))+  \delta_{ij}     (\nabla_{x} w[Y](Y_i),0) \label{eq:nabla_X.R.rot},
\end{align}
provided these gradients exist and where the second term after the first identity in the second line vanishes because the constraint  $e w[Y] = 0$ in $B_i$ from \eqref{eq:inertialess.intermediate.rot} implies that $\nabla^2 w[Y] = 0$ in $B_i$.

As in Section~\ref{sec:EstR}, we define $  w_0[Y] \in \dot H^1(\R^3)$ as the solution to 
\begin{align}
    -\Delta   w_0[Y] + \nabla   p_0 [Y]  = \sum_i \delta_{\partial B_i} g,
\end{align}
and note that we already have information on $w_0$ by Lemma~\ref{lem:nabla.w_0}. Again, the strategy is to estimate $w_0[Y]$ and $w[Y]-w_0[Y]$.
% and by linearity we can write $ w_0 = \sum_i  w^{(i)}_0$, where $-\Delta w_0^{(i)}+\nabla p=\delta_{\partial B_i}g$ and it is well-known that the functions have the explicit form
%     \begin{align} \label{eq:w_0.explicit.rot}
%           w^{(i)}_0(x) = \begin{cases} 
%           \frac g {6 \pi R} &\qquad \text{in } B_i, \\  \left(1 - \frac {R^2}{6} \Delta \right)\Phi(x - Y_i) g & \qquad \text{otherwise}.
%         \end{cases}
% \end{align}
% \begin{lem} \label{lem:nabla.w_0.rot}
%     Under the assumptions of Proposition \ref{pro:gradient.resistance.rot} we have
%     \begin{align}
%         \left\|\left(\nabla_{Y_j}   (w_0[Y],R \nabla_x w_0[Y])(Y_i) \right)_{ij}\right\| \le  S_2, \label{nabla.w_0.1.rot} 
%     \end{align}
% \end{lem}
% \begin{proof}
%     By the explicit form \eqref{eq:w_0.explicit}, it holds that
%     \begin{align}
%        |\nabla_{Y_j} (w_0[Y],R \nabla_x w_0[Y])(Y_i) | \le \frac{1}{d_{ij}^2}.
%        %\deleted{|\nabla_{x}   w_0[Y](Y_i)| \le \sum_{k \neq i} \frac{1}{d_{ik}^2}}. 
%     \end{align}
%     We conclude by  \eqref{eq:infty_1}.
% \end{proof}

To estimate the difference $  w[Y] -   w_0[Y]$, we define for $k \in \N$ the functions $  w_k[Y]$ obtained through the method of reflections from $  w_0[Y]$.
More precisely, following the notation from \cite{Hofer18MeanField}, we introduce $Q_i$ this time as the solution operator that maps a  function $\varphi \in H^1_{\sigma}(B_i)$ to the solution $\psi \in \dot H^1_{\sigma}(\R^3)$ of 
\begin{equation} \label{eq:Q.rot}
\left\{
\begin{array}{rl}
	- \Delta \psi + \nabla p = 0, ~\dv \psi &=  0 \quad \text{in} ~ \R^3 \setminus B_i, \\
	e \psi &= e \varphi \quad \text{in} ~ B_i, \\
\int_{\partial B_i} \sigma[\psi] n\dd \mathcal{H}^2&=\int_{\partial B_i} (x- X_i) \times \sigma[\psi] n\dd \mathcal{H}^2 = 0.
	\end{array}\right.
\end{equation} 
% Here the index $\sigma$ in $H^1_\sigma(B_i)$ and $\dot H_\sigma^1(\R^3)$ denotes the subspace of divergence free functions in $H^1(B_i)$ and $\dot H^1(\R^3)$, respectively.
% Note that the operator $Q_i$ depends only on the position of $Y_i$.
Then, we know from \cite[Proposition 3.12]{Hofer18MeanField} that 
\begin{align} \label{eq:MOR.rot}
	  &w_k[Y] \coloneqq (1 - \sum_i Q_i)^k   w_0[Y] \to   w[Y] \quad \text{in} ~ \dot H^1_\sigma(\R^3) \cap L^\infty(\R^3), \\
   &\|e w_{k} \|_{L^\infty(\cup_i B_i)} \leq {S_2} (C  R^3 S_3)^k. \label{eq:MOR.rot.2}
\end{align}
Note that \cite[Proposition 3.12]{Hofer18MeanField} is stated only for the case without rotations but can be generalized to include this case (cf. \cite{NiethammerSchubert19, Hoefer19}). Also note that $\alpha_2$ in \cite[Proposition 3.12]{Hofer18MeanField} is defined as $\alpha_2 = \frac 1 N S_2$, but instead of forces $\bar g$, in \cite{Hofer18MeanField}, the forces are $\frac{\bar g}{N}$.

We recall the following decay estimates:
\begin{lem}[{\cite[Lemma 3.1, Lemma 4.4]{Hoefer19}}]\label{lem:decay.Q.rot}
Let $\varphi \in H^1_{\sigma}({B}_i) $ such that $ \nabla \varphi \in L^\infty({B}_i)$. Then,  for all $x \in \R^3 \setminus B(Y_i,2R)$, it holds that
\begin{align} \label{est:Q_i.pointwise.rot}
	|\nabla^l (Q_i \varphi)(x)| \lesssim \frac { R^3}{|x - Y_i|^{l+2}} \|e \varphi\|_{L^\infty(B_i)} .
\end{align}
Moreover, 
	\begin{align} \label{eq:Q_i.average.rot}
		Q_i \varphi = \varphi - \fint_{\partial B_i} \varphi \dd \mathcal{H}^2 - \frac 1 2 \fint_{B_i}  \curl \varphi \dd \mathcal{H}^2  \times  (x- X_i) \quad \text{in} ~ B_i.
	\end{align}
\end{lem}

With these estimates, the first right-hand side term in \eqref{eq:nabla_X.R.rot} is dealt with exactly as before in the proof of Proposition \ref{pro:gradient.resistance}.  For the additional term $R\curl w[Y](Y_i)$ in this term, one observes that the additional gradient is compensated by the factor $R$ since  $R \nabla Q_i \varphi$ has the same decay as $Q_i \varphi$.  The same reasoning applies to estimating $R \nabla_{Y_j}\curl_x\nabla w_{0}[Y](Y_i)$, which can be estimated as in Lemma~\ref{lem:nabla.w_0} and by trading the $R$ for the $\curl$.

The last right-hand side term in \eqref{eq:nabla_X.R.rot} is estimated through the following lemma.
\begin{lem}
    Under the assumptions of Proposition \ref{pro:gradient.resistance.rot} we have for all $1 \leq i \leq N$
    \begin{align}
        \left\|\left(\delta_{ij}     (\nabla_{x} w[Y](Y_i),0) \right)_{ij}\right\| \le S_2\label{nabla.w_0.2.rot} 
    \end{align}
\end{lem}
\begin{proof}
The explicit form \eqref{eq:w_0.explicit} yields
    \begin{align}
        \sup_i \sup_{x \in B_i} |\nabla_{x} w_0[Y](x) | \le S_2, \\
        \sup_i \sup_{x \in B_i} |\nabla^2_{x} w_0[Y](x) | \le S_3 
    \end{align}
    Moreover, by Lemma \ref{lem:decay.Q.rot} and \eqref{eq:MOR.rot.2}, for all $1 \leq i \leq N$
    we have
    \begin{align}
        \sup_{i} \|\nabla^2_x w_{k+1}[Y]\|_{L^\infty(B_i)} &\leq \sum_{j \neq i}\|\nabla^2_x Q_j w_{k}[Y]\|_{L^\infty(B_i)} \\
        &\leq  \sum_{j \neq i} \frac{R^3}{d_{ij}^4} \|e w_{k} \|_{L^\infty(\cup_i B_i)} \\
        &\lesssim  {S_2} R^3 S_4 (C  R^3 S_3)^k.
    \end{align}
    Using this and combining with Lemma \ref{lem:decay.Q.rot} and \eqref{eq:MOR.rot.2} yields
    \begin{align}
     & |\nabla_{x} (w_{k+1}[Y] - w_k[Y])(Y_i)| \\ &\leq  \sum_{j} |\nabla_{x}  Q_j w_k[Y])(Y_i)| \\
      &\lesssim |(e w_k[Y])(Y_i)| + \fint_{B_i} \left|\curl_x w_k[Y](Y_i) - \curl_x w_k[Y](x) \right| \dd  x \\
      &\quad+ \sum_{j \neq i} |\nabla_{x}  Q_j w_k[Y])(Y_i)| \\
      &\lesssim {S_2} (C  R^3 S_3)^k +  {S_2} R^4 S_4 (C  R^3 S_3)^{k-1} + \sum_{j \neq i} \frac{R^3}{d_{ij}^3} \|e w_{k} \|_{L^\infty(\cup_i B_i)} \\
      &\lesssim {S_2} (C  R^3 S_3)^k(1+R^3 S_3),
      \end{align}
      where we used $R S_4 \leq S_3$. 
      Hence, under the assumptions of  Proposition \ref{pro:gradient.resistance.rot}, $\nabla_{x} w_k[Y](Y_i)$ is a Cauchy sequence and
      \begin{align}
          |\nabla_{x} w[Y](Y_i)| &\leq \limsup_{k\to \infty} |\nabla_{x} w_{k+1}[Y](Y_i)| \\
          &\leq \limsup_{k\to \infty} |\nabla_{x} w_{0}[Y](Y_i)| +  \sum_{l=0}^k  |\nabla_{x} (w_{l+1}[Y] - w_l[Y])(Y_i)| \\
          & \lesssim S_2.
      \end{align}
      We conclude by the fact that the operator norm of a  diagonal matrix is estimated by the maximum of its entries.
\end{proof}
}